%% file: MarkIV.tex
\documentclass[11pt,reqno]{amsart}
\input{mystyle}

\begin{document}
	\vspace{1cm}
	\title[Semilinear Bernoulli-type problem]{The free boundary for a semilinear non-homogeneous Bernoulli problem}
	\author{Lili Du}
	\address{College of Mathematics and Statistics, Shenzhen University,Shenzhen 518061, P. R. China.}
	\email{dulili@szu.edu.cn}
	\thanks{The first author is supported by National Nature Science Foundation of China Grant 11971331, 12125102, and Sichuan Youth Science and Technology Foundation 2021JDTD0024.}
	\author{Chunlei Yang}
	\address[]{Department of Mathematics, Sichuan University, Chengdu 610064, P. R. China.}
	\email{yang\_chunlei@stu.scu.edu.cn}
\begin{abstract}
	In the classical homogeneous one-phase Bernoulli-type problem, the free boundary consists of a "regular" part and a "singular" part, as Alt and Caffarelli have shown in their pioneer work (\emph{J. Reine Angew. Math.}, 325, 105-144, 1981) that regular points are $C^{1,\gamma}$ in two-dimensions. Later, Weiss (\emph{J. Geom. Anal.}, 9, 317-326, 1999) first realized that in higher dimensions a critical dimension $d^{*}$ exists so that the singularities of the free boundary can only occur when $d\geqslant d^{*}$. 
	
	In this paper, we consider a non-homogeneous semilinear one-phase Bernoulli-type problem, and we show that the free boundary is a disjoint union of a regular and a singular set. Moreover, the regular set is locally the graph of a $C^{1,\gamma}$ function for some $\gamma\in(0,1)$. In addition, there exists a critical dimension $d^{*}$ so that the singular set is empty if $d<d^{*}$, discrete if $d=d^{*}$ and of locally finite $\mathcal{H}^{d-d^{*}}$ Hausdorff measure if $d>d^{*}$. As a byproduct, we relate the existence of viscosity solutions of a non-homogeneous problem to the Weiss-boundary adjusted energy, which provides an alternative proof to existence of viscosity solutions for non-homogeneous problems.
\end{abstract}
\maketitle
\section{Introduction and main results}
Free boundary problems are a particular type of boundary value problem where the domain is a part of the solution that is up to be determined. Typically, the classical \emph{one-phase Bernoulli-type problem} deals with the minimization problem
\begin{align}\label{Formula: AC1981 problem}
	\min\left\lbrace\int_{\Omega}|\nabla u|^{2}+Q^{2}(x)\chi_{\{u>0\}}\:dx\colon u\in H^{1}(\Omega),u=u^{0}\text{ on }S\right\rbrace.
\end{align}
Here $\Omega\subset\mathbb{R}^{d}$ ($d\geqslant2$) is a bounded open set and locally $\partial\Omega$ is a Lipschitz graph. $S\subset\partial\Omega$ is measurable and $\mathcal{H}^{d-1}(S)>0$. The Dirichlet data on $S$ is given by $u^{0}\in H^{1}(\Omega)$, $u^{0}\geqslant0$. The given force function $Q(x)\colon\Omega\to\mathbb{R}$ is assumed to be non-negative and measurable. $\chi_{\{u>0\}}$ denotes the indicator function of the set $\{u>0\}$. Minimizers of problem \eqref{Formula: AC1981 problem} are solutions to the following problem:
\begin{align}\label{Formula: Homogeneous one-phase Bernoulli}
	\Delta u=0\quad\text{ in }\quad\varOmega^{+}(u),\qquad u=0,\ \ |\nabla u|=Q(x)\quad\text{ on }\quad\partial\varOmega^{+}(u),
\end{align}
where $\varOmega^{+}(u):=\Omega\cap\{u>0\}$ and $\partial\varOmega^{+}(u):=\Omega\cap\partial\{u>0\}$. 

In this paper, we investigate a natural generalization of the above one-phase Bernoulli-type problem in the presence of a right-hand side which depends on the solution itself. Let $\mathbb{R}^{d}$, $\Omega$, $\partial\Omega$, $S$, $u^{0}$ and $Q(x)$ be the same mathematical objects as in \eqref{Formula: AC1981 problem} and \eqref{Formula: Homogeneous one-phase Bernoulli}. Let $F(t)\in C^{2,\beta}(\mathbb{R})$ be a function for some $\beta\in(0,1)$, and let us define
\begin{align}\label{Formula: f(t)}
	f(t):=-\tfrac{1}{2}F'(t).
\end{align} 
Consider the functional
\begin{align}\label{Formula: Functional J(u)}
	J(u)=\int_{\Omega}|\nabla u|^{2}+F(u)+Q^{2}(x)\chi_{\{u>0\}}\:dx
\end{align}
over the class of admissible functions
\begin{align}\label{Formula: Admissible functions K}
	\mathcal{K}:=\{u\in H^{1}(\Omega)\colon u\leqslant\Psi\text{ a.e. in }\Omega,\ \ u=u^{0}\text{ on }S\}.
\end{align}
Here $\Psi\in C^{0}(\bar{\Omega})\cap C^{2,\alpha}(\Omega)$ is a given function, satisfying
\begin{align}\label{Formula: Psi}
	\Lace\Psi+f(\Psi)\leqslant0\quad\text{ in }\Omega,\qquad\Psi>0\quad\text{ in }\Omega,\qquad 0\leqslant u^{0}\leqslant\Psi\quad\text{ on }S.
\end{align}
Observe that when the free boundary $\partial\varOmega^{+}(u)$ is smooth, the Euler-Lagrange equations of \eqref{Formula: Functional J(u)} are given by
\begin{align}\label{Formula: Semilinear problem}
	-\Delta u=f(u)\quad&\text{ in }\quad\varOmega^{+}(u),\qquad u=0,\ \ |\nabla u|=Q(x)\quad\text{ on }\quad\partial\varOmega^{+}(u).
\end{align}
Our main purposes of this article is to investigate the impact of non-zero right-hand side $f(u)$ on the structure and the regularity of the free boundary. This is because that the system \eqref{Formula: Semilinear problem} has numerous applications in physics and the right-hand side $f(u)$ has many physical explanations. For instance when $d=2$, the unknown variable $u$ can be regarded as the scalar stream function and the function $f(u)$ stands for vorticity strength and the system \eqref{Formula: Semilinear problem} is used for modeling the incompressible inviscid jet and cavity flow with general vorticity (\cite{CDW2017,CDZ2017}), and for the periodic water waves with non-zero vorticity (\cite{CS2004,V2008,VW2012,WZ2012}).

Let us shortly summarize the mathematical results relevant in the context of homogeneous one-phase problem \eqref{Formula: Homogeneous one-phase Bernoulli}. In their seminal paper \cite{AC1981}, Alt and Caffarelli derived the smoothness of the free boundary for weak solutions. After that, Caffarelli introduced the notion of viscosity solutions to the problem \eqref{Formula: Homogeneous one-phase Bernoulli} in a series of artworks \cite{C1987,C1989,C1988}, and proved the smoothness of the free boundary for a larger class of solutions. These approaches appears to be universal and were used to study quasilinear (even degenerate) Bernoulli-type problems \cite{ACF1984,DP2005}. This paper aims to extend the main theorem in the paper by Alt and Caffarelli to more general semilinear situations as well as the higher-dimensional case. However, our paper replace the core of that paper, which relies on a Weiss-type monotonicity formula and a viscosity approach developed by De Silva, which promises to be applicable to one-phase problems governed by more complex operators \cite{DY2023}.

It should be also worth noting that the appearance of singularities and the structure of the free boundary in higher dimensions ($d\geqslant3$) play a significant role in the study of minimal surfaces. The first far-reaching observation from this perspective was dedicated to Weiss \cite{W1999}, who revealed that for minimizers of the problem \eqref{Formula: AC1981 problem}, there is a critical dimension $d^{*}\geqslant3$ such that the singular set of the free boundary is empty if $d<d^{*}$, and is a discrete set of isolated points if $d=d^{*}$, and is a closed set of Hausdorff dimension $d-d^{*}$ if $d>d^{*}$. To the best of our knowledge, $d^{*}\in\{5,6,7\}$ and we refer readers to \cite{CJK2004,SJ2009,JS2015} for detail. In this article, we also study the properties of the singular set. We will prove a similar result as Weiss in \cite{W1999} for our semilinear case. The main obstacle here is that the blow-up limit of our problem appears not to minimize the original functional so that we can no longer follow directly the classical approach of Weiss. As a final twist, we add, is the precise value of $d^{*}$, which is still an open problem, even to the homogeneous problem \eqref{Formula: Homogeneous one-phase Bernoulli}.
\subsection{Assumptions, notations and definitions}
Let $F(t)\in C^{2,\beta}(\mathbb{R})$ be a convex function for some $\beta\in(0,1)$, satisfying $F''(t)\leqslant F_{0}$ for all $t\in\R$, and
\begin{align}\label{Formula: Assumption on F(t)}
	F(0)=0,\qquad -F_{0}\leqslant F'(t)\leqslant0\quad\text{ for }\quad t\leqslant0,
\end{align}
where $F_{0}\geqslant0$ is a given constant. Assume that $Q(x)\in C^{0,\beta}(\bar{\Omega})$ is a function which satisfies
\begin{align}\label{Formula: Assumption on Q(x)}
	0<Q_{\mathrm{min}}\leqslant Q(x)\leqslant Q_{\mathrm{max}}<+\infty\quad\text{ for each }\quad x\in\bar{\Omega}.
\end{align}
Here $Q_{\mathrm{min}}<Q_{\mathrm{max}}$ are two given positive constants. Note that in view of the assumptions we imposed on $F$, it is obvious that $f(t)\in C^{1,\beta}(\mathbb{R})$ defined in \eqref{Formula: f(t)} is function which satisfies
\begin{align}\label{Formula: Assumption on f(t)}
	0\leqslant f(t)\leqslant\frac{F_{0}}{2}\quad\text{ for }\quad t\leqslant0,\qquad-\frac{F_{0}}{2}\leqslant f'(t)\leqslant0\quad\text{ for }\quad t\in\mathbb{R}.
\end{align}
It should be worth noting that if $x^{0}\in\partial\varOmega^{+}(u)$ is a free boundary point, then by \eqref{Formula: Assumption on f(t)}, $f(u(x^{0}))=f(0)\geqslant0$ on $\partial\varOmega^{+}(u)$. This implies that the non-linear force term is non-negative on the free boundary. This is the main difference between the one considered in \cite{LT2015}, whose right-hand side is vanishing on the free boundary.

We further make some comments to our monotone assumptions \eqref{Formula: Assumption on F(t)} on $F(t)$. Firstly, $F(0)=0$ is quite natural since $F(t)$ is actually the primitive of $f(t)$ and $F(t):=-2\int_{0}^{t}f(s)\:ds$. $F'(t)\leqslant0$ for $t\leqslant0$ was applied to prove that minimizers are non-negative in $\Omega$. The convexity of $F''(t)\geqslant0$ is needed for applying the comparison principle. These assumptions are imposed because that all the results demonstrated in this paper are \emph{not the conditional regularity} of the free boundary and there are no apriori assumptions on the regularity of the solutions and topological properties of the free boundary. This is the main difference between our work and other non-homogeneous problems \cite{S2011}.

Throughout this paper $\mathbb{R}^{d}$ will be the usual Euclidean space of dimension $d$. We denote by $c_{d}$ or $C_{d}$ the constants depending only on the dimensions. For any $t\in\mathbb{R}$, we denote the \emph{positive} and \emph{negative} parts of $t$ by $t^{\pm}:=\max\{\pm t,0\}$, so that $t=t^{+}-t^{-}$ and $|t|=t^{+}+t^{-}$. The open $d$-dimensional ball with its center at $x^{0}=(x_{1}^{0},\dots,x_{d}^{0})\in\mathbb{R}^{d}$ of radius $r>0$ and volume $\omega_{d}r^{d}$ will be denoted by $B_{r}(x^{0}):=\{x\in\mathbb{R}^{d}\colon|x-x^{0}|<r\}$. If in particular $x^{0}=0$, we denote by $B_{r}:=B_{r}(0)$ for simplicity. The unit sphere of $d$-dimensional will be denoted by $\mathbb{S}^{d-1}$. If $u\colon B_{R}\to\mathbb{R}$ is a function expressed in polar coordinates as $u=u(\rho,\theta)$, then
\begin{align}\label{Formula: The spherical Laplacian}
    |\nabla u|^{2}=(\partial_{\rho}u)^{2}+\frac{1}{\rho^{2}}|\nabla_{\theta}u|^{2},\qquad\Delta u=\partial_{\rho\rho}u+\frac{d-1}{\rho}\partial_{\rho}u+\frac{1}{\rho^{2}}\Delta_{\mathbb{S}}u,
\end{align}
where $\nabla_{\theta}u$ is the sphere gradient and $\Delta_{\mathbb{S}}u$ is the spherical Laplacian. 

Given a set $E\subset\mathbb{R}^{d}$, we denote by $\chi_{E}$ the characteristic function of the set $E$. We denote by $D\subset\subset E$ the compact subsets $D$ of $E$. In the text we use the $d$-dimensional \emph{Lebesgue-measure} $|\cdot|$ and the $m$-dimensional \emph{Hausdorff measure} $\mathcal{H}^{m}$. Sometimes, $|\cdot|$ will also represent the absolute value, depending on the context. The \emph{Lebesgue density} for $E$ is defined to be the limit
\begin{align*}
	\vartheta:=\lim_{r\to0}\frac{|E\cap B_{r}(x^{0})|}{|B_{r}(x^{0})|},\quad\text{ where }\quad x^{0}=(x_{1}^{0},\ldots,x_{d}^{0})\in\mathbb{R}^{d}.
\end{align*}
It is easy to see that $\vartheta\in[0,1]$. For any $E\subset\mathbb{R}^{d}$ and $\vartheta\in[0,1]$, the set $E^{(\gamma)}$ will be the set of points at which $E$ has Lebesgue density equal to $\vartheta$, namely,
\begin{align}\label{Formula: Definition of E^{gamma}}
	E^{(\vartheta)}=\left\lbrace x^{0}=(x_{1}^{0},\ldots,x_{d}^{0})\in\mathbb{R}^{d}\colon\lim_{r\to0^{+}}\frac{|E\cap B_{r}(x^{0})|}{|B_{r}(x^{0})|}=\vartheta\right\rbrace.
\end{align}
Let $\partial_{\mathrm{red}}(E):=\{ x\in\mathbb{R}^{d}\colon\text{ there exists }\nu\in\partial B_{1}\text{ such that } r^{-d}\int_{B_{r}(x)}|\chi_{E}-\chi_{\{y\colon(y-x)\cdot\nu(x)<0\}}|\to0\text{ as }r\to0\}$ be the \emph{reduced boundary} of $E$.

We will use the notation $J(u,D)$ for any $D\subset\Omega$ to highlight the domain of the functional which was restricted to. Sometimes we use the notation $B_{1}^{+}(u):=B_{1}\cap\{u>0\}$ and $\partial B_{1}^{+}(u):=B_{1}\cap\partial\{u>0\}$. For each $r>0$ and $x^{0}\in\mathbb{R}^{d}$, we denote by $u_{x^{0},r}(x)$ and $u_{r}(x)$ the functions $u_{x^{0},r}(x):=\frac{1}{r}u(x^{0}+rx)$ and $u_{r}:=\frac{1}{r}u(rx)$, respectively. Let $\{r_{n}\}_{n=1}^{\infty}$ be a vanishing sequence of positive numbers, and we denote by $u_{n}(x):=u_{x^{0},r_{n}}(x)$ the \emph{blow-up sequence} for every $x\in B_{R}(x^{0})$ and $R>0$. We say that a function $u_{0}\colon\mathbb{R}^{d}\to\mathbb{R}$ is a \emph{blow-up limit} of $u$ at $x^{0}\in\Omega$ if for every $r_{n}\to0$ as $n\to\infty$ and for every $R>0$, $u_{n}$ converges uniformly to $u_{0}$ in $B_{R}$, that is, $\lim_{n\to\infty}\|u_{n}-u_{0}\|_{L^{\infty}(B_{R})}=0$. We say that $x^{0}\in\partial\varOmega^{+}(u)$ is a \emph{regular point} if there exists a blow-up limit $u_{0}$ of $u$ at $x^{0}$ which has the form
\begin{align*}
	u_{0}(x)=Q(x^{0})(x\cdot\nu)^{+},\quad\text{ for every }\quad x\in\mathbb{R}^{d},\quad\text{ and some }\quad\nu\in\mathbb{R}^{d}.
\end{align*}
The set of regular points will be denoted by $\partial_{\mathrm{reg}}\varOmega^{+}(u)$ and we define the \emph{singular part} of the free boundary as the complementary of $\partial_{\mathrm{reg}}\varOmega^{+}(u)$
\begin{align*}
	\partial_{\mathrm{sing}}\varOmega^{+}(u):=\partial\varOmega^{+}(u)\setminus\partial_{\mathrm{reg}}\varOmega^{+}(u).
\end{align*}
In the majority of this article, we work with the so-called \emph{minimizers}.
\begin{definition}[Minimizers]
	We say that $u\in\mathcal{K}$ is a \emph{minimizer} of $J$ in $\Omega$,  provided that
	\begin{align}\label{Definition: Minimizers}
		J(u)\leqslant J(v)\qquad\text{ for all }v\in\mathcal{K}.
	\end{align}
	In some contexts, $u$ is called an \emph{absolute minimizer} of $J$ in $\Omega$ if \eqref{Definition: Minimizers} is satisfied.
\end{definition}
\begin{remark}
	Most of our results in this paper can be generalized also for \emph{local minimizers}, for which \eqref{Definition: Minimizers} is satisfied for $v\in u+H_{0}^{1}(D)$ and for any $D\cc\Omega$. 
\end{remark}
The notion of viscosity solutions was first systematically introduced by Lions and his cooperators in \cite{CIL1992}, who established an integrated theory on the existence and uniqueness theory for fully-nonlinear partial differential equations. The viscosity approach to the associated free boundary problem was later developed by Caffarelli. We now define the notion of viscosity solutions to our problem and we first introduce the concept of "touch".
\begin{definition}[Touch]
	Given $u$, $\phi\in C^{0}(\Omega)$, we say that $\phi$ \emph{touches} $u$ from below (above) at $x^{0}\in\Omega$ if $u(x^{0})=\phi(x^{0})$ and $u(x)\geqslant\phi(x)$ ($u(x)\leqslant\phi(x)$) in a small neighborhood of $x^{0}$.
\end{definition}
With the aid of "touch", we formally define viscosity solutions of problem \eqref{Formula: Semilinear problem}.
\begin{definition}[Viscosity solutions]\label{Definition of viscosity solutions}
	Let $u\in C^{0}(\Omega)$ be a non-negative function in $\Omega$. Then $u$ is a solution to the problem \eqref{Formula: Semilinear problem} in the sense of \emph{viscosity}, provided that for every $x^{0}\in\overline{\varOmega^{+}(u)}$, the following does hold.
	\begin{enumerate}
		\item If $x^{0}\in\varOmega^{+}(u)$ and
		\begin{itemize}
			\item if $\phi\in C^{2}(\varOmega^{+}(u))$ touches $u$ from below at $x^{0}$, then $\Delta\phi(x^{0})+f(\phi(x^{0}))\leqslant0$,
			\item if $\phi\in C^{2}(\varOmega^{+}(u))$ touches $u$ from above at $x^{0}$, then $\Delta\phi(x^{0})+f(\phi(x^{0}))\geqslant0$.
		\end{itemize}
		\item If $x^{0}\in\partial\varOmega^{+}(u)$ and
		\begin{itemize}
			\item if $\phi\in C^{2}(\Omega)$ touches $u$ from below at $x^{0}$, then $|\nabla\phi(x^{0})|\leqslant Q(x^{0})$,
			\item if $\phi\in C^{2}(\Omega)$ touches $u$ from above at $x^{0}$, then $|\nabla\phi(x^{0})|\geqslant Q(x^{0})$.
		\end{itemize}
	\end{enumerate}
\end{definition}
\subsection{Main results and plan of the paper}
The main results in this paper state that the free boundary of the problem \eqref{Formula: Semilinear problem} can be decomposed into a disjoint union of a regular set and a singular set (presented in Theorem \ref{Theorem: Main theorem1}), the regular set of the free boundary is the graph of a $C^{1,\gamma}$ function locally, while for the singular set, there exists a critical dimension $d^{*}\in\{5,6,7\}$ so that the singularities can only occur when $d>d^{*}$ (presented in Theorem \ref{Theorem: Main theorem2}).

In proving these two results we need an array of technique tools including a Weiss-type monotonicity formula (presented in Proposition \ref{Proposition: Weiss-type monotonicity formula}), optimal linear growth of solutions (presented in Proposition \ref{Proposition: Optimal linear growth}) and an optimal linear non-degeneracy property (presented in Proposition \ref{Proposition: non-degeneracy of minimizers}). In addition, due to the Lipschitz continuity of minimizers (presented in Proposition \ref{Proposition: Lipschitz regularity of minimizers}), a blow-up analysis (presented in Proposition \ref{Proposition: Properties of blow-up limits}) was derived for every free boundary point.

The monotonicity formula for a homogeneous one-phase Bernoulli-type problem \eqref{Formula: Assumption on f(t)} was first obtained in \cite{W1999}. There special cases were considered for $f(t)\equiv0$. In this paper, we generalize it to more general settings. It should be noted that in Proposition \ref{Proposition: second monotonicity formula} we establish a monotonicity formula relating the rescaling $u_{x^{0},r}$ with its one-homogeneous extension $z_{x^{0},r}$ (See \eqref{Formula: M(21)} for the precise definition).
\begin{theorem}[Structure of the free boundary]\label{Theorem: Main theorem1}
	Let $J$ and $\mathcal{K}$ be given as in \eqref{Formula: Functional J(u)} and \eqref{Formula: Admissible functions K}, respectively. Then there is a solution to the minimization problem $\min_{v\in\mathcal{K}}J(v)$. Every solution $u$ is a Lipschitz viscosity solution to the problem \eqref{Formula: Semilinear problem} in the sense of Definition \ref{Definition of viscosity solutions}.
	
	The set $\varOmega^{+}(u)$ has locally finite perimeter in $\Omega$ and the free boundary $\partial\varOmega^{+}(u)$ can be decomposed into a disjoint union of a regular set $\partial_{\mathrm{reg}}\varOmega^{+}(u)$ and a singular set $\partial_{\mathrm{sing}}\varOmega^{+}(u)$. Namely,   
	\begin{align*}
		\partial\varOmega^{+}(u)=\bigcup_{\tfrac{1}{2}\leqslant\gamma<1}(\varOmega^{+}(u))^{(\vartheta)},
	\end{align*}
    where $(\varOmega^{+}(u))^{(\vartheta)}$ is defined in the sense of \eqref{Formula: Definition of E^{gamma}}. The regular set consists of points in which the Lebesgue density of the set $\varOmega^{+}(u)$ is precisely $\tfrac{1}{2}$, and the singular set consists of points in which the Lebesgue density of $\varOmega^{+}(u)$ is strictly between $\tfrac{1}{2}$ and $1$.
	\begin{align*}
		\partial_{\mathrm{reg}}\varOmega^{+}(u)=\left(\varOmega^{+}(u)\right)^{(1/2)}\qquad\text{ and }\qquad\partial_{\mathrm{sing}}\varOmega^{+}(u)=\bigcup_{\tfrac{1}{2}<\vartheta<1}\left(\varOmega^{+}(u)\right)^{(\vartheta)}.
	\end{align*}	
    Moreover, for every $\vartheta\in[\tfrac{1}{2},1)$ and every blow-up limit $u_{0}$, we have
    \begin{align*}
    	(\varOmega^{+}(u))^{(\vartheta)}=\left\lbrace x\in\partial\varOmega^{+}(u)\colon\vartheta=\frac{|B_{1}\cap\{u_{0}>0\}|}{|B_{1}|}\right\rbrace.
    \end{align*}
\end{theorem}
\begin{theorem}[Regularity of the regular set and the dimension for the singular set]\label{Theorem: Main theorem2}
	Let $u\in\mathcal{K}$ be a minimizer of $J$ in $\Omega$, then the regular part of the free boundary $\partial_{\mathrm{reg}}\varOmega^{+}(u)$ is an open subset of $\partial\varOmega^{+}(u)$ and is locally a graph of a $C^{1,\gamma}$ function for $\gamma\in(0,\beta)$. The singular part of the free boundary $\partial_{\mathrm{sing}}\varOmega^{+}(u)$ is a closed subset of $\partial\varOmega^{+}(u)$ and
	\begin{align*}
		\mathcal{H}^{d-1}(\partial_{\mathrm{sing}}\varOmega^{+}(u))=0.
	\end{align*} 
    Moreover, there exists $d^{*}\in\{5,6,7\}$ such that
	\begin{itemize}
		\item If $d<d^{*}$, then $\partial_{\mathrm{sing}}\varOmega^{+}(u)=\varnothing$ and therefore all free boundary points are regular and $\partial\varOmega^{+}(u)=\partial_{\mathrm{reg}}\varOmega^{+}(u)$.
		\item If $d=d^{*}$, then the singular set $\partial_{\mathrm{sing}}\varOmega^{+}(u)$ contains at most a finite number of isolated points,
		\item If $d>d^{*}$, then the Hausdorff dimension of $\partial_{\mathrm{sing}}\varOmega^{+}(u)$ is less than $d-d^{*}$.
	\end{itemize}
\end{theorem}
\section{Preliminaries}\label{Section: perliminaries}
In this section, we establish some fundamental facts about minimizers of $J$ in $\mathcal{K}$, including the non-negativity and the governing equation (weak sense) for $u$ in $\Omega$. Note that since $u^{0}\in H^{1}(\Omega)$, it is easy to check that $u^{0}\in\mathcal{K}$ and thus $\mathcal{K}\neq\varnothing$. Therefore, the existence of minimizers can be obtained by a compactness argument as in \cite[Lemma 3.1]{AC1981}.
\begin{proposition}[H\"{o}lder regularity]\label{Proposition: Minimizers are Holder continuous}
	Let $u\in\mathcal{K}$ be a minimizer of $J$ in $B_{r}(x^{0})\subset\Omega$, then there exists a constant $C$, so that
	\begin{align*}
		|u(x)-u(y)|\leqslant C|x-y|\log\left(\frac{1}{|x-y|}\right)\quad\text{ in }B_{r/2}(x^{0}).
	\end{align*} 
	In particular, $u\in C^{0,\alpha}(B_{r/2}(x^{0}))$ for some $\alpha\in(0,1)$.
\end{proposition}
\begin{proof}
	Assume without loss of generality that $x^{0}=0$. Let $v\in B_{r}$ be a function given by $\Delta v+f(v)=0$ in $B_{r}$ and $v=u$ on $\partial B_{r}$. The maximum principle gives that $v\leqslant\Psi$ in $\Omega$ and thus $v\in\mathcal{K}$. By the minimality of $u$ in $B_{r}$, one has $J(u;B_{r})\leqslant J(v;B_{r})$ and this implies
	\begin{align}\label{Formula: P(0)}
		\begin{alignedat}{2}
				0&\leqslant\int_{B_{r}}-|\nabla(v-u)|^{2}+2\nabla(v-u)\cdot\nabla v+(F(v)-F(u))\:dx+\int_{B_{r}}Q^{2}(x)(\chi_{\{v>0\}}-\chi_{\{u>0\}})\:dx\\
			&\leqslant\int_{B_{r}}-|\nabla(v-u)|^{2}\:dx+Q_{\mathrm{max}}^{2}\int_{B_{r}}\chi_{\{v>0\}}\:dx,
		\end{alignedat}
	\end{align}
	where we have used the assumption \eqref{Formula: Assumption on Q(x)} in the second inequality. Therefore,
	\begin{align*}
		\int_{B_{r}}|\nabla(u-v)|^{2}\:dx\leqslant Q_{\mathrm{max}}^{2}|B_{r}\cap\{v>0\}|\leqslant Q_{\mathrm{max}}^{2}|B_{r}|\leqslant Cr^{d},
	\end{align*}
	where $C$ depends only on $d$ and $Q_{\mathrm{max}}$. Proceeding as in \cite[Theorem 2.1]{ACF1984TMS} we complete the proof of the proposition.
\end{proof}
\begin{proposition}[Properties of minimizers]\label{Proposition: Properties of minimizers}
	Let $u\in\mathcal{K}$ be a minimizer of $J$ in $\Omega$, then $u\geqslant0$ in $\Omega$ and $u$ satisfies the equation
	\begin{align}\label{Formula: P(1)}
		\Delta u+f(u)\geqslant0\quad\text{ in }\quad\Omega,
	\end{align}
    in the weak sense. Moreover, $u$ satisfies the equation
    \begin{align}\label{Formula: P(2)}
    	\Delta u+f(u)=0\quad\text{ in }\quad\varOmega^{+}(u),
    \end{align}
    in the classical sense. That is, $u\in C^{2,\alpha}(D)$ satisfies \eqref{Formula: P(2)} for any $D\subset\subset\varOmega^{+}(u)$.
\end{proposition}
\begin{proof}
	Choose $u_{\varepsilon}:=u-\varepsilon\min\{u,0\}$ for any $\varepsilon\in(0,1)$. Then $u_{\varepsilon}\in\mathcal{K}$ is admissible. By the minimality of $u$, we have
	\begin{align}\label{Formula: P(2')}
		\int_{\Omega}|\nabla\min\{u,0\}|^{2}\:dx\leqslant-\int_{\Omega}\frac{F'((1-\varepsilon)\min\{u,0\})}{2}\min\{u,0\}\:dx.
	\end{align}
    Since $F'(t)\leqslant0$ for $t\leqslant0$, we have from \eqref{Formula: P(2')} that $\int_{\Omega}|\nabla\min\{u,0\}|^{2}\:dx\leqslant0$ and thus $u\geqslant0$ in $\Omega$. Consider now the first variation $\lim_{\varepsilon\to0}\frac{1}{\varepsilon}(J(u-\varepsilon\xi)-J(u))\geqslant 0$ for any $\xi\geqslant0$, $\xi\in C_{0}^{\infty}(\Omega)$, and we have that
    \begin{align*}
    	0\leqslant-\int_{\Omega}\nabla u\cdot\nabla\xi+\frac{F'(u)}{2}\xi\:dx,
    \end{align*}
    which gives \eqref{Formula: P(1)}. By Proposition \ref{Proposition: Properties of minimizers}, $u$ is a continuous function in $\Omega$ and this implies that $\varOmega^{+}(u)$ is an open subset of $\Omega$. For any $\xi\in C_{0}^{\infty}(\varOmega^{+}(u))$, it is easy to check that the function $u-\varepsilon\xi\in\mathcal{K}$ and $\varOmega^{+}(u-\varepsilon\xi)\subset\varOmega^{+}(u)$ for $|\varepsilon|>0$ small. Thus, by the minimality of $u$ we have
    \begin{align}\label{Formula: P(3)}
    	0&=\int_{\Omega}\nabla u\cdot\nabla\xi+\frac{F'(u-\varepsilon\xi)}{2}\xi\:dx\quad\text{ for any }\quad\varepsilon.
    \end{align}
    Passing to the limit as $\varepsilon\to0$ in \eqref{Formula: P(3)} gives \eqref{Formula: P(2)}. It now follows from the Schauder interior estimate in \cite{GT1998} that $u\in C^{2,\alpha}(D)$ for any compact subset $D$ of $\varOmega^{+}(u)$.
\end{proof}
We now introduce the following Lemma, which helps us to establish the optimal linear growth of minimizers up to the free boundary.
\begin{lemma}\label{Lemma: An estimates}
	Let $\Psi$ be given in \eqref{Formula: Psi}, then for every non-negative function $u\in H^{1}(B_{r}(x^{0}))$ where $B_{r}(x^{0})\subset\Omega$ the following estimates does hold:
	\begin{align}\label{Formula: P(4)}
		|B_{r}(x^{0})\cap\{u=0\}|\left(\frac{1}{r}\dashint_{\partial B_{r}(x^{0})}u\:d\mathcal{H}^{d-1}-r\frac{F_{0}M_{0}}{2}\right)^{2}\leqslant C_{d}\int_{B_{r}(x^{0})}|\nabla(u-h)|^{2}.
	\end{align}
    Here $h$ is a function defined by $\Delta h+f(h)=0$ in $B_{r}(x^{0})$ with boundary value $h=u$ on $\partial B_{r}(x^{0})$ and $M_{0}:=\sup_{\bar{\Omega}}\Psi$.
\end{lemma}
\begin{remark}
	An estimate of the type \eqref{Formula: P(4)} was first brought by Alt and Caffarelli [Lemma 3.2, \cite{AC1981}], and we generalize the result of Lemma 3.2 in \cite{AC1981} to our semilinear case.
\end{remark}
\begin{proof}
	Assume without loss of generality $x^{0}=0$. Let $v\in H^{1}(B_{r})$ be the solution to the problem
	\begin{align*}
		\min\left\lbrace\int_{B_{r}}|\nabla v|^{2}+F(v)\:dx\colon v-u\in H_{0}^{1}(B_{r}),\ \ v\geqslant u\right\rbrace.
	\end{align*}
    Then $v$ satisfies $\Delta v+f(v)\leqslant0$ in $B_{r}$ and $\Delta v+f(v)=0$ in $B_{r}\cap\{v>u\}$. Now for each $|z|\leqslant\frac{1}{2}$ we consider a transformation $T\colon B_{r}\to B_{r}$ given by $T_{z}(x):=(r-|x|)z+x$ for each $x\in B_{r}$. Define
    \begin{align*}
    	u_{z}(x):=u(T_{z}(x))\qquad\text{ and }\qquad v_{z}(x):=v(T_{z}(x)).
    \end{align*}
    Let
    \begin{align*}
    	S_{z}:=\left\lbrace\xi\in\partial B_{1}\colon u_{z}(\rho\xi)=0\quad\text{ for some }\quad\rho\in\left[\frac{r}{8},r\right]\right\rbrace.
    \end{align*}
    For $\xi\in S_{z}$, we define
    \begin{align*}
    	r_{\xi}:=\inf\left\lbrace\rho\in\left[\frac{r}{8},r\right]\colon u_{z}(\rho\xi)=0,\ \ \xi\in S_{z}\right\rbrace.
    \end{align*}
    By a slicing argument for $\mathcal{H}^{1}$ a.e. $\xi\in S_{z}$ the function $\rho\mapsto\nabla u_{z}(\rho\xi)$ and $\rho\mapsto\nabla v_{z}(\rho\xi)$ are square integrable. Define a function $g(\rho):=u_{z}(\rho\xi)-v_{z}(\rho\xi)$ for $0<\rho\leqslant r$, then $g(\rho)$ is absolutely continuous in $(0,r)$, and so, using also the fact that $g(r)=0$ (since $u=v$ on $\partial B_{r}$) and $u_{z}(r_{\xi}\xi)=0$, we have $v_{z}(r_{\xi}\xi)=g(r)-g(r_{\xi})=\int_{r_{\xi}}^{r}\nabla(v_{z}(\rho\xi)-u_{z}(\rho\xi))\cdot\xi\:d\rho$. Using H\"{o}lder inequality we have
    \begin{align}\label{Formula: P(5)}
    	v_{z}(r_{\xi}\xi)\leqslant\sqrt{r-r_{\xi}}\left(\int_{r_{\xi}}^{r}|\nabla(v_{z}-u_{z})(\rho\xi)|^{2}\:d\rho\right)^{1/2}.
    \end{align}
    Since $\Delta h=-f(h)\leqslant-f(0)+\frac{F_{0}M_{0}}{2}\leqslant\frac{F_{0}M_{0}}{2}$ for any $x\in B_{r}$, where we have used the fact $f(0)\geqslant0$ in the last inequality. Since $h\geqslant0$ on $\partial B_{r}$ satisfying $\Delta h\leqslant M$ in $B_{r}$ for $M:=\frac{F_{0}M_{0}}{2}\geqslant0$, it follows from Lemma 2.4 (a) in \cite{GS1996} that
    \begin{align*}
    	h(x)&\geqslant  r^{d}\frac{r-|x|}{(r+|x|)^{d-1}}\left(\frac{1}{r^{2}}\dashint_{\partial B_{r}}u\:d\mathcal{H}^{d-1}-\frac{2^{d-1}M}{d}\right)\geqslant C_{d}\frac{r-|x|}{r}\left(\dashint_{\partial B_{r}}u\:d\mathcal{H}^{d-1}-r^{2}M\right).
    \end{align*}
    Since $v$ is a super-solution in $B_{r}$ with the same boundary value of $h$, we have $v(x)\geqslant h(x)$ in $B_{r}$. Taking $x=(r-r_{\xi})z+r_{\xi}\xi$ yields
    \begin{align}\label{Formula: P(6)}
    	v_{z}(r_{\xi}\xi)=v((r-r_{\xi})z+r_{\xi}\xi)\geqslant C_{d}
    	\frac{r-r_{\xi}}{r}\left(\dashint_{\partial B_{r}}u\:d\mathcal{H}^{d-1}-r^{2}M\right).
    \end{align}
    Combined \eqref{Formula: P(6)} with \eqref{Formula: P(5)} together gives
    \begin{align*}
    	(r-r_{\xi})\left(\frac{1}{r}\dashint_{\partial B_{r}}u\:d\mathcal{H}^{d-1}-rM\right)^{2}&\leqslant C_{d}\int_{r_{\xi}}^{r}|\nabla(v_{z}-u_{z})(\rho\xi)|^{2}\:d\rho.
    \end{align*}
    Integrating over $\xi\in S_{z}\subset\mathbb{S}^{d-1}$, we obtain the inequality
    \begin{align*}
    	\left(\int_{S_{z}}(r-r_{\xi})\:d\xi\right)\left(\frac{1}{r}\dashint_{\partial B_{r}}u\:d\mathcal{H}^{d-1}-rM\right)^{2}\leqslant C_{d}\int_{\partial B_{1}}\int_{r_{\xi}}^{r}|\nabla(v_{z}-u_{z})(\rho\xi)|^{2}\:d\rho d\xi.
    \end{align*}
    By the uniform estimate that $r/8\leqslant r_{\xi}\leqslant r$, we have
    \begin{align*}
    	|B_{r}\setminus B_{r/4}(rz)\cap\{u=0\}|\left(\frac{1}{r}\dashint_{\partial B_{r}}u\:d\mathcal{H}^{d-1}-rM\right)^{2}&\leqslant C_{d}\int_{B_{r}}|\nabla(v_{z}-u_{z})|^{2}\:dx\\ 
    	&\leqslant C_{d}\int_{B_{r}}|\nabla(v-u)|^{2}\:dx,
    \end{align*}
    where we have used the fact that the Jacobian of $T$ is bounded by a dimensional constant which is independent of $r$. The desired result \eqref{Formula: P(4)} follows by integrating over $z$.
\end{proof}
With Lemma \ref{Lemma: An estimates} at hand, we immediately have
\begin{proposition}[Optimal linear growth of minimizers]\label{Proposition: Optimal linear growth}
	Let $u\in\mathcal{K}$ be a minimizer of $J$ in $\Omega$, and let $x^{0}\in\Omega$ be such that $|B_{r}(x^{0})\cap\{u=0\}|\neq0$ for $B_{r}(x^{0})\subset\Omega$, then
	\begin{align}\label{Formula: P(7)}
		\frac{1}{r}\dashint_{\partial B_{r}(x^{0})}u\:d\mathcal{H}^{d-1}\leqslant\sqrt{C_{d}}Q_{\mathrm{max}}+r\frac{F_{0}M_{0}}{2},
	\end{align}
	where $F_{0}$ is given in \eqref{Formula: Assumption on F(t)} and $M_{0}=\sup_{\bar{\Omega}}\Psi$.
\end{proposition}
\begin{proof}
	Assume $x^{0}=0$ and let $h$ be a function defined by $\Delta h+f(h)=0$ in $B_{r}$ and $h=u$ on $\partial B_{r}$. The maximum principle gives that $h\leqslant\Psi$ in $B_{r}$ and thus $h\in\mathcal{K}$ is admissible. By the optimality $J(u)\leqslant J(h)$ and a similar calculation as in \eqref{Formula: P(0)} gives
	\begin{align}\label{Formula: P(8)}
		\int_{B_{r}}|\nabla(u-h)|^{2}\leqslant Q_{\mathrm{max}}^{2}|B_{r}\cap\{u=0\}|.
	\end{align}
    Applying Lemma \ref{Lemma: An estimates} and it follows from \eqref{Formula: P(4)} and \eqref{Formula: P(8)} that
    \begin{align*}
    	|B_{r}\cap\{u=0\}|\left(\frac{1}{r}\dashint_{\partial B_{r}}u\:d\mathcal{H}^{d-1}-r\frac{F_{0}M_{0}}{2}\right)^{2}\leqslant C_{d}Q_{\mathrm{max}}^{2}|B_{r}\cap\{u=0\}|.
    \end{align*}
    By our assumption $|B_{r}\cap\{u=0\}|\neq0$ and we obtain \eqref{Formula: P(7)}.
\end{proof}
\begin{remark}\label{Remark: Optimal linear growth}
	In view of Proposition \ref{Proposition: Optimal linear growth} and if in particular  $x^{0}\in\partial\varOmega^{+}(u)$ is a free boundary point, then choose $B_{r}(x^{0})$ be a small ball with $r\leqslant\min\{\frac{2Q_{\mathrm{max}}}{F_{0}M_{0}},1\}$, we have
	\begin{align*}
		\frac{1}{r}\dashint_{\partial B_{r}(x^{0})}u\:d\mathcal{H}^{d-1}\leqslant C_{d}^{*}Q_{\mathrm{max}},
	\end{align*}
    where $C_{d}^{*}:=\sqrt{C_{d}}+1$.
\end{remark}
\section{Optimal regularity and non-degeneracy}
The first goal of this section is to prove that minimizers are locally Lipschitz functions in $\Omega$. Precisely, we found that if a minimizer grows like its distance to the free boundary, then it is Lipschitz continuous near the free boundary. We summarize it as follows:
\begin{proposition}[Lipschitz regularity of minimizers]\label{Proposition: Lipschitz regularity of minimizers}
	Let $u\in\mathcal{K}$ be a minimizer of $J$ in $\Omega$, then $u\in C^{0,1}(\Omega)$. Moreover, for every $D\subset\subset\Omega$ containing a free boundary point, one has
	\begin{align}\label{Formula: O(''(0))}
		\sup_{D}|\nabla u(x)|\leqslant C(d,F_{0},D,\Omega)(Q_{\mathrm{max}}+F_{0}),
	\end{align}
    where $F_{0}$ is given in \eqref{Formula: Assumption on F(t)}.
\end{proposition}
\begin{proof}
	Suppose that $x^{0}\in\Omega$ with $\rho(x^{0})<\frac{1}{2}\min\{\operatorname{dist}(x^{0},\partial\Omega),1\}$ where $\rho(x^{0}):=\operatorname{dist}(x^{0},\{u=0\})$. Then we claim that
	\begin{align}\label{Formula: O('0)}
		\frac{u(x^{0})}{\rho(x^{0})}\leqslant C(d,F_{0})(Q_{\mathrm{max}}+F_{0}),
	\end{align}
    where $F_{0}$ is given in \eqref{Formula: Assumption on F(t)}. Assume $\tfrac{u(x^{0})}{\rho(x^{0})}>M>0$ for some $M>0$, and we derive an upper bound for $M$. Consider the function $u_{\rho}(x):=\frac{u(x^{0}+\rho(x-x^{0}))}{\rho}$ for $\rho:=\rho(x^{0})$ and $x\in B_{1}(x^{0})$. Then $u_{\rho}$ satisfies $\Delta u_{\rho}+\rho f(\rho u_{\rho})=0$ in $B_{1}(x^{0})$. A direct calculation gives that
    \begin{align}\label{Formula: O(0)}
    	|\rho f(\rho u_{\rho})|\leqslant\rho|f(0)|+\rho^{2}\max_{t\geqslant0}|f'(t)|u_{\rho}\leqslant\frac{\rho}{2}F_{0}+\frac{\rho^{2}}{2}F_{0}u_{\rho}.
    \end{align}
    Thanks to the Harnack inequality \cite[Corollary 1.1]{T1967}, we have
    \begin{align}\label{Formula: O(1)}
    	\inf_{B_{3/4}(x^{0})}u_{\rho}\geqslant cM-CF_{0},
    \end{align}
    where $c$ and $C$ are constants depending on $d$ and $F_{0}$. Let $y\in\partial B_{1}(x^{0})\cap\{u_{\rho}=0\}$ be a point and $v$ be a function defined by $\Delta v+\rho f(\rho v)=0$ in $B_{1}(y)$ and $v=u_{\rho}$ on $\partial B_{1}(y)$. Since $u_{\rho}$ is a minimizer, we have from  $J(u_{\rho};B_{1}(y))\leqslant J(v;B_{1}(y))$, the convexity of $F$, and a similar calculation as in \eqref{Formula: P(0)} that
    \begin{align}\label{Formula: O(2)}
    	\int_{B_{1}(y)}|\nabla(u_{\rho}-v)|^{2}\:dx\leqslant Q_{\mathrm{max}}^{2}\int_{B_{1}(y)}\chi_{\{u_{\rho}=0\}}\:dx.
    \end{align}
    It then follows from the maximum principle that $v\geqslant u_{\rho}$ in $B_{1}(y)$. This combined with \eqref{Formula: O(1)} yields $v\geqslant u_{\rho}\geqslant cM-CF_{0}>0$ in $B_{3/4}(x^{0})\cap B_{1}(y)$. Using Harnack inequality \cite[Corollary 1.1]{T1967} for the function $v$ once again gives
    \begin{align}\label{Formula: O(3)}
    	v(x)\geqslant cM-CF_{0}\quad\text{ in }\quad B_{1/2}(y).
    \end{align}
    Take $y=0$ for simplicity and define
    \begin{align*}
    	\varphi(x):=C_{0}(e^{-\beta|x|^{2}}-e^{-\beta})\qquad C_{0}:=cM-CF_{0}.
    \end{align*}
    After a direct computation, $\Delta\varphi+\rho f(\rho\varphi)>0$ for $\frac{1}{2}<|x|<1$, provided that $\beta>0$ is sufficiently large. This implies that $\Delta(\varphi-v)+\rho f(\rho\varphi)-\rho f(\rho v)>0$ in $B_{1}\setminus B_{1/2}$. It follows from the maximum principle that $v(x)\geqslant\varphi(x)=C_{0}(e^{-\beta|x|^{2}}-e^{-\beta})\geqslant cC_{0}(1-|x|)$ in $B_{1}\setminus B_{1/2}$, this together with \eqref{Formula: O(3)} gives
    \begin{align}\label{Formula: O(4)}
    	v(x)\geqslant(cM-CF_{0})(1-|x|)\quad\text{ in }\quad B_{1}.
    \end{align}
    With the aid of \eqref{Formula: O(2)} and \eqref{Formula: O(4)}, it follows from a similar argument as in \cite[Lemma 3.2]{AC1981} and \cite[Lemma 2.2]{ACF1984} that $(cM-CF_{0})^{2}\leqslant CQ_{\mathrm{max}}^{2}$. This implies $M\leqslant C(Q_{\mathrm{max}}+F_{0})$ and proves \eqref{Formula: O('0)}.
    
    We now prove that $u\in C^{0,1}(\Omega)$. Let $\rho(x^{0})<\tfrac{1}{3}\min\{\operatorname{dist}(x^{0},\partial\Omega),1\}$ where $\rho(x^{0})=\operatorname{dist}(x^{0},\{u=0\})$, and let us define $\tilde{u}(\tilde{x}):=\frac{u(x^{0}+\rho(x^{0})\tilde{x})}{\rho(x^{0})}$. Then it follows from \eqref{Formula: O('0)} that $\tilde{u}(\tilde{x})\leqslant C(Q_{\mathrm{max}}+F_{0})$ in $B_{1}$ for $C$ depends on $d$ and $F_{0}$. Since $\Delta\tilde{u}+\rho(x^{0})f(\rho(x^{0})\tilde{u})$, we have from \eqref{Formula: O('0)},  \eqref{Formula: O(0)} and elliptic estimates in \cite{GT1998} that
    \begin{align*}
    	|\nabla\tilde{u}(0)|\leqslant C(\|u\|_{C^{0}(B_{1})}+\rho(x^{0})\|f(\rho(x^{0})u)\|_{L^{\infty}(B_{1})})\leqslant C(Q_{\mathrm{max}}+F_{0}),
    \end{align*}
    where $C$ depends only on $d$ and $F_{0}$. Therefore
    \begin{align*}
    	|\nabla u(x^{0})|=|\nabla\tilde{u}(0)|\leqslant C(Q_{\mathrm{max}}+F_{0}).
    \end{align*}
    Thus, for any $D\subset\subset\Omega$, $|\nabla u(x)|$ is bounded in $D\cap\varOmega^{+}(u)\cap N$, where $N$ is a small neighborhood of the free boundary. Since $u\in C^{2,\alpha}(\varOmega^{+}(u))$ and $\nabla u=0$ a.e. in  $\{u=0\}\cap\Omega$, we have $u\in C^{0,1}(\Omega)$. The rest of the proof is standard and can be obtained by proceeding as in [Page 11, Theorem 2.3, \cite{ACF1984}].
\end{proof}
The second goal of this section is to establish a non-degeneracy property for minimizers, which gives us a uniform $L^{\infty}$ bound of minimizers up to the free boundary $\partial\varOmega^{+}(u)$.
\begin{proposition}[Non-degeneracy of minimizers]\label{Proposition: non-degeneracy of minimizers}
	For any $\kappa\in(0,1)$, there exists a positive constant $c_{d}^{*}=c_{d}^{*}(\kappa)$, so that for any $B_{r}(x^{0})\subset\Omega$ with $r<\frac{1}{F_{0}}\max\{c_{d}^{*}Q_{\mathrm{min}},1\}$, the following does hold:
	\begin{align}\label{Formula: Nondegeneray}
		\frac{1}{r}\dashint_{\partial B_{r}(x^{0})}u\:d\mathcal{H}^{d-1}\leqslant c_{d}^{*}Q_{\mathrm{min}}
	\end{align}
	implies that
	\begin{align*}
		u(x)\equiv0\quad\text{ in }B_{\kappa r}(x^{0}).
	\end{align*}
\end{proposition}
\begin{proof}
	Assume $x^{0}=0$ and consider the function $u_{r}$ for some $r>0$ so that $\{x\in\Omega\colon rx\}\subset\Omega$. Then obviously, $\Delta u_{r}+rf(ru_{r})=0$ in $\varOmega^{+}(u_{r})$. Define
	\begin{align}\label{Formula: O(5)}
		\ell_{u}:=\sup_{B_{\sqrt{\kappa}}}u_{r}\quad\text{ for some }\kappa\in(0,1).
	\end{align}
	Since $\Delta u_{r}\geqslant-rf(ru_{r})\geqslant-rf(0)\geqslant-\frac{rF_{0}}{2}$ in $\Omega$, it follows from the maximum principle that
	\begin{align}\label{Formula: O(6)}
		u_{r}(x)\leqslant\frac{rF_{0}}{2}\frac{1-|x|^{2}}{2d}+\frac{1-|x|^{2}}{d\omega_{d}}\int_{\partial B_{1}}\frac{u_{r}}{|x-y|^{d}}\:d\mathcal{H}^{d-1}\quad\text{ for every }x\in B_{\sqrt{\kappa}}.
	\end{align}
    Taking supremum on the both sides of \eqref{Formula: O(6)} and using \eqref{Formula: O(5)}, we have
    \begin{align}\label{Formula: O(7)}
    	\ell_{u}\leqslant\frac{rF_{0}}{2}+C_{d}(\kappa)\dashint_{\partial B_{1}}u_{r}\:d\mathcal{H}^{d-1}.
    \end{align}
    Let $v$ be a function defined by $\Delta v+rf(rv)=0$ in $B_{\sqrt{\kappa}}\setminus B_{\kappa}$, $v=0$ in $B_{\kappa}$ and $v=\ell_{u}$ outside $B_{\sqrt{\kappa}}$. Then the standard elliptic estimates yields
    \begin{align}\label{Formula: O(8)}
    	\sup_{\partial B_{\kappa}}|\nabla v|\leqslant C(\sup_{B_{\sqrt{\kappa}}}v+r\|f(rv)\|_{L^{\infty}(B_{\sqrt{\kappa}})})\leqslant C(\ell_{u}+rF_{0}),
    \end{align}
    where $C$ is a constant depending on $\kappa$ and $F_{0}$. Note that the function $w:=\min\{u_{r},v\}$ is an admissible function and so $J(u_{r};B_{\sqrt{\kappa}})\leqslant J(w;B_{\sqrt{\kappa}})$. On the other hand, $\int_{B_{\kappa}}|\nabla w|^{2}+F(rw)+Q^{2}(rx)\chi_{\{w>0\}}\:dx=0$, due to the fact that $v=0$ in $B_{\kappa}$. Thus
    \begin{align}\label{Formula: O(9)}
    	\begin{alignedat}{5}
    		J(u_{r};B_{\kappa})&=\int_{B_{\kappa}}|\nabla u_{r}|^{2}+F(ru_{r})+Q^{2}(rx)\chi_{\{u_{r}>0\}}\:dx\\
    		&\leqslant\int_{B_{\sqrt{\kappa}}\setminus B_{\kappa}}|\nabla w|^{2}-|\nabla u_{r}|^{2}+F(rw)-F(ru_{r})+Q^{2}(rx)\left(\chi_{\{w>0\}}-\chi_{\{u_{r}>0\}}\right)\:dx\\
    		&\leqslant\int_{B_{\sqrt{\kappa}}\setminus B_{\kappa}}-|\nabla(w-u_{r})|^{2}+2\nabla w\cdot\nabla(w-u_{r})+rF'(rw)(w-u_{r})\:dx\\
    		&\leqslant\int_{B_{\sqrt{\kappa}}\setminus B_{\kappa}}-|\nabla\min\{v-u_{r},0\}|^{2}+2\nabla v\cdot\nabla\min\{v-u_{r},0\}+rF'(rv)\min\{v-u_{r},0\}\:dx\\
    		&\leqslant2\int_{\partial B_{\kappa}}\min\{v-u_{r},0\}\pd{v}{\nu}\:d\mathcal{H}^{d-1},
    	\end{alignedat}
    \end{align}
    where $\nu$ is the unit normal vector. Here we have used the fact that $\int_{B_{\sqrt{\kappa}}\setminus B_{\kappa}}\chi_{\{w>0\}}-\chi_{\{u_{r}>0\}}\:dx\leqslant0$ in the third inequality since $\{w>0\}\subset\{u_{r}>0\}$. By the convexity of $F$, we have $F(ru_{r})\geqslant F'(0)ru_{r}\geqslant-rF_{0}u_{r}$ and this gives
    \begin{align}\label{Formula: O(10)}
    	\int_{B_{\kappa}}|\nabla u_{r}|^{2}+Q^{2}(rx)\chi_{\{u_{r}>0\}}\:dx\leqslant J(u_{r};B_{\kappa})+\int_{B_{\kappa}}rF_{0}u_{r}\:dx.
    \end{align}
    It follows from \eqref{Formula: O(7)}, \eqref{Formula: O(8)}, \eqref{Formula: O(9)}, \eqref{Formula: O(10)} and the trace theorem that
    \begin{align}\label{Formula: O(11)}
    	\begin{alignedat}{5}
    		&\int_{B_{\kappa}}|\nabla u_{r}|^{2}+Q^{2}(rx)\chi_{\{u_{r}>0\}}\:dx\\
    		&\leqslant2\int_{\partial B_{\kappa}}\min\{v-u_{r},0\}\pd{v}{\nu}\:d\mathcal{H}^{d-1}+\int_{B_{\kappa}}rF_{0}u_{r}\:dx\\
    		&\leqslant C(\ell_{u}+rF_{0})\left(\int_{\partial B_{\kappa}}u_{r}\:d\mathcal{H}^{d-1}+\int_{B_{\kappa}}u_{r}\:dx\right)\\
    		&\leqslant C(\ell_{u}+rF_{0})\left(\int_{B_{\kappa}}u_{r}\chi_{\{u_{r}>0\}}\:dx+C\int_{B_{\kappa}}|\nabla u_{r}|\chi_{\{u_{r}>0\}}\:dx\right)\\
    		&\leqslant C(\ell_{u}+rF_{0})\left(\frac{\ell_{u}+rF_{0}}{Q_{\mathrm{min}}^{2}}\int_{B_{\kappa}}\chi_{\{u_{r}>0\}}\:dx+\frac{C}{Q_{\mathrm{min}}}\int_{B_{\kappa}}|\nabla u_{r}|^{2}+Q^{2}(rx)\chi_{\{u_{r}>0\}}\:dx\right)\\
    		&\leqslant\frac{C(\ell_{u}+rF_{0})}{Q_{\mathrm{min}}}\left[\frac{(\ell_{u}+rF_{0})}{Q_{\mathrm{min}}}+1\right]\int_{B_{\kappa}}|\nabla u_{r}|^{2}+Q^{2}(rx)\chi_{\{u_{r}>0\}}\:dx,
    	\end{alignedat}
    \end{align}
    where $C$ is a constant depending on $d$, $F_{0}$ and $\kappa$. Thanks to \eqref{Formula: Nondegeneray} and \eqref{Formula: O(7)}, we have
    \begin{align*}
    	\frac{\ell_{u}+rF_{0}}{Q_{\mathrm{min}}}\leqslant C\left(\frac{1}{Q_{\mathrm{min}}}\dashint_{\partial B_{1}}u_{r}\:d\mathcal{H}^{d-1}+\frac{rF_{0}}{Q_{\mathrm{min}}}\right)\leqslant Cc_{d}^{*}.
    \end{align*}
    This implies that $\frac{\ell_{u}+rF_{0}}{Q_{\mathrm{min}}}$ is sufficiently small, provided that $c_{d}^{*}$ is sufficiently small. This together with the last inequality in \eqref{Formula: O(11)} implies that
    \begin{align*}
    	\int_{B_{\kappa}}|\nabla u_{r}|^{2}+Q^{2}(rx)\chi_{\{u_{r}>0\}}\:dx=0,
    \end{align*}
    for $c_{d}^{*}$ small enough. Thus $u_{r}=0$ in $B_{\kappa}$, provided that $c_{d}^{*}$ is small enough. This concludes the proof.
\end{proof}
\begin{remark}
	Proposition \ref{Proposition: non-degeneracy of minimizers} remains true if $B_{r}(x^{0})$ is not contained in $\Omega$, provided that $u=0$ in $B_{r}(x^{0})\cap\partial\Omega$.
\end{remark}
The following corollaries are direct applications of Proposition \ref{Proposition: non-degeneracy of minimizers} and Remark \ref{Remark: Optimal linear growth}.
\begin{corollary}\label{Corollary: Weak solution-the second requirement}
	Let $u\in\mathcal{K}$ be a minimizer of $J$ in $\Omega$, then for any $D\subset\subset\Omega$, there are constants $0\leqslant c_{d}^{*}\leqslant C_{d}^{*}$ so that for every $B_{r}(x^{0})\subset D$ with $x^{0}\in\partial\varOmega^{+}(u)$
	\begin{align*}
		c_{d}^{*}Q_{\mathrm{min}}\leqslant\frac{1}{r}\dashint_{\partial B_{r}(x^{0})}u\:d\mathcal{H}^{d-1}\leqslant C_{d}^{*}Q_{\mathrm{max}}.
	\end{align*}
\end{corollary}
\begin{corollary}[optimal linear nondegeneracy]\label{Corollary: Optimal linear nondegeneracy}
	Let $u\in\mathcal{K}$ be a local minimizer of $J$ in $B_{r}(x^{0})\subset\Omega$ with $x^{0}\in\partial\varOmega^{+}(u)$. Then 
	\begin{align*}
		\sup_{B_{r}(x^{0})}|u|\geqslant rc_{d}^{*}Q_{\mathrm{min}}.
	\end{align*}
\end{corollary}
In the end of this section, we prove a density estimate that will be used in our forthcoming study.
\begin{proposition}[Density estimates]\label{Proposition: density estimates}
	Let $D\subset\subset\Omega$, then there exists a positive constant $c\in(0,1)$, such that for any $B_{r}(x^{0})\subset D$, centered at $x^{0}\in\partial\varOmega^{+}(u)$ with $r>0$ small,
	\begin{align}\label{Formula: O(12)}
		c<\frac{|B_{r}(x^{0})\cap\{u>0\}|}{|B_{r}(x^{0})|}<1-c.
	\end{align}
\end{proposition}
\begin{proof}
	Without loss of generality, we assume that $x^{0}=0$, and we first prove the estimate by below. It follows from Corollary \ref{Corollary: Optimal linear nondegeneracy} that there exists a point $y\in\partial B_{r}$ such that $u(y)>\eta r$ for some constant $\eta>0$. Define a function $v(x)=u(x)+\frac{F_{0}|x|^{2}}{4d}$. A direct calculation gives that $\Delta v=\Delta u+\frac{F_{0}}{2}\geqslant f(0)-f(u)\geqslant0$ in $B_{r}$. This implies that for $\kappa>0$ small, 
	\begin{align*}
		\frac{F_{0}r}{2d\kappa}+\frac{1}{\kappa r}\dashint_{\partial B_{\kappa r}(y)}u\:d\mathcal{H}^{d-1}\geqslant\frac{1}{\kappa r}\dashint_{\partial B_{\kappa r}(y)}v\:d\mathcal{H}^{d-1}\geqslant\frac{v(y)}{\kappa r}\geqslant\frac{\eta}{\kappa},
	\end{align*}
	for sufficiently small $\kappa>0$. It follows that 
	\begin{align*}
		\frac{1}{\kappa r}\dashint_{\partial B_{\kappa r}(y)}u\:d\mathcal{H}^{d-1}\geqslant\frac{\eta}{\kappa}-\frac{F_{0}r}{2d\kappa}>\frac{\eta}{4\kappa},
	\end{align*}
	provided that $r>0$ is sufficiently small. By Proposition \ref{Proposition: non-degeneracy of minimizers}, we have that $u>0$ in $B_{\kappa r}(y)$, which gives the lower bound.
	
	 Consider a function $v$ defined in $\bar{B}_{r}$ by $\Delta v+f(v)=0$ in $B_{r}$ and $v=u$ on $\partial B_{r}$. The maximum principle implies that $u\leqslant v\leqslant\Psi$ in $B_{r}$ where $\Psi$ is defined in \eqref{Formula: Psi}. Therefore $v\in\mathcal{K}$ is admissible. The minimality of $u$ and a similar calculation as in \eqref{Formula: P(0)} gives $\int_{B_{r}}|\nabla(u-v)|^{2}\:dx\leqslant Q_{\mathrm{max}}^{2}|B_{r}\cap\{u=0\}|$. Thanks to Poincar\'{e} and H\"{o}lder inequality, one has
     \begin{align}\label{Formula: O(13)}
     	Q_{\mathrm{max}}^{2}|B_{r}\cap\{u=0\}|\geqslant\int_{B_{r}}|\nabla(v-u)|^{2}\:dx\geqslant\frac{c}{r^{2}}\int_{B_{r}}|v-u|^{2}\:dx.
     \end{align}
     For any $y\in B_{kr}$, since $\Delta v\leqslant \frac{M_{0}F_{0}}{2}-f(0)$ with $M_{0}=\sup_{\bar{\Omega}}\Psi$, a direct calculation gives
     \begin{align}\label{Formula: O(14)}
     	\begin{alignedat}{2}
     		v(y)&\geqslant\frac{r^{2}-|y|^{2}}{d\omega_{d}r}\int_{\partial B_{r}}\frac{u}{|x-y|^{d}}\:d\mathcal{H}^{d-1}+(M_{0}F_{0}-f(0))\frac{|y|^{2}-r^{2}}{2d}\\
     		&\geqslant(1-C\kappa)\dashint_{\partial B_{r}}u\:d\mathcal{H}^{d-1}-Cr^{2}.
     	\end{alignedat}
     \end{align}
     It follows from \eqref{Formula: O(14)}, Proposition \ref{Proposition: Lipschitz regularity of minimizers} Proposition \ref{Proposition: non-degeneracy of minimizers} that
     \begin{align}\label{Formula: O(15)}
     	\begin{alignedat}{3}
     		v(y)-u(y)&\geqslant v(u)-C\kappa r\\
     		&\geqslant(1-C\kappa)\dashint_{\partial B_{r}}u\:d\mathcal{H}^{d-1}-C\kappa r-Cr^{2}\\
     		&\geqslant cQ_{\mathrm{min}}r,
     	\end{alignedat}
     \end{align}
     provided that $\kappa$ and $r$ are sufficiently small. 
     
     Combining with \eqref{Formula: O(13)} and \eqref{Formula: O(15)}, one has
     \begin{align}\label{Formula: O(16)}
     	Q_{\mathrm{max}}^{2}|B_{r}\cap\{u=0\}|\geqslant\frac{c}{r^{2}}\int_{B_{r}}|v-u|^{2}\:dx\geqslant\frac{c}{r^{2}}\int_{B_{kr}}|v-u|^{2}\:dx\geqslant c(\kappa)Q_{\mathrm{min}}^{2}r^{d},
     \end{align}
     for $r>0$ sufficiently small. Thus, in view of \eqref{Formula: O(16)}, we have
     \begin{align*}
     	\frac{|B_{r}\cap\{u=0\}|}{|B_{r}|}\geqslant c\frac{Q_{\mathrm{min}}}{Q_{\mathrm{max}}},
     \end{align*}
     where $c$ is independent of $r$. This gives the upper bound in \eqref{Formula: O(12)} since $\frac{Q_{\mathrm{min}}}{Q_{\mathrm{max}}}>0$.
\end{proof}
\section{Locally finite perimeter of the positivity set and blow-up limits}
In this section we prove that the set $\varOmega^{+}(u)$ has (locally) finite perimeter in $\Omega$. The first step is a measure estimate of the free boundary. To begin with, let us define
\begin{align}\label{Formula: W(0)}
	\lambda:=\Delta u+f(u),
\end{align}
and
\begin{align}\label{Formula: W(1)}
	\lambda_{0}:=\Delta u+f(u)\chi_{\{u>0\}}.
\end{align}
Then we immediately have
\begin{lemma}
	Let $\lambda$ and $\lambda_{0}$ be defined as in \eqref{Formula: W(0)} and \eqref{Formula: W(1)}, respectively, then  $\lambda$ is a Radon measure supported on $\Omega\cap\{u=0\}$ and its singular point is contained in the free boundary. $\lambda_{0}$ is a positive Radon measure supported in $\partial\varOmega^{+}(u)$.
\end{lemma}
\begin{proof}
	For any $\xi\in C_{0}^{\infty}(\Omega)$ with $\xi\geqslant0$, define $\eta(t):=\max\{\min\{2-t,1\},0\}$ for any $t>0$. Then
	\begin{align}\label{Formula: W(1')}
		\begin{alignedat}{2}
			\int_{\Omega}\nabla u\cdot\nabla(\xi\eta(tu))\:dx&=\int_{\Omega}\nabla u\cdot\nabla\xi\:dx+\int_{\Omega\cap\{u\geqslant\tfrac{1}{t}\}}\nabla u\cdot\nabla(\xi\eta(tu)-\xi)\:dx\\
			&=\int_{\Omega}\nabla u\cdot\nabla\xi\:dx-\int_{\varOmega^{+}(u)}f(u)\xi\:dx\\
			&\quad+\int_{\Omega\cap\{\tfrac{1}{t}\leqslant u\leqslant\tfrac{2}{t}\}}(2-ut)f(u)\xi\:dx+\int_{\Omega\cap\{u\leqslant\tfrac{2}{t}\}}f(u)\xi\:dx\\
			&\geqslant\int_{\Omega}\nabla u\cdot\nabla\xi\:dx-\int_{\varOmega^{+}(u)}f(u)\xi\:dx,
		\end{alignedat}
	\end{align}
    for $t>0$ large enough, where we have used the fact $f(u)\geqslant\frac{f(0)}{2}\geqslant0$ in $\Omega\cap\{u\leqslant\tfrac{2}{t}\}$ for $t>0$ large enough. On the other hand, one has
    \begin{align}\label{Formula: W(1'')}
    	\begin{alignedat}{2}
    		\int_{\Omega}\nabla u\cdot\nabla(\xi\eta(tu))\:dx&=\int_{\Omega\cap\{\tfrac{1}{t}\leqslant u\leqslant\tfrac{2}{t}\}}(2-ut)\nabla u\cdot\nabla\xi-\xi t|\nabla u|^{2}\:dx+\int_{\Omega\cap\{0<u\leqslant\tfrac{1}{t}\}}\nabla u\cdot\nabla\xi\:dx\\
    		&\leqslant\int_{\Omega\cap\{0<u\leqslant\tfrac{2}{t}\}}|\nabla u||\nabla\xi|\:dx.
    	\end{alignedat}
    \end{align}
    It follows from \eqref{Formula: W(1')} and \eqref{Formula: W(1'')} that
    \begin{align}\label{Formula: W(1''')}
    	\int_{\Omega}\nabla u\cdot\nabla\xi\:dx-\int_{\varOmega^{+}(u)}f(u)\xi\:dx\leqslant\int_{\Omega\cap\{0<u\leqslant\tfrac{2}{t}\}}|\nabla u||\nabla\xi|\:dx.
    \end{align} 
   Passing to the limit as $n\to\infty$ in \eqref{Formula: W(1''')}, we conclude that $\Delta u+f(u)\chi_{\{u>0\}}\geqslant0$ in the sense of distributions. Consequently, there exists a positive Radon measure $\lambda_{0}$ supported on $\partial\varOmega^{+}(u)$ so that $\lambda_{0}=\Delta u+f(u)\chi_{\{u>0\}}$. Recalling \eqref{Formula: Assumption on f(t)} that $f(0)\geqslant0$ there exists a positive Radon measure $\lambda$ supported on $\Omega\cap\{u=0\}$ so that $\lambda=\lambda_{0}+f(0)\chi_{\{u=0\}}=\Delta u+f(u)$.
\end{proof}
Compare to \cite[Theorem 4.3]{AC1981}, we have the following result adjusted to our semilinear settings.
\begin{proposition}\label{Proposition: The measure lambda and lambda0}
	Let $\lambda$ and $\lambda_{0}$ be defined as in \eqref{Formula: W(0)} and \eqref{Formula: W(1)}, respectively, and let $D\subset\subset\Omega$. Then there exist positive constants $c$ and $C$, depending on $d$, $F_{0}$ and the Lipschitz coefficient of $u$ in $B_{r}(x^{0})$, so that for any $B_{r}(x^{0})\subset G$ with $x^{0}\in\partial\varOmega^{+}(u)$ and $r<r_{0}$,
	\begin{align}\label{Formula: W(2)}
		cr^{d-1}\leqslant\int_{B_{r}(x^{0})}\:d\lambda\leqslant Cr^{d-1}.
	\end{align}
	Moreover,
	\begin{align}\label{Formula: W(2')}
		cr^{d-1}\leqslant\int_{B_{r}(x^{0})\cap\partial\{u>0\}}\:d\lambda=\int_{B_{r}(x^{0})}\:d\lambda_{0}\leqslant Cr^{d-1}.
	\end{align}
\end{proposition}
\begin{proof}
    For any $\varepsilon>0$, introduce the function $d_{\varepsilon,B_{r}(x^{0})}=\min\{\tfrac{\operatorname{dist}(x,\mathbb{R}^{d}\setminus B_{r}(x^{0}))}{\varepsilon},1\}$ as a test function for $B_{r}(x^{0})\subset D$. Then
    \begin{align}\label{Formula: W(6)}
    	\int_{\Omega}d_{\varepsilon,B_{r}(x^{0})}\:d\lambda=\int_{\Omega}\nabla u\cdot\nabla d_{\varepsilon,B_{r}(x^{0})}+f(u)d_{\varepsilon,B_{r}(x^{0})}\:dx.
    \end{align}
    Due to the fact that $d_{\varepsilon,B_{r}(x^{0})}$ converges to $\chi_{B_{r}(x^{0})}$ as $\varepsilon\to0$, we take the limit $\varepsilon\to0$ in \eqref{Formula: W(6)} and obtain
    \begin{align}\label{Formula: W(7)}
    	\int_{B_{r}(x^{0})}\:d\lambda=\int_{\partial B_{r}(x^{0})}\nabla u\cdot\nu\:d\mathcal{H}^{d-1}+\int_{B_{r}(x^{0})}f(u)\:dx\leqslant Cr^{d-1}+CF_{0}r^{d}\leqslant Cr^{d-1},
    \end{align}
    where the constant $C$ depends on $d$, $F_{0}$ and the Lipschitz constant of $u$ in $B_{r}(x^{0})$. This gives the right hand side of \eqref{Formula: W(2)}.
    
    Let $y\in B_{r}(x^{0})$ and $G_{y}(x)$ be the Green function for Laplacian in $B_{r}(x^{0})$ with pole $y$. If $u(y)>0$, then $y$ is outside the support of the Radon measure $\lambda$, and this implies
    \begin{align}\label{Formula: W(8)}
    	\int_{B_{r}(x^{0})}G_{y}(x)\:d\lambda=-u(y)-\int_{B_{r}(x^{0})}f(u)G_{y}(x)\:dx+\int_{\partial B_{r}(x^{0})}u\partial_{-\nu}G_{y}(x)\:d\mathcal{H}^{d-1}.
    \end{align}
    With the help of Corollary \ref{Corollary: Optimal linear nondegeneracy} there exists a point $y\in\partial B_{\kappa r}$ so that $u(y)\geqslant c_{d}^{*}(\kappa)Q_{\mathrm{min}}r>0$ for $\kappa>0$ small. On the other hand, since $u$ is Lipschitz continuous, we have $u(y)=u(y)-u(x^{0})\leqslant C\kappa r$ and $u>0$ in $B_{c(\kappa)r}(y)$ for small constants $c(\kappa)$. This combined with \eqref{Formula: W(8)} gives
    \begin{align}\label{Formula: W(9)}
    	\int_{B_{r}(x^{0})}G_{y}(x)\:d\lambda\geqslant-C\kappa r-CF_{0}r^{d}+c(1-\kappa)\dashint_{\partial B_{r}(x^{0})}u\:d\mathcal{H}^{d-1}\geqslant c_{0}r,\quad(c_{0}>0),
    \end{align}
    provided that $\kappa$ and $r$ small enough. Note that we have
    \begin{align}\label{Formula: W(10)}
    	\begin{alignedat}{2}
    		\int_{B_{r}(x^{0})}G_{y}(x)\:d\lambda&=\int_{B_{r}(x^{0})\cap\{u=0\}}G_{y}(x)\:d\lambda\leqslant\sup_{B_{r}(x^{0})\cap\{u=0\}}G_{y}(x)\int_{B_{r}(x^{0})}d\lambda\leqslant C(\kappa)r^{2-d}\int_{B_{r}(x^{0})}\:d\lambda.
    	\end{alignedat}
    \end{align}
    Here we used the fact that $\operatorname{dist}(y,B_{r}(x^{0})\cap\{u=0\})\geqslant c(\kappa)r$. Combining \eqref{Formula: W(9)} and \eqref{Formula: W(10)} together gives the left hand side of \eqref{Formula: W(2)}.
    
    We now prove \eqref{Formula: W(2')}. Firstly, the right-hand side of \eqref{Formula: W(2')} is obvious since $\int_{\partial B_{r;x^{0}}^{+}(u)}\:d\lambda\leqslant\int_{B_{r}(x^{0})}\:d\lambda$. Moreover,
    \begin{align*}
    	\int_{B_{r}(x^{0})}\:d\lambda&=\int_{B_{r}(x^{0})\cap\partial\{u>0\}}\:d\lambda+\int_{B_{r}(x^{0})\cap\{u=0\}}\:d\lambda\\
    	&\leqslant\int_{B_{r}(x^{0})\cap\partial\{u>0\}}\:d\lambda+\int_{B_{r}(x^{0})\cap\{u=0\}}f(0)\:dx\\
    	&\leqslant\int_{B_{r}(x^{0})\cap\partial\{u>0\}}\:d\lambda+C_{d}F_{0}r^{d},
    \end{align*}
    which implies that
    \begin{align*}
    	\int_{ B_{r}(x^{0})\cap\partial\{u>0\}}\:d\lambda\geqslant\int_{B_{r}(x^{0})}\:d\lambda-C_{d}F_{0}r^{d}\geqslant cr^{d-1}-C_{d}F_{0}r^{d}\geqslant cr^{d-1},
    \end{align*}
    for sufficiently small $r$, which gives the left-hand side of \eqref{Formula: W(2')}.
\end{proof}
The next theorem easily follows easily from Proposition \ref{Proposition: The measure lambda and lambda0}, precisely as in \cite{AC1981}.
\begin{theorem}[Representation theorem]\label{Theorem: representation theorem}
	Let $u$ be a minimizer. Then:
	\begin{enumerate}
		\item $\mathcal{H}^{d-1}(D\cap\partial\{u>0\})<+\infty$ for every $D\subset\subset\Omega$.
		\item There exists a Borel function $q_{u}$, such that
		\begin{align*}
			\Delta u+f(u)\chi_{\{u>0\}}=q_{u}\mathcal{H}^{d-1}\mres{\partial\{u>0\}}.
		\end{align*}
	   That is, for any $\xi\in C_{0}^{\infty}(\Omega)$,
	   \begin{align*}
	   	   -\int_{\Omega}\nabla u\cdot\nabla\xi\:dx+\int_{\varOmega^{+}(u)}f(u)\xi\:dx=\int_{\partial\varOmega^{+}(u)}q_{u}\xi\:d\mathcal{H}^{d-1}.
	   \end{align*}
	   \item For any compact subset $D$ of $\Omega$, and for any $B_{r}(x^{0})\subset D$ with $x^{0}\in\partial\varOmega^{+}(u)$, there exists $0<c\leqslant C<\infty$, independent of $r$ and $x^{0}$, so that
	   \begin{align*}
	   	    0<c\leqslant q_{u}(x)\leqslant C<+\infty\qquad\text{ and }\qquad cr^{d-1}\leqslant\mathcal{H}^{d-1}(B_{r}(x^{0})\cap\partial\{u>0\})\leqslant Cr^{d-1}.
	   \end{align*}
	\end{enumerate}
\end{theorem}
As a direct corollary of Theorem \ref{Theorem: representation theorem} and \cite[Theorem 1, Section 5.11]{E2015}, we have
\begin{corollary}[Finite perimeter of $\varOmega^{+}(u)$]\label{Corollary: Finite perimeter of the positive phase}
	Let $u\in\mathcal{K}$ be a minimizer of $J$, then the set $\varOmega^{+}(u)$ has locally finite perimeter in $\Omega$.
\end{corollary}
In the rest of this section, we shall deal with the blow-up sequences
\begin{align*}
	u_{n}(x)=\frac{u(x^{0}+r_{n}x)}{r_{n}}\quad\text{  for }x^{0}\in\partial\varOmega^{+}(u).
\end{align*}
Since $|\nabla u_{n}|\leqslant C$ and $u_{n}(0)=0$, we have by a diagonal argument and the theorem of Ascoli-Arzela that there exists a non-relabel subsequence and a function $u_{0}\in W^{1,\infty}_{\mathrm{loc}}(\mathbb{R}^{d})$, so that
\begin{align}\label{Formula: W(11)}
	u_{n}\to u_{0}\quad\text{ in }\quad C_{\mathrm{loc}}^{\alpha}(\mathbb{R}^{d})\quad\text{ for some }\quad\alpha\in(0,1),
\end{align}
and
\begin{align}\label{Formula: W(12)}
	\nabla u_{n}\stackrel{*}{\rightharpoonup}\nabla u_{0}\quad\text{ in }\quad L_{\mathrm{loc}}^{\infty}(\mathbb{R}^{d}).
\end{align}
The function $u_{0}$ is usually called a \emph{blow-up limit}.
\begin{proposition}[Properties of blow-up limits]\label{Proposition: Properties of blow-up limits}
	Let $\{u_{n}\}$ be the blow-up sequence and let $u_{0}$ be given as in \eqref{Formula: W(11)} and \eqref{Formula: W(12)}, respectively. Then
	\begin{enumerate}
		\item $\partial\varOmega^{+}(u_{n})\to\partial\varOmega^{+}(u_{0})$ locally in the Hausdorff distance of $\mathbb{R}^{d}$.
		\item $\chi_{\{u_{n}>0\}}\to\chi_{\{u_{0}>0\}}$ in $L_{\mathrm{loc}}^{1}(\mathbb{R}^{d})$.
		\item $\nabla u_{n}\to\nabla u_{0}$ a.e. in $\mathbb{R}^{d}$.
		\item $u_{n}\to u_{0}$ strongly in $H_{\mathrm{loc}}^{1}(\mathbb{R}^{d})$.
		\item $u_{0}$ is a global minimizer of the functional $J_{0}(v)$ in $B_{r}$ for any $r>0$. Namely,
		\begin{align}\label{Formula: W(12')}
			J_{0}(u_{0})=\min_{v}J_{0}(v)\quad\text{ for any }v\in H^{1}(B_{r})\text{ with }v=u_{0}\text{ on }\partial B_{r},
		\end{align}
	   where $J_{0}(v):=\int_{B_{r}}|\nabla v|^{2}+Q(x^{0})\chi_{\{v>0\}}\:dx$.
	   \item $u_{0}$ satisfies the problem $\Delta u_{0}=0$ in $B_{1}^{+}(u_{0})$ and $|\nabla u_{0}|=Q(x^{0})$ on $B_{1}\cap\partial\{u_{0}>0\}$ and the non-degeneracy property: there exists some constants $0<c_{0}<1$ so that
	   \begin{align}\label{Formula: W(12'')}
	   	    c_{0}<\frac{|B_{1}\cap\{u_{0}>0\}|}{|B_{1}|}<1-c_{0}.
	   \end{align}
	\end{enumerate}
\end{proposition}
\begin{proof}
	The argument (1)-(3) are standard. We only remark that these three properties strongly rely on Corollary \ref{Corollary: Weak solution-the second requirement} in our semilinear case. For the fourth property, it suffices to show that there exists a non-relabel subsequence $\{u_{n}\}$ such that $\nabla u_{n}\to\nabla u_{0}$ strongly in $L_{\mathrm{loc}}^{2}(\mathbb{R}^{d})$, namely, for any $R>0$ the inequality
	\begin{align*}%\label{Formula: W(13)}
		\limsup_{n\to\infty}\int_{B_{R}}|\nabla u_{n}|^{2}\xi\:dx\leqslant\int_{B_{R}}|\nabla u_{0}|^{2}\xi\:dx\quad\text{ for every }\xi\in C_{0}^{\infty}(B_{R}).
	\end{align*}
    For each $n$, $u_{n}$ is a solution to the problem
    \begin{align}\label{Formula: W(14)}
    	\Delta u_{n}+r_{n}f(r_{n}u_{n})=0\quad\text{ in }\quad\varOmega^{+}(u_{n}),\qquad|\nabla u_{n}|=Q(x^{0}+r_{n}x)\quad\text{ on }\quad\partial\varOmega^{+}(u_{n}).
    \end{align}
    Since $u_{n}$ converges to $u_{0}$ locally uniformly, it follows from \eqref{Formula: W(14)} that $u_{0}$ is harmonic in $\{u_{0}>0\}$. Also, using the uniform convergence, the continuity of $u_{0}$ and its harmonicity in $\{u_{0}>0\}$ we obtain 
    \begin{align*}
    	\int_{\mathbb{R}^{d}}|\nabla u_{n}|^{2}\xi\:dx&=-\int_{\mathbb{R}^{d}}u_{n}\left(\nabla u_{n}\cdot\nabla\xi-r_{n}f(r_{n}u_{n})\xi\right)\:dx\\
    	&\to-\int_{\mathbb{R}^{d}}u_{0}\nabla u_{0}\cdot\nabla\xi\:dx=\int_{\mathbb{R}^{d}}|\nabla u_{0}|^{2}\xi\:dx\quad\text{ as }\quad n\to\infty.
    \end{align*}
    It therefore follows that $u_{n}$ converges to $u_{0}$ strongly in $W_{\mathrm{loc}}^{1,2}(\mathbb{R}^{d})$ as $n\to\infty$. 
    
    We now prove (5). Consider a function $\eta_{\varepsilon}(x):=\min\{\frac{\operatorname{dist}(x,\mathbb{R}^{d}\setminus B_{r})}{\varepsilon},1\}$, it is easy to see that $\eta_{\varepsilon}\in C_{0}^{0,1}(B_{r})$ and $0\leqslant\eta_{\varepsilon}\leqslant1$. Define $v_{n}:=v+(1-\eta_{\varepsilon})(u_{n}-u_{0})$ and it follows immediately that $v_{n}=u_{n}$ outside $B_{r}$ and $\{v_{n}>0\}\subset\{v>0\}\cap\{\eta_{\varepsilon}<1\}$. Thus,
    \begin{align*}
    	&\int_{B_{r}}|\nabla u_{n}|^{2}+F(r_{n}u_{n})+Q^{2}(x^{0}+r_{n}x)\chi_{\{u_{n}>0\}}\:dx\\
    	&\leqslant\int_{B_{r}}|\nabla v_{n}|^{2}+F(r_{n}v_{n})+Q^{2}(x^{0}+r_{n}x)\chi_{\{v_{n}>0\}}\:dx\\
    	&\leqslant\int_{B_{r}}|\nabla v_{n}|^{2}+F(r_{n}v_{n})+Q^{2}(x^{0}+r_{n}x)\chi_{\{v_{n}>0\}}\:dx+\int_{B_{r}}Q^{2}(x^{0}+r_{n}x)\chi_{\{\eta_{\varepsilon}<1\}}\:dx.
    \end{align*}
    Passing to the limit as $n\to\infty$ gives
    \begin{align}\label{Formula: W(15)}
    	\int_{B_{r}}|\nabla u_{0}|^{2}+Q^{2}(x^{0})\chi_{\{u_{0}>0\}}\:dx\leqslant\int_{B_{r}}|\nabla v|^{2}+Q^{2}(x^{0})\chi_{\{v>0\}}+Q^{2}(x^{0})\int_{B_{r}}\chi_{\{\eta_{\varepsilon}<1\}}\:dx,
    \end{align}
    where we have used the statements (2) and (3). Taking $\varepsilon\to0$ in \eqref{Formula: W(15)} gives the desired result. The fact that $\Delta u_{0}=0$ follows from the uniform convergence of $u_{n}$ to $u_{0}$ and the first equation in \eqref{Formula: W(14)}. The second equation in \eqref{Formula: W(14)} follows from the fact that $u_{0}$ is a global minimizer of $J_{0}(v)$. Finally for the last property, we point out it directly follows from (5) and \cite[Lemma 3.7]{AC1981}.
\end{proof}
\begin{remark}\label{Remark: Homo v.s. Nonhomo}
	In conclusion, we compare the classical homogeneous problems with our semilinear non-homogeneous problems as follows:
	\begin{table}[H]
		\begin{center}
			\caption{Homogeneous v.s. Nonhomogeneous}
			\label{table:1}
			\begin{tabular}{|c|c|c|}
				\hline \textbf{Problem} & \textbf{Equations for Blow-up sequences} & \textbf{Equations for Blow-ups} \\
				\hline Problem \eqref{Formula: Semilinear problem} & $\Delta u_{n}+r_{n}f(u_{n})=0$ & $\Delta u_{0}=0$ \\
				\hline Problem \eqref{Formula: Homogeneous one-phase Bernoulli} & $\Delta u_{n}=0$ & $\Delta u_{0}=0$ \\
				\hline
			\end{tabular}
		\end{center}
	\end{table}
	One of the significant differences we read from the second column of Table \ref{table:1} between the two problems is that the blow-up sequence satisfies different equations. However, in view of the last column of Table \ref{table:1}, we see that the blow-up limits of the two problems are both harmonic functions. We stress that this is the main reason we can establish the regularity of the regular set $\partial_{\mathrm{reg}}\varOmega^{+}(u)$ and determine the dimension of the singular set. We will discuss these in the coming four sections precisely.
\end{remark}
\section{Weiss-type monotonicity formulas}\label{Section: Weiss-type monotonicity formula}In this section, we establish a Weiss-type monotonicity formula, which will be a useful tool in the study of homogeneity of blow-ups $u_{0}$. The Weiss boundary-adjusted energy developed by Weiss in \cite{W1999} originally considered the problem \eqref{Formula: AC1981 problem}, who proved that blow-up limits of \eqref{Formula: AC1981 problem} are homogeneous functions of degree $1$. We borrow this idea and adapt it to our settings. However, the strategy here is slightly different, since our functional $J$ involves a non-linear term $F$. Our final results reveal that the functional $J$ still possesses monotonicity after subtracting the boundary-adjusted term. Before establishing that, we need to introduce some notation.
For every $B_{r}(x^{0})\subset\Omega$ and $u\in H^{1}(B_{r}(x^{0}))$, we define
\begin{align}\label{Formula: M(1)}
	W(u,r;Q)=\frac{1}{r^{d}}\int_{B_{r}(x^{0})}(|\nabla u|^{2}+F(u)+Q^{2}(x)\chi_{\{u>0\}})\:dx-\frac{1}{r^{d+1}}\int_{\partial B_{r}(x^{0})}u^{2}\:d\mathcal{H}^{d-1},
\end{align}
and 
\begin{align}\label{Formula: M(1')}
	\begin{alignedat}{2}
		W(u_{x^{0},r})&:=W(u_{x^{0},r},1,Q(x^{0}+rx))\\
		&=\int_{B_{1}}|\nabla u_{x^{0},r}|^{2}+F(ru_{x^{0},r})+Q^{2}(x^{0}+rx)\chi_{\{u_{x^{0},r}>0\}}\:dx-\int_{\partial B_{1}}u_{x^{0},r}^{2}\:d\mathcal{H}^{d-1}.
	\end{alignedat}
\end{align}
\begin{remark}\label{Remark: Minimality of Weiss energy}
	It is easy to verify the following fact: If $u\in H^{1}(B_{1})$ is a minimizer of $J$ for all $v\in H^{1}(B_{1})$ with $v=u$ on $\partial B_{1}$. Then $u$ is also a minimizer of $W=W(\cdot,r;Q)$ in the sense that $W(u)\leqslant W(v)$ for every $v\in H^{1}(B_{1})$ with $v=u$ on $\partial B_{1}$.
\end{remark}
\begin{remark}
	Notice that the first integral in \eqref{Formula: M(1)} is exactly our standard energy $J(u;B_{r})$, while the second integral in \eqref{Formula: M(1)} is the noted \emph{Weiss boundary-adjusted} energy. We pause here for a second to display our pattern for constructing the formula \eqref{Formula: M(1)}. Firstly, we focus on the integral involving the gradient of $u$ and the non-linear term $F$. The term $|\nabla u|^{2}+F(u)$ is invoked in the first integral because we require a minimizer to satisfy the equation. Moreover, the order of $|\nabla u|^{2}+F(u)$ is in fact $r^{d-1}$, due to Proposition \ref{Proposition: Lipschitz regularity of minimizers}. Indeed, if $u$ grows like $Cr$, then $\nabla u$ would be bounded and consequently, $\int_{B_{r}(x^{0})}|\nabla u|^{2}\:dx$ should grow like $Cr^{d}$, and so we divide $J(u)$ by $r^{d}$. Secondly, our final goal is to neutralize the order $r^{d}$ caused by the first integral using the boundary adjusted term so that $W(\cdot)$ exhibits some boundedness property. Therefore, using the fact again $u$ grows like $Cr$, it is handy that $\int_{\partial B_{r}(x^{0})}u^{2}\:d\mathcal{H}^{d-1}$ grows like $Cr^{d+1}$ and this is the reason we divide this boundary integral by $\tfrac{1}{r^{d+1}}$. The matched order between two integrals is crucial so that $W$ is preserved under the scaling $u_{x^{0},r}(x)=u(x^{0}+rx)/r$ in the sense that $W(u,r,Q)=W(u_{x^{0},r})$. This allows us to passing the limit $\lim_{r\to0^{+}}W(\cdot,\cdot,r)$ and obtains the wanted homogeneity for blow up limits.
\end{remark}
The main body in this section is to prove the following proposition.
\begin{proposition}[Weiss-type monotonicity formula]\label{Proposition: Weiss-type monotonicity formula}
	Let $B_{r}(x^{0})\subset\Omega$ with $r\in(0,\delta_{0})$, where $\delta_{0}=\operatorname{dist}(x^{0},\partial\Omega)$. Let $u\in\mathcal{K}$ be a minimizer in $B_{\delta_{0}}(x^{0})$. Then
	\begin{align}\label{Formula: M(2)}
		\frac{d}{dr}W(u,r;Q)=\frac{2}{r^{d}}\int_{\partial B_{r}(x^{0})}\left(\nabla u\cdot\nu-\frac{u}{r}\right)^{2}\:d\mathcal{H}^{d-1}+\mathcal{F}_{1}(u,r;F)+\mathcal{Q}_{1}(r;Q),
	\end{align}
	where
	\begin{align}\label{Formula: M(3)}
		\mathcal{F}_{1}(u,r;F):=\frac{1}{r^{d+1}}\int_{B_{r}(x^{0})}uF'(u)\:dx,
	\end{align}
	and
	\begin{align}\label{Formula: M(4)}
		\mathcal{Q}_{1}(r;Q)&=:-\frac{1}{r^{d+1}}\int_{B_{r}}Q^{2}(x)x\cdot d\mu(x)-\frac{d}{r^{d+1}}\int_{B_{r}}Q^{2}(x)\chi_{\{u>0\}}\:dx+\frac{1}{r^{d}}\int_{\partial B_{r}}Q^{2}(x)\chi_{\{u>0\}}\:d\mathcal{H}^{d-1}.
	\end{align}
	Here we employed the notation
	\begin{align}\label{Formula: M(4')}
		d\mu=\nabla\chi_{\{u>0\}}.
	\end{align}
\end{proposition}
\begin{remark}
	Since $\varOmega^{+}(u)$ is locally finite perimeter in $\Omega$, we have that $\chi_{\{u>0\}}$ is a function of bounded variation and thus \eqref{Formula: M(4')} is well-defined.
\end{remark}
\begin{remark}\label{Remark: Q}
	If in particular $Q(x)\equiv Q(x^{0})>0$, then $\mathcal{Q}_{r}(r;Q)=0$. This can be easily obtained by \eqref{Formula: M(4)} and the following simple fact
	\begin{align}\label{Formula: M(9)}
	   \begin{alignedat}{2}
	       \frac{1}{2}r^{d+1}\mathcal{Q}_{1}(r;Q)&=-\frac{1}{2}\int_{B_{r}(x^{0})}x\cdot d\mu(x)-\frac{d}{2}\int_{B_{r}(x^{0})}\chi_{\{u>0\}}\:dx+\frac{r}{2}\int_{\partial B_{r}(x^{0})}\chi_{\{u>0\}}\:d\mathcal{H}^{d-1}\\
	       &=0.
	   \end{alignedat}
	\end{align}
\end{remark}
To prove Proposition \ref{Proposition: Weiss-type monotonicity formula}, we need some fundamental tools, the first is a domain variation formula.
\begin{lemma}[Domain variation formula]\label{Lemma: Domain variation formula}
	Let $Q(x)\in C^{0}(\Omega)$ be any continuous function, let $u\in\mathcal{K}$ be a local minimizer of $J$ in $\Omega$ and let $\xi\in C_{0}^{\infty}(\Omega;\mathbb{R}^{d})$. Then we have
	\begin{align}\label{Formula: M(10)}
		\int_{\Omega}2\nabla uD\xi(\nabla u)^{\mathrm{T}}-(|\nabla u|^{2}+F(u))\operatorname{div}\xi\:dx+\int_{\Omega}Q^{2}(x)\xi(x)\cdot d\mu=0,
	\end{align}
	where $d\mu$ is defined in \eqref{Formula: M(4')}.
\end{lemma}
\begin{proof}
	Let $L:=|D\xi|$ be the Lipschitz constant of $\xi$. For $t>0$, we define $\Psi_{t}(x)=x+t\xi(x)$ and $u_{t}(x)=u(x+t\xi(x))$. One can show that, for $0<t<\frac{1}{L}$, $\Psi_{t}\colon\Omega\to\Omega$ is bijective, $\Psi_{t}^{-1}\in C^{\infty}(\Omega;\mathbb{R}^{d})$, and moreover $\Psi_{t}^{-1}(y)-y\in C_{0}^{\infty}(\Omega;\mathbb{R}^{d})$ with  $\operatorname{supp}(\Psi_{t}^{-1}(y)-y)\subset\operatorname{supp}(\Psi_{t})$. A direct calculation gives $D\Psi_{t}=\operatorname{Id}+tD\xi$ and $|\det D\Psi_{t}(x)|=1+t\operatorname{div}\xi+o(t)$.
    By a change of variables,
    \begin{align}\label{Formula: M(12)}
    	\begin{alignedat}{4}
    		\int_{\Omega}|\nabla u_{t}(x)|^{2}\:dx&=\int_{\Omega}|\nabla u_{t}(\Psi_{t}^{-1}(y))|^{2}|\det D\Psi_{t}^{-1}(y)|\:dy\\
    		&=\int_{\Omega}(|\nabla u(y)|^{2}+2t\nabla uD\xi(\nabla u)^{\mathrm{T}}+o(t))(1-t\operatorname{div}\xi+o(t))\:dy\\
    		&=\int_{\Omega}|\nabla u(y)|^{2}\:dy+t\int_{\Omega}2\nabla uD\xi(\nabla u)^{\mathrm{T}}\:dy-t\int_{\Omega}|\nabla u(y)|^{2}\operatorname{div}\xi\:dy+o(t).
    	\end{alignedat}
    \end{align}
    Similarly, one has
    \begin{align}\label{Formula: M(13)}
    	\begin{alignedat}{3}
    		\int_{\Omega}F(u_{t}(x))\:dx&=\int_{\Omega}F(u_{t}(\Psi_{t}^{-1}(y)))|\det D\Psi_{t}^{-1}(y)|\:dy\\
    		&=\int_{\Omega}F(u(y))(1-t\operatorname{div}\xi+o(t))\:dy\\
    		&=\int_{\Omega}F(u(y))\:dy-t\int_{\Omega}F(u(y))\operatorname{div}\xi\:dy+o(t).
    	\end{alignedat}    	
    \end{align}
    By differentiating \eqref{Formula: M(12)} and \eqref{Formula: M(13)} with respect to $t$ we obtain
    \begin{align}\label{Formula: M(13')}
        \left.\frac{d}{dt}\left(\int_{\Omega}|\nabla u_{t}|^{2}+F(u_{t})\:dx\right)\right|_{t=0}=\int_{\Omega}2\nabla uD\xi(\nabla u)^{\mathrm{T}}-(|\nabla u|^{2}+F(u))\operatorname{div}\xi\:dy.
    \end{align}
    We also differentiate the last term in our functional
    \begin{align}\label{Formula: M(14)}
       \begin{alignedat}{2}
            \frac{d}{dt}\left.\left(\int_{\Omega}Q^{2}(x)\chi_{\{u_{t}>0\}}(x)\:dx\right)\right|_{t=0}&=\frac{d}{dt}\int_{\Omega}Q^{2}(x)\chi_{\{u>0\}}(\Psi_{t}(x))\:dx=\int_{\Omega}Q^{2}(x)\xi\cdot d\mu(x).
       \end{alignedat}
    \end{align}
    Combining \eqref{Formula: M(13')}, \eqref{Formula: M(14)} and the fact that $u$ is a minimizer, we have
    \begin{align*}
    	0=\left.\frac{d}{dt}\right|_{t=0}J(u_{t})=\int_{\Omega}2\nabla uD\xi(\nabla u)^{\mathrm{T}}-(|\nabla u|^{2}+F(u))\operatorname{div}\xi\:dx+\int_{\Omega}Q^{2}(x)\xi\cdot d\mu.
    \end{align*}
\end{proof}
Next we introduce an energy identity comes from the equation satisfied by minimizers.
\begin{lemma}[Energy identity]\label{Lemma: Energy identity}
	Let $u\in\mathcal{K}$ be a local minimizer of $J$ in $\Omega$, then for a.e. $r\in(0,\delta_{0})$ with $\delta_{0}:=\operatorname{dist}(x^{0},\partial\Omega)$,
	\begin{align}\label{Formula: M(15)}
		\int_{B_{r}(x^{0})}\left[|\nabla u|^{2}+\frac{uF'(u)}{2}\right]\:dx=\int_{\partial B_{r}(x^{0})}u\nabla u\cdot\nu\:d\mathcal{H}^{d-1}.
	\end{align}
\end{lemma}
\begin{proof}
	Without loss of generality, we might assume that $x^{0}=0$. For any $0<\varepsilon<r$, we define $\phi_{\varepsilon,r}\in C_{0}^{\infty}(B_{r})$ such that $\phi_{\varepsilon,r}(x)=1$ in $B_{r-\varepsilon}$ and $\nabla\phi_{\varepsilon,r}(x)=-\frac{1}{\varepsilon}\frac{x}{|x|}\chi_{\{B_{r}\setminus B_{r-\varepsilon}\}}$. It is easy to see that $u\phi_{\varepsilon,r}\in H_{0}^{1}(B_{r}^{+}(u))$. Inserting $u\phi_{\varepsilon,r}$ as a test function in \eqref{Formula: P(3)} yields
	\begin{align*}
		-\int_{B_{r}^{+}(u)}\frac{F'(u)}{2}(\phi_{\varepsilon,r}u)\:dx&=\int_{B_{r}^{+}(u)}\nabla u\cdot\nabla(\phi_{\varepsilon,r}u)\:dx\\
		&=\int_{B_{r}}|\nabla u|^{2}\phi_{\varepsilon,r}\:dx-\frac{1}{\varepsilon}\int_{B_{r}\setminus\overline{B_{r-\varepsilon}}}\nabla u\cdot\frac{x}{|x|}u\:dx.
	\end{align*}
	Letting $\varepsilon\to0$, we obtain that for a.e. $r\in(0,\delta_{0})$,
	\begin{align*}
		\int_{B_{r}}\left[|\nabla u|^{2}+\frac{uF'(u)}{2}\right]\:dx=\int_{\partial B_{r}}u\nabla u\cdot\nu\:d\mathcal{H}^{d-1}.
	\end{align*}
\end{proof}
In \cite{CSY2018}, Caffarelli and et. al. developed a Poho\v{z}aev-type identity for vectorial Bernoulli problem, we borrow their idea and adapt it to our settings, which will be applied to prove the monotonicity formula.
\begin{lemma}[Poho\v{z}aev-type identity]\label{Lemma: Pohozaev-type identity}
	Let $u\in\mathcal{K}$ be a local minimizer of $J$ in $\Omega$, let $B_{r}(x^{0})\subset\Omega$ with $r\in(0,\delta_{0})$ where $\delta_{0}:=\operatorname{dist}(x^{0},\partial\Omega)$. Assume that $Q(x)$ is a continuous function, then we have the following identity
	\begin{align}\label{Formula: M(16)}
		\begin{alignedat}{2}
			d\int_{B_{r}(x^{0})}(|\nabla u|^{2}+F(u))\:dx&=2\int_{B_{r}(x^{0})}|\nabla u|^{2}\:dx+r\int_{\partial B_{r}(x^{0})}(|\nabla u|^{2}-2(\nabla u\cdot\nu)^{2})\:d\mathcal{H}^{d-1}\\
			&\quad+r\int_{\partial B_{r}(x^{0})}F(u)\:d\mathcal{H}^{d-1}+\int_{B_{r}(x^{0})}Q^{2}(x)(x-x^{0})\cdot d\mu(x),
		\end{alignedat}
	\end{align}
	where the vector Radon measure $\mu$ is defined in \eqref{Formula: M(4')}.
\end{lemma}
\begin{proof}
	Assume $x^{0}=0$ and define $\xi(x):=x\phi_{\varepsilon,r}$ where $\phi_{\varepsilon,r}$ is the test function given in Lemma \ref{Lemma: Energy identity}. Introducing $\xi(x)$ as a test function into the domain variation formula \eqref{Formula: M(10)} gives
	\begin{align}\label{Formula: M(17)}
		\begin{alignedat}{3}
			&d\int_{\Omega}(|\nabla u|^{2}+F(u))\phi_{\varepsilon,r}\:dx\\
			&\quad=\int_{\Omega}2|\nabla u|^{2}\phi_{\varepsilon,r}+2(\nabla u\cdot x)(\nabla\phi_{\varepsilon,r}\cdot\nabla u)\:dx-\int_{\Omega}(|\nabla u|^{2}+F(u))(x\cdot\nabla\phi_{\varepsilon,r}(x))\:dx\\
			&\qquad+\int_{\Omega}Q^{2}(x)(x\phi_{\varepsilon,r})\cdot\:d\mu(x).
		\end{alignedat}
	\end{align}
    Passing to the limit as $\varepsilon\to0$ in \eqref{Formula: M(17)} gives the desired result.
\end{proof}
We next show that the formula \eqref{Formula: M(2)} is a direct application of Lemma \ref{Lemma: Energy identity} and Lemma \ref{Lemma: Pohozaev-type identity}.
\begin{proof}[Proof of the formula \eqref{Formula: M(2)}]
	Assume $x^{0}=0$, then for a.e. $r\in(0,\delta_{0})$, a straightforward calculation yields
	\begin{align}\label{Formula: M(17')}
	    \frac{d}{dr}\left(r^{-d-1}\int_{\partial B_{r}}u^{2}\:d\mathcal{H}^{d-1}\right)=-2r^{-d-2}\int_{\partial B_{r}}u^{2}\:d\mathcal{H}^{d-1}+2r^{-d-1}\int_{\partial B_{r}}u\nabla u\cdot\nu\:d\mathcal{H}^{d-1}.
	\end{align}
	We further compute
	\begin{align*}
		\frac{1}{2}r^{d+1}\frac{d}{dr}W(u,r;Q)&=-\frac{d}{2}\int_{B_{r}}|\nabla u|^{2}+F(u)\:dx-\frac{d}{2}\int_{B_{r}}Q^{2}(x)\chi_{\{u>0\}}\:dx\\
		&\quad+\frac{r}{2}\int_{\partial B_{r}}|\nabla u|^{2}+F(u)+Q^{2}(x)\chi_{\{u>0\}}\:d\mathcal{H}^{d-1}\\
		&\quad+\frac{1}{r}\int_{\partial B_{r}}u^{2}\:d\mathcal{H}^{d-1}-\int_{\partial B_{r}}u\nabla u\cdot\nu\:d\mathcal{H}^{d-1}.
	\end{align*}
    Using the identity \eqref{Formula: M(16)} for the first integral on the right-hand side, we get
    \begin{align*}
    	\frac{1}{2}r^{d+1}\frac{d}{dr}W(u,r;Q)&=-\frac{1}{2}\left(2\int_{B_{r}}|\nabla u|^{2}+r\int_{\partial B_{r}}|\nabla u|^{2}-2(\nabla u\cdot\nu)^{2}\:d\mathcal{H}^{d-1}\right.\\
    	&\qquad\quad\left.+r\int_{\partial B_{r}}F(u)\:d\mathcal{H}^{d-1}+\int_{B_{r}}Q^{2}(x)x\cdot d\mu(x)\right)\\
    	&\qquad\quad-\frac{d}{2}\int_{B_{r}}Q^{2}(x)\chi_{\{u>0\}}\:dx+\frac{r}{2}\int_{\partial B_{r}}(|\nabla u|^{2}+F(u))\:d\mathcal{H}^{d-1}\\
    	&\qquad\quad+\frac{r}{2}\int_{\partial B_{r}}Q^{2}(x)\chi_{\{u>0\}}\:d\mathcal{H}^{d-1}+\frac{1}{r}\int_{\partial B_{r}}u^{2}\:d\mathcal{H}^{d-1}\\
    	&\qquad\quad-\int_{\partial B_{r}}u\nabla u\cdot\nu\:d\mathcal{H}^{d-1}.
    \end{align*}
    Tidying the results, we obtain
    \begin{align}\label{Formula: M(18)}
    	\begin{alignedat}{3}
    		\frac{1}{2}r^{d+1}\frac{d}{dr}W(u,r;Q)&=-\int_{B_{r}}|\nabla u|^{2}\:dx+r\int_{\partial B_{r}}(\nabla u\cdot\nu)^{2}\:d\mathcal{H}^{d-1}+\frac{1}{r}\int_{\partial B_{r}}u^{2}\:d\mathcal{H}^{d-1}\\
    		&\qquad-\int_{\partial B_{r}}u\nabla u\cdot\nu\:d\mathcal{H}^{d-1}-\frac{1}{2}\int_{B_{r}}Q^{2}(x)x\cdot d\mu(x)\\
    		&\qquad-\frac{d}{2}\int_{B_{r}}Q^{2}(x)\chi_{\{u>0\}}\:dx+\frac{r}{2}\int_{\partial B_{r}}Q^{2}(x)\chi_{\{u>0\}}\:d\mathcal{H}^{d-1},
    	\end{alignedat}
    \end{align}
    where $\mu$ is the vector Radon measure defined in \eqref{Formula: M(4')}. It follows from Lemma \ref{Lemma: Energy identity} and applying \eqref{Formula: M(15)} for the first integral on the right-hand side of \eqref{Formula: M(18)}, we obtain for a.e. $r\in(0,\delta_{0})$,
    \begin{align}\label{Formula: M(19)}
    	\begin{alignedat}{4}
    		\frac{1}{2}r^{d+1}\frac{d}{dr}W(u,r;Q)&=r\int_{\partial B_{r}}(\nabla u\cdot\nu)^{2}\:d\mathcal{H}^{d-1}+\frac{1}{r}\int_{\partial B_{r}}u^{2}\:d\mathcal{H}^{d-1}\\
    		&\qquad-2\int_{\partial B_{r}}u\nabla u\cdot\nu\:d\mathcal{H}^{d-1}+\int_{B_{r}}\frac{uF'(u)}{2}\:dx\\
    		&\qquad-\frac{1}{2}\int_{B_{r}}Q^{2}(x)x\cdot d\mu(x)-\frac{d}{2}\int_{B_{r}}Q^{2}(x)\chi_{\{u>0\}}\:dx\\
    		&\qquad+\frac{r}{2}\int_{\partial B_{r}}Q^{2}(x)\chi_{\{u>0\}}\:d\mathcal{H}^{d-1}.
    	\end{alignedat}
    \end{align}
    Notice that
    \begin{align}\label{Formula: M(20)}
    	r\int_{\partial B_{r}}(\nabla u\cdot\nu)^{2}\:d\mathcal{H}^{d-1}-2\int_{\partial B_{r}}u\nabla u\cdot\nu\:d\mathcal{H}^{d-1}+\frac{1}{r}\int_{\partial B_{r}}u^{2}\:d\mathcal{H}^{d-1}=r\int_{\partial B_{r}}\left(\nabla u\cdot\nu-\frac{u}{r}\right)^{2}\:d\mathcal{H}^{d-1}.
    \end{align}
    The desired result follows from \eqref{Formula: M(19)} and \eqref{Formula: M(20)}.
\end{proof}
We next introduce another monotonicity formula, which is different from \eqref{Formula: M(2)} and is obtained by comparing $u_{x^{0},r}$ with its one-homogeneous extensions in $B_{1}$.
\begin{proposition}\label{Proposition: second monotonicity formula}
	Let $B_{r}(x^{0})\subset\Omega$ with $r\in(0,\delta_{0})$, where $\delta_{0}=\operatorname{dist}(x^{0},\partial\Omega)$. Let $u\in\mathcal{K}$ be a local minimizer in $B_{\delta_{0}}(x^{0})$ and let $z_{x^{0},r}\colon B_{1}\to\mathbb{R}$ be the one-homogeneous extension of $u_{x^{0},r}$ defined by
	\begin{align}\label{Formula: M(21)}
		z_{x^{0},r}:=|x|u_{x^{0},r}\left(\frac{x}{|x|}\right).
	\end{align}
    Then
    \begin{align}\label{Formula: M(22)}
    	\begin{alignedat}{2}
    		\frac{d}{dr}W(u_{x^{0},r})&=\frac{d}{r}(W(z_{x^{0},r})-W(u_{x^{0},r}))+\frac{1}{r}\int_{\partial B_{1}}|x\cdot\nabla u_{x^{0},r}-u_{x^{0},r}|^{2}\:d\mathcal{H}^{d-1}\\
    		&\ \ +\mathcal{F}_{2}(u_{x^{0},r};F)+\cQ_{2}(u_{x^{0},r};Q),
    	\end{alignedat}
    \end{align}
    where
    \begin{align}\label{Formula: M(23)}
    	\mathcal{F}_{2}(u_{x^{0},r};F)=\frac{1}{r}\int_{\partial B_{1}}F(ru_{x^{0},r})\:d\mathcal{H}^{d-1}-\frac{d}{r}\int_{B_{1}}F(rz_{x^{0},r})\:dx,
    \end{align}
    and
    \begin{align}\label{Formula: M(24)}
    	\mathcal{Q}_{2}(u_{x^{0},r};Q)=\frac{1}{r}\int_{\partial B_{1}}Q^{2}(x^{0}+rx)\chi_{\{ u_{x^{0},r}>0\}}\:d\mathcal{H}^{d-1}-\frac{d}{r}\int_{B_{1}}Q^{2}(x^{0}+rx)\chi_{\{ z_{r}>0\}}\:dx.
    \end{align}
\end{proposition}
\begin{remark}\label{Remark: Reifenberg}
	In \cite{R1964}, Reifenberg exhibited the idea that a purely variational inequality, known as an epiperimetric inequality, will encode the local behavior of the free boundary. Since then, extensive literature has been dedicated to prove different kinds of epiperimetric inequalities. (We refer to \cite{SV2019} for epiperimetric inequality for problem \eqref{Formula: AC1981 problem}). The  original idea of Reifenberg is: if $\mathcal{E}\colon H^{1}(B_{1})\to\mathbb{R}$ is a given energy which satisfies the following properties:
	\begin{enumerate}
		\item $\mathcal{E}(u_{r})\leqslant\mathcal{E}(v_{r})$ for every $v_{r}\in H^{1}(B_{1})$ with $v_{r}=u_{r}$ on $\partial B_{1}$;
		\item $\frac{d}{dr}\mathcal{E}(u_{r})\geqslant\frac{C_{1}}{r}\int_{\partial B_{1}}|x\cdot\nabla u_{r}-u_{r}|^{2}\:d\mathcal{H}^{d-1}$ for every $r\in (0,1)$ and for some positive constants $C_{1}>0$;
		\item $\frac{d}{dr}\mathcal{E}(u_{r})\geqslant\frac{C_{2}}{r}(\mathcal{E}(z_{r})-\mathcal{E}(u_{r}))$ for every $r\in(0,1)$ and for some positive constants $C_{2}>0$. Here $z_{r}$ is the one-homogeneous extension of $u_{r}$ defined in \eqref{Formula: M(21)}.
		\item $\mathcal{E}(h_{r})\leqslant(1-\varepsilon)\mathcal{E}(u_{r})$ for all $h_{r}=u_{r}=z_{r}$ on $\partial B_{1}$.
	\end{enumerate}
	Then the uniqueness of blow-up limit is a direct corollary of (1)-(4). As an application of the uniqueness of blow-up limit, one can determine the regularity of the regular set  $\partial_{\mathrm{reg}}\varOmega^{+}(u)$, even the regularity of the  singular set $\partial_{\mathrm{sing}}\varOmega^{+}(u)$. In our context, the energy to which will apply (1)-(4) is 
	\begin{align*}
		\mathcal{E}(u)-\frac{\omega_{d}}{2}=\left(\int_{B_{1}}|\nabla u|^{2}+F(u)+\chi_{\{u>0\}}\:dx-\int_{\partial B_{1}}u^{2}\:d\mathcal{H}^{d-1}\right)-\frac{\omega_{d}}{2}.
	\end{align*}
	This is the core reason we put forward the formula \eqref{Formula: M(22)}. Observe that (1)-(3) is a direct result of the formula \eqref{Formula: M(22)} (Please see Remark \ref{Remark: (2) and (3)} below). However, (4) is unknown and remains to be an open problem. In this paper, we will use an approach to show the uniqueness of blow-up limit (see Section ) heavily dependent on the geometry (flatness) of the free boundary. However, we believe that the regularity of problem \eqref{Formula: Semilinear problem} can also be obtained following Reifenberg's idea and Proposition \ref{Proposition: second monotonicity formula} will help us to prove (2) and (3).
\end{remark}
\begin{proof}
	Assume $x^{0}=0$. It follows from the definition of $z_{r}$ that for each $r\in(0,\delta_{0})$ one can write $z_{r}\colon B_{1}\to\mathbb{R}$ in polar coordinates $\rho\in(0,1]$, $\theta\in\mathbb{S}^{d-1}$ as $z_{r}(\rho,\theta)=\rho z_{r}(1,\theta)$. Therefore, we have the following identity.
    \begin{align*}
    	\int_{B_{1}}|\nabla z_{r}|^{2}\:dx-\int_{\partial B_{1}}z_{r}^{2}d\mathcal{H}^{d-1}&=\int_{0}^{1}\rho^{d-1}\:d\rho\int_{\mathbb{S}^{d-1}}(z_{r}^{2}(1,\theta)+|\nabla_{\theta}z_{r}(1,\theta)|^{2})\:d\theta-\int_{\mathbb{S}^{d-1}}z_{r}^{2}\:d\theta\\
    	&=\frac{1}{d}\int_{\mathbb{S}^{d-1}}|\nabla_{\theta}z_{r}(1,\theta)|^{2}\:d\theta-\frac{d-1}{d}\int_{\mathbb{S}^{d-1}}z_{r}^{2}(1,\theta)\:d\theta\\
    	&=\frac{1}{d}\int_{\partial B_{1}}(|\nabla u_{r}|^{2}-(x\cdot\nabla u_{r})^{2})\:d\mathcal{H}^{d-1}-\frac{d-1}{d}\int_{\partial B_{1}}u_{r}^{2}\:d\mathcal{H}^{d-1},
    \end{align*}
    where we have used $|\nabla u|^{2}=(\partial_{\rho}u)^{2}+\rho^{-2}|\nabla_{\theta}u|^{2}$ if $u=u(\rho,\theta)$. Thus, let $\mathcal{F}_{2}(u_{r};F)$ and $\mathcal{Q}_{2}(u_{r};Q)$ be defined in \eqref{Formula: M(23)} and \eqref{Formula: M(24)}, and it follows from \eqref{Formula: M(1')} that
    \begin{align*}
    	&\frac{d}{r}(W(z_{r})-W(u_{r}))+\frac{1}{r}\int_{\partial B_{1}}|x\cdot \nabla u_{r}-u_{r}|^{2}\:d\mathcal{H}^{d-1}+\mathcal{F}_{2}(u_{r};F)+\mathcal{Q}_{2}(u_{r};F)\\
    	&=-\frac{d}{r}\int_{B_{1}}(|\nabla u_{r}|^{2}+F(ru_{r})+Q^{2}(rx)\chi_{\{u_{r}>0\}})\:dx\\
    	&\quad+\frac{1}{r}\int_{\partial B_{1}}(|\nabla u_{r}|^{2}+F(ru_{r})+Q^{2}(rx)\chi_{\{u_{r}>0\}})\:d\mathcal{H}^{d-1}\\
    	&\quad+\frac{2}{r}\int_{\partial B_{1}}u_{r}^{2}\:d\mathcal{H}^{d-1}-\frac{2}{r}\int_{\partial B_{1}}u_{r}(\nabla u_{r}\cdot x)\:d\mathcal{H}^{d-1}.
    \end{align*}
    On the other hand, by a direct calculation, we have
    \begin{align*}
    	&\frac{d}{dr}\left(r^{-d}\int_{B_{r}}|\nabla u|^{2}+F(u)+Q^{2}(x)\chi_{\{ u>0\}}\right)\\
    	&=-\frac{d}{r}\int_{B_{1}}(|\nabla u_{r}|^{2}+F(ru_{r})+Q^{2}(rx)\chi_{\{u_{r}>0\}})\:dx\\
    	&\quad+\frac{1}{r}\int_{\partial B_{1}}(|\nabla u_{r}|^{2}+F(ru_{r})+Q^{2}(rx)\chi_{\{u_{r}>0\}})\:d\mathcal{H}^{d-1}.
    \end{align*}
    This proves the result by employing the equality \eqref{Formula: M(17')}.
\end{proof}
In the end of this section, we collect some properties of the functions $\mathcal{F}_{i}(u,r;F)$ and $\mathcal{Q}_{i}(u,r;Q)$ for $i=1$, $2$.
\begin{proposition}\label{Proposition: properties of F and Q}
	Let $\mathcal{F}_{i}(u,r;F)$ and $\mathcal{Q}_{i}(u,r;Q)$ for $i=1$, $2$  be defined as in Proposition \ref{Proposition: Weiss-type monotonicity formula} and Proposition \ref{Proposition: second monotonicity formula}, respectively. Then the following does hold:
	\begin{enumerate}
		\item If $u\in\mathcal{K}$ is a minimizer of $J$ in $B_{r}(x^{0})$ and let $F(t)$ satisfies \eqref{Formula: Assumption on F(t)} with $F'(0)=0$, then $\mathcal{F}_{i}(u,r;F)\geqslant0$ for a.e. $r\in(0,\delta_{0})$.
		\item If $u\in\mathcal{K}$ is a minimizer of $J$ in $B_{r}(x^{0})$ and let $F(t)$ satisfies \eqref{Formula: Assumption on F(t)}, then
		\begin{align}\label{Formula: M(26)}
			\mathcal{F}_{i}(u,r;F)\geqslant-C\quad\text{ for }\quad i=1,2,
		\end{align}
	   where $C$ is a constant depending on $d$, $M_{0}$ and $F_{0}$ but independent of $r$ and $x^{0}$.
	   \item If $Q(x)\in C^{0}(B_{r}(x^{0}))$, then
	   \begin{align}\label{Formula: M(27)}
	   	    \mathcal{Q}_{i}(r;Q)\geqslant-C\frac{Q_{\mathrm{max}}}{r}\operatorname{osc}_{B_{r}(x^{0})}Q\quad\text{ for a.e. }\quad r\in(0,\delta_{0}).
	   \end{align}
      Here $C$ is a constant depending only on the continuity norm of $Q$ and $\operatorname{osc}_{B_{r}(x^{0})}(Q):=\max_{B_{r}}Q-\min_{B_{r}(x^{0})}Q$.
      \item If $Q(x)\in C^{\beta}(B_{r}(x^{0}))$, then 
      \begin{align}\label{Formula: M(28)}
      	   \mathcal{Q}_{i}(r;Q)\geqslant-C\frac{Q_{\mathrm{max}}}{r^{1-\beta}}[Q]_{\beta;B_{r}(x^{0})}\text{ for a.e. }\quad r\in(0,\delta_{0}),
      \end{align}
      Here $C$ is a constant depending only on $[Q]_{\beta;B_{r}(x^{0})}:=\sup_{x\neq y}\frac{|Q(x)-Q(y)|}{|x-y|^{\beta}}$.
	\end{enumerate}
\end{proposition}
\begin{proof}
	Assume $x^{0}=0$. (1). In view of \eqref{Formula: Assumption on F(t)}, $F'(t)$ is an increasing function for $t\in\mathbb{R}$, therefore $F'(0)=0$ implies that $F'(t)\geqslant F'(0)=0$ for all $t\geqslant0$. It follows from Proposition \ref{Proposition: Properties of minimizers} that $u\geqslant0$ in $B_{r}(x^{0})$ and the definition of $\mathcal{F}_{1}(u,r;F)$ in \eqref{Formula: M(3)} that $\mathcal{F}_{1}(u,r;F)\geqslant0$. In order to prove $\mathcal{F}_{2}(u,r;F)\geqslant0$, it suffices to prove that
	\begin{align}\label{Formula: M(29)}
		\frac{d}{r}\int_{B_{1}}F(rz_{r})\:dx=\frac{d}{r}\int_{0}^{1}\rho^{d-1}\:d\rho\int_{\mathbb{S}^{d-1}}F(r\rho z_{r}(1,\theta))\:d\theta\leqslant\frac{1}{r}\int_{\partial B_{1}}F(ru_{r})\:d\mathcal{H}^{d-1}.
	\end{align}
    In fact, it follows from $F'(t)\geqslant0$ for all $t\geqslant0$ that
    \begin{align*}
    	\int_{0}^{1}\rho^{d-1}\:d\rho\int_{\mathbb{S}^{d-1}}F(r\rho z_{r}(1,\theta))\:d\theta&\leqslant\int_{0}^{1}\rho^{d-1}\:d\rho\int_{\mathbb{S}^{d-1}}F(rz_{r}(1,\theta))\:d\theta\\
    	&=\frac{1}{d}\int_{\mathbb{S}^{d-1}}F(rz_{r}(1,\theta))\:d\theta=\frac{1}{d}\int_{\partial B_{1}}F(ru_{r})\:d\mathcal{H}^{d-1},
    \end{align*}
    where we have used the fact that $0\leqslant r\rho z_{1}(1,\theta)\leqslant rz_{r}(1,\theta)$ for $\rho\in(0,1]$ in the first inequality and $u_{r}=z_{r}\in H^{1}(\partial B_{1})$ in the last inequality. This gives \eqref{Formula: M(29)}.
    
    (2). Since $u$ is a minimizer, we have from Proposition \ref{Proposition: Lipschitz regularity of minimizers} that $|\nabla u|\leqslant C$ for some constant depending on $d$ and $F_{0}$. This implies that $|u(x)|\leqslant C|x-x^{0}|\leqslant Cr$ in $B_{r}(x^{0})$. On the other hand, in via of \eqref{Formula: Assumption on F(t)} and a direct calculation, we obtain
    \begin{align}\label{Formula: M(30)}
    	\max_{u\geqslant0}|F'(u)|\leqslant2(|f(0)|+\max_{t\geqslant0}|f'(t)||u|)\leqslant2\left(\frac{F_{0}}{2}+\frac{F_{0}}{2}M_{0}\right)\leqslant C(F_{0},M_{0}).
    \end{align}
    Therefore, it follows from \eqref{Formula: M(3)} and \eqref{Formula: M(30)} that
    \begin{align}\label{Formula: M(25')}
    	|\mathcal{F}_{1}(u,r;F)|\leqslant\frac{1}{r^{d+1}}\int_{B_{r}}|uF'(u)|\:dx\leqslant\frac{C}{r^{d}}\max_{u\geqslant0}|F'(u)|\int_{B_{r}}\:dx\leqslant C(d,F_{0},M_{0}).
    \end{align}
    This proves \eqref{Formula: M(26)} for $i=1$. For $\mathcal{F}_{2}(u_{r};F)$, we compute
    \begin{align*}
    	\mathcal{F}_{2}(u_{r};F)&=\frac{1}{r}\int_{\partial B_{1}}F(0)+F'(r\xi_{r})ru_{r}\:d\mathcal{H}^{d-1}-\frac{d}{r}\int_{0}^{1}\rho^{d-1}\:d\rho\int_{\mathbb{S}^{d-1}}F(0)+F'(r\zeta_{r})r\rho z_{r}(1,\theta)\:d\theta\\
    	&=\int_{\partial B_{1}}F'(r\xi_{r})u_{r}\:d\mathcal{H}^{d-1}-d\int_{0}^{1}\rho^{d}\:d\rho\int_{\mathbb{S}^{d-1}}F'(r\zeta_{r})z_{r}(1,\theta)\:d\theta\\
    	&=\int_{\partial B_{1}}F'(r\xi_{r})u_{r}\:d\mathcal{H}^{d-1}-\frac{d}{d+1}\int_{\partial B_{1}}F'(r\zeta_{r})u_{r}\:d\mathcal{H}^{d-1}\\
    	&=\frac{1}{r^{d}}\int_{\partial B_{r}}F'(\xi)u\:d\mathcal{H}^{d-1}-\frac{d}{d+1}\frac{1}{r^{d}}\int_{\partial B_{r}}F'(\zeta)u\:d\mathcal{H}^{d-1}
    \end{align*}
    Then it follows from \eqref{Formula: M(30)} and a similar calculation as in \eqref{Formula: M(25')} that \eqref{Formula: M(26)} holds for $i=2$.
    
    (3). Recalling \eqref{Formula: M(4)} and compute 
    \begin{align}\label{Formula: M(31)}
    	\begin{alignedat}{5}
    		\mathcal{Q}_{1}(r;Q)&=-\frac{1}{r^{d+1}}\int_{B_{r}}(Q^{2}(x)-Q^{2}(0))x\cdot d\mu(x)-\frac{d}{r^{d+1}}\int_{B_{r}}(Q^{2}(x)-Q^{2}(0))\chi_{\{u>0\}}\:dx\\
    		&\quad+\frac{1}{r^{d}}\int_{\partial B_{r}}(Q^{2}(x)-Q^{2}(0))\chi_{\{u>0\}}\:d\mathcal{H}^{d-1}\\
    		&\geqslant-C\frac{Q_{\mathrm{max}}\operatorname{osc}_{B_{r}}Q}{r^{d}}\int_{B_{r}}d|\mu|(x)-C\frac{dQ_{\mathrm{max}}\operatorname{osc}_{B_{r}}Q}{r^{d+1}}\int_{B_{r}}\chi_{\{u>0\}}\:dx\\
    		&\quad-C\frac{Q_{\mathrm{max}}\operatorname{osc}_{B_{r}}Q}{r^{d}}\int_{\partial B_{r}}\chi_{\{u>0\}}\:d\mathcal{H}^{d-1}\\
    		&\geqslant-C\frac{Q_{\mathrm{max}}}{r}\operatorname{osc}_{B_{r}}Q,
    	\end{alignedat}
    \end{align}
    where we have used Remark \ref{Remark: Q} in the first identity. As for $\mathcal{Q}_{2}$, we first notice that if $Q$ is a constant, then $\mathcal{Q}_{2}\equiv0$, indeed,
    \begin{align*}
    	\int_{B_{1}}\chi_{\{ z_{r}>0\}}=\int_{0}^{1}s^{d-1}ds\int_{\mathbb{S}^{d-1}}\chi_{\{ u_{r}>0\}}d\mathcal{H}^{d-1}=\frac{1}{d}\int_{\partial B_{1}}\chi_{\{ u_{r}>0\}}d\mathcal{H}^{d-1}.
    \end{align*}
    Therefore, we can rewrite $\mathcal{Q}_{2}$ in the following form
    \begin{align*}
    	\mathcal{Q}_{2}(u_{r};Q)=\frac{1}{r}\int_{\partial B_{1}}\chi_{\{ u_{r}>0\}}(Q^{2}(rx)-Q^{2}(0))\:d\mathcal{H}^{d-1}-\frac{d}{r}\int_{B_{1}}\chi_{\{ z_{r}>0\}}(Q^{2}(rx)-Q^{2}(0))\:dx.
    \end{align*}
    Then we can follow a similar argument as in \eqref{Formula: M(31)} to conclude the proof.
    
    (4). The desired result \eqref{Formula: M(28)} follows by a similar calculation as in (3), so we omit it.
\end{proof}
\begin{remark}\label{Remark: (2) and (3)}
	If in particular $f(0)=0$ and $Q\equiv\operatorname{const}$, it follows from \eqref{Formula: M(22)}, \eqref{Formula: M(23)} and the first property in Proposition \ref{Proposition: properties of F and Q} that 
	\begin{align*}
		\frac{d}{dr}W(u_{x^{0},r})\geqslant\frac{1}{r}\int_{\partial B_{1}}|x\cdot\nabla u_{x^{0},r}-u_{x^{0},r}|^{2}\:d\mathcal{H}^{d-1}\quad\text{ for a.e. }\quad r\in(0,\delta_{0}),
	\end{align*}
	and
	\begin{align*}
		\frac{d}{dr}W(u_{x^{0},r})\geqslant\frac{d}{r}(W(z_{x^{0},r})-W(u_{x^{0},r}))\quad\text{ for a.e. }\quad r\in(0,\delta_{0}).
	\end{align*}
	Notice that these two inequalities were exactly the second and the third property we have mentioned in Remark \ref{Remark: Reifenberg}. 
\end{remark}
\section{Structure of the free boundary}
This section dedicates to prove Theorem \ref{Theorem: Main theorem1} using those technical tools obtained in the previous sections. The first one is a direct application of the Weiss-type monotonicity formula we obtained in Proposition \ref{Proposition: Weiss-type monotonicity formula} and Proposition \ref{Proposition: second monotonicity formula}, which tells us that blow-up limits of minimizers are homogeneous functions of degree one.
\begin{proposition}[Homogeneity of blow-up limits]\label{Proposition: Homogeneity of blow-up limits}
	Let $u$ be a minimizer of $J$ in $\Omega$ with $0\in\Omega$ and $u(0)=0$, let $F(t)$ satisfy \eqref{Formula: Assumption on F(t)}. Assume that $Q$ is Dini-continuous at the origin, namely,
	\begin{align}\label{Formula: S(0)}
		\lim_{r\to0^{+}}\int_{r}\frac{1}{r}\operatorname{osc}_{B_{r}}Q\:dx<\infty.
	\end{align}
	Let $B_{r_{n}}\subset\Omega$ with $r_{n}\to0$ as $n\to\infty$. Let $u_{0}$ be a blow-up limit of $u$ at the origin with respect to the sequence $r_{n}$. Then
	\begin{enumerate}
		\item The limit $\lim_{r\to0^{+}}W(u,r,Q)$ exists and is finite for a.e. $r\in(0,\delta_{0})$;
		\item $u_{0}$ is a homogeneous function of degree one.
	\end{enumerate}
\end{proposition}
\begin{proof}
	Define $\rho(r):=\int_{0}^{r}\frac{1}{s}\operatorname{osc}_{B_{s}}Q\:ds$. Then it follows from the assumption \eqref{Formula: S(0)} that $\rho(r)$ is absolutely continuous with respect to $r\in(0,\delta_{0})$ and thus $\rho'(r)$ and $\lim_{r\to0^{+}}\rho(r)$ exists for a.e. $r\in(0,\delta_{0})$. It follows from \eqref{Formula: M(2)} that $\frac{d}{dr}W(u,r;Q)\geqslant\mathcal{F}_{i}(u,r;F)+\mathcal{Q}_{i}(r;Q)$ for $i=1$ and $i=2$ with $\mathcal{Q}_{2}(r;Q)=0$. It follows from the second and the third property of Proposition \eqref{Proposition: properties of F and Q} that $\frac{d}{dr}W(u,r;Q)\geqslant-C-\frac{Q_{\mathrm{max}}}{r}\operatorname{osc}_{B_{r}}Q$. Thus the function $G(u,r;Q):=W(u,r;Q)+Cr+CQ_{\mathrm{max}}\rho(r)$ satisfies $\frac{d}{dr}G(u,r;Q)\geqslant0$ for a.e. $r\in(0,\delta_{0})$ and therefore the limit $\lim_{r\to0^{+}}G(u,r;Q)$ exists. Since $\lim_{r\to0^{+}}Cr=0$, due to the fact that $C$ is independent of $r$ and $\lim_{r\to0^{+}}\rho(r)=0$ we obtain that the limit $\lim_{r\to0^{+}}W(u,r;Q)\in[-\infty,+\infty)$ exists. Since $u(0)=0$ we may see by regularity results that $W(u,0^{+};Q)>-\infty$. This proves (1).
	
	(2). To prove the second statement, we use the formula \eqref{Formula: M(2)} and one can also use the formula \eqref{Formula: M(22)} to obtain the same result without effort. We integrate the inequality \eqref{Formula: M(2)} and we have
	\begin{align}\label{Formula: S(1)}
		\begin{alignedat}{2}
			&\int_{r_{n}\varrho}^{r_{n}\sigma}\frac{1}{r}\int_{\partial B_{1}}|x\cdot\nabla u_{r}-u_{r}|^{2}\:d\mathcal{H}^{d-1}dr\\
			&\leqslant W(u,r_{n}\sigma,Q)-W(u,r_{n}\varrho,Q)+\int_{r_{n}\varrho}^{r_{n}\sigma}C\:dr+CQ_{\mathrm{max}}(\rho(r_{n}\sigma)-\rho(r_{n}\varrho)),
		\end{alignedat}
	\end{align}
    where we have used \eqref{Formula: M(26)}, \eqref{Formula: M(27)} and \eqref{Formula: S(0)}. It follows by rescaling \eqref{Formula: S(1)} that
    \begin{align*}
    	0&\leqslant\int_{B_{\sigma}\setminus B_{\varrho}}|z|^{-d-2}(\nabla u_{n}\cdot z-u_{n})\:dz\\
    	&\leqslant W(u,r_{n}\sigma,Q)-W(u,r_{n}\varrho,Q)+Cr_{n}(\sigma-\varrho)+CQ_{\mathrm{max}}(\rho(\sigma r_{n})-\rho(\varrho r_{n})).
    \end{align*}
    Passing to the limit as $n\to\infty$ we have the desired homogeneity.
\end{proof}
\begin{remark}
	Although in the above Proposition, we mainly focused on the homogeneity of free boundary point $0\in\partial\varOmega^{+}(u)$, one can also consider the case  $0\in\varOmega^{+}(u)$. In fact, if $0\in\Omega$ with $u(0)>0$, then it follows from the regularity results $|\nabla u|^{2}\leqslant C$ in $\Omega$ that $W(u,r,Q)\leqslant C_{1}-\tfrac{C_{2}}{r^{2}}$ for any $0<r<\delta_{0}$. This implies that $W(u,0^{+},Q)=-\infty$ for all such points. The homogeneity of $u_{0}$ at these points can be verified same as in the second part of the proof in Proposition \ref{Proposition: Homogeneity of blow-up limits}.
\end{remark}
\begin{remark}
	In a significant work \cite{VW2012} on two-dimensional gravity water waves V\v{a}rv\v{a}ruc\v{a} and Weiss on two-dimensional gravity water waves with vorticity. They put forward a monotonicity formula to investigate the local behavior of "stagnation points", at which points $u$ is a  $\tfrac{3}{2}$-homogeneous function. We point out that one can apply Proposition \ref{Proposition: Homogeneity of blow-up limits} to deduce the homogeneity for those points away from the stagnation points.
\end{remark}
In the rest of this section, we will focus on the limit
\begin{align}\label{Formula: S(2)}
	\vartheta:=\lim_{r\to0^{+}}\frac{|B_{r}(x^{0})\cap\{u>0\}|}{|B_{r}(x^{0})|}\quad\text{ for every }\quad x^{0}\in\partial\varOmega^{+}(u).
\end{align}
It follows from Proposition \ref{Proposition: density estimates} that $\vartheta\in(0,1)$ and we will finally prove that $\vartheta\in[\tfrac{1}{2},1)$ for every free boundary point and that the regular part of the free boundary $\partial_{\mathrm{reg}}\varOmega^{+}(u)$ is fully characterized by this limit. Most of the ideas come from \cite{MTV2017} and we begin with the following observation.
\begin{proposition}
	Let $u\in\mathcal{K}$ be a minimizer of $J$ in $\Omega$ with $0\in\Omega$ and $u(0)=0$, let $F(t)$ satisfy \eqref{Formula: Assumption on F(t)} and let $Q(x)$ be a H\"{o}lder continuous function at the origin which satisfies \eqref{Formula: Assumption on Q(x)}, then the limit $\vartheta$ defined in \eqref{Formula: S(2)} exists and
	\begin{align}\label{Formula: S(3)}
		\vartheta=\frac{|B_{1}\cap\{u_{0}>0\}|}{|B_{1}|}.
	\end{align}
\end{proposition}
\begin{proof}
	Let $\{r_{n}\}_{n=1}^{\infty}>0$ be a vanishing sequence and let $u_{n}$ be a blow-up sequence at with respect to $B_{r_{n}}$. It follows from Proposition \ref{Proposition: Properties of blow-up limits} (3), (4) and Proposition \ref{Proposition: Homogeneity of blow-up limits} (2) that
	\begin{align*}
		\lim_{n\to\infty}W(u_{n};Q(r_{n}x))&=\int_{B_{1}}|\nabla u_{0}|^{2}\:dx-\int_{\partial B_{1}}u_{0}^{2}\:d\mathcal{H}^{d-1}+Q^{2}(0)\lim_{r\to0^{+}}\frac{\int_{B_{r}}\chi_{\{u>0\}}\:dx}{r^{d}}\\
		&=Q^{2}(0)\omega_{d}\lim_{r\to0^{+}}\frac{|B_{r}\cap\{u>0\}|}{|B_{r}|}=Q^{2}(0)\omega_{d}\vartheta,
	\end{align*}
    The limit here exists because $\lim_{r\to0^{+}}W(u,r;Q)$ exists (by Proposition \ref{Proposition: Homogeneity of blow-up limits} (1)). On the other hand, since $\chi_{\{u_{n}>0\}}\to\chi_{\{u_{0}>0\}}$ in $L_{\mathrm{loc}}^{1}(\mathbb{R}^{d})$ (by Proposition \ref{Proposition: Properties of blow-up limits} (2)), we have that
    \begin{align*}
    	\lim_{n\to\infty}\int_{B_{1}}\chi_{\{u_{n}>0\}}Q^{2}(r_{n}x)\:dx=Q^{2}(0)\int_{B_{1}}\chi_{\{u_{0}>0\}}\:dx
    \end{align*}
    Thus 
    \begin{align*}
    	Q^{2}(0)\omega_{d}\frac{|B_{1}\cap\{u_{0}>0\}|}{|B_{1}|}=\lim_{n\to\infty}W(u_{n};Q(r_{n}x))=Q^{2}(0)\omega_{d}\vartheta,
    \end{align*}
    which is exactly \eqref{Formula: S(3)}.
\end{proof}
\begin{remark}\label{Remark: u_{0}notequiv0}
	In view of \eqref{Formula: S(3)} and \eqref{Formula: W(12'')} that $\vartheta\in(0,1)$ and thus $u_{0}$ is a non-constantly-vanishing function in $B_{1}$. That is, $u_{0}\not\equiv0$ in $B_{1}(0)$.
\end{remark}
At this stage, we point out Theorem \ref{Theorem: Main theorem1} is a direct corollary of the following proposition.
\begin{proposition}\label{Proposition: Density estimates}
	Let $u\in\mathcal{K}$ be a minimizer of $J$ in $\Omega$ with $0\in\Omega$ and $u(0)=0$, let $F(t)$ satisfy \eqref{Formula: Assumption on F(t)} and let $Q(x)$ be a H\"{o}lder continuous function at the origin which satisfies \eqref{Formula: Assumption on Q(x)} with $Q(0)=1$. Then for $\vartheta$ be defined in \eqref{Formula: S(2)}, we have
	\begin{enumerate}
		\item $\vartheta\in\left[\frac{1}{2},1\right)$ for every $x^{0}\in\partial\varOmega^{+}(u)$.
		\item  $\partial_{\mathrm{reg}}\varOmega^{+}(u)=(\varOmega^{+}(u))^{(1/2)}$.
	\end{enumerate}
\end{proposition}
\begin{proof}
	The upper bound of $\vartheta$ follows directly from \eqref{Formula: S(3)} and \eqref{Formula: W(12'')}. Assume that $x^{0}\in\partial\varOmega^{+}(u)$ and that $u_{n}$ is the blow-up sequence converges to a blow-up limit $u_{0}\colon\mathbb{R}^{d}\to\mathbb{R}$. It follows once again from Proposition \ref{Proposition: Properties of blow-up limits} (6) and Proposition \ref{Proposition: Homogeneity of blow-up limits} (2) that $u_{0}$ is one-homogeneous and harmonic in $B_{1}^{+}(u_{0})$. This implies that $u_{0}$ is a solution to the equation (in polar coordinates) $-\Delta_{\mathbb{S}}u_{0}=(d-1)u_{0}$ on $\partial B_{1}^{+}(u_{0})$, where $\Delta_{\mathbb{S}}$ is defined in \eqref{Formula: The spherical Laplacian}. It follows from Remark 4.8 in \cite{MTV2017} that $\mathcal{H}^{d-1}(\partial B_{1}^{+}(u_{0}))\geqslant\frac{d\omega_{d}}{2}$. Using the homogeneity of $u_{0}$, we immediately have $|B_{1}\cap\{u_{0}>0\}|\geqslant\frac{\omega_{d}}{2}$. Now, the convergence  $\lim_{n\to\infty}\int_{B_{1}}\chi_{\{u_{n}>0\}}\:dx=\int_{B_{1}}\chi_{\{u_{0}>0\}}\:dx$ by Proposition \ref{Proposition: Properties of blow-up limits} (2) implies that
	\begin{align*}
		\vartheta=\lim_{n\to\infty}\frac{|B_{1}\cap\{u_{n}>0\}|}{|B_{1}|}=\frac{|B_{1}\cap\{u_{0}>0\}|}{|B_{1}|}\geqslant\frac{1}{2},
	\end{align*}
    which concludes (1). In the case of equality $\gamma=\frac{1}{2}$, we have by Remark 4.8 in \cite{MTV2022} that $u_{0}|_{\partial B_{1}}$ is precisely the first eigenvalue on $\mathbb{S}^{d-1}_{+}$, whose one-homogeneous extension is precisely the one-plane solution $(x\cdot\nu)^{+}$ for some $\nu\in\partial B_{1}$. This implies that if $u$ is a minimizer of $J$ in $\Omega$, then $\partial_{\mathrm{reg}}\varOmega^{+}(u)=(\varOmega^{+}(u))^{(1/2)}$.
\end{proof}
\section{Regularity of the free boundary near regular points}
In this section, we focus on the regularity of regular free boundary points, and the structure of singular free boundary points will be considered in the forthcoming section. We will first prove the uniqueness of blow-up limit and the polynomial convergence rate of a blow-up sequence. Then regularity of free boundaries is a direct consequence of these two facts. Such techniques are prevalent in recent research and are robust because there are quantities of arguments, for instance, via an epiperimetric inequality, to verify the uniqueness of blow-up limit and the rate of convergence for blow-up sequences. This method was mainly borrowed from [Chapter 8, \cite{V2019}]. In this section we assume without loss of generality that $0\in\Omega$ and that  $Q(0)=1$. We begin by showing that minimizers of $J$ in $\Omega$ are viscosity solutions in the sense of Definition \ref{Definition of viscosity solutions}.
\begin{proposition}[Minimizers are viscosity solutions]\label{Proposition: Minimizers are viscosity solutions}
	Let $u\in\mathcal{K}$ be a minimizer of $J$ in $\Omega$. Then $u$ is a viscosity solution of \eqref{Formula: Semilinear problem} in the sense of definition \ref{Definition of viscosity solutions}.
\end{proposition}
\begin{proof}
	\textbf{Step I.} Assume that $x^{0}\in\varOmega^{+}(u)$ and that $\phi\in C^{\infty}(\Omega)$ touching $u$ from below at $x^{0}$. Then $u\in C^{2,\alpha}$ smooth at $x^{0}$, $(u-\phi)(x^{0})=0$ and $(u-\phi)(x)\geqslant0$ in a small neighborhood of $x^{0}$. This implies that the function $(u-\phi)$ attains a local minimum at $x^{0}$ and so $\Delta(u-\phi)(x^{0})\geqslant0$ by the second derive test of calculus. Therefore,
	\begin{align*}
		\Delta\phi(x^{0})+f(\phi(x^{0}))&=\Delta\phi(x^{0})+f(u(x^{0}))\leqslant\Delta u(x^{0})+f(u(x^{0}))=0,
	\end{align*}
    where we have used the fact that $\Delta u(x^{0})+f(u(x^{0}))=0$ for $x^{0}\in\varOmega^{+}(u)$. It is easy to check the case in a similar way when $\phi$ touches $u$ from above at $x^{0}$.
    
    \textbf{Step II.} It remains to check that $u$ satisfies the free boundary condition $|\nabla u|=Q(x^{0})$ in the viscosity sense. To this end, suppose that the function $\phi$ touches $u$ from below at  $x^{0}\in\partial\varOmega^{+}(u)$ and assume without loss of generality that $x^{0}=0$. Consider the following two blow-up sequences $u_{n}(x):=\frac{u(r_{n}x)}{r_{n}}$ and $\phi_{n}(x):=\frac{\phi(r_{n}x)}{r_{n}}$ where $r_{n}\to0^{+}$ as $n\to\infty$. Since $u$ is a minimizer, we have that the convergence is uniform in $B_{1}$. Since $\phi$ is a smooth function in $\Omega$, the limit $\phi_{0}:=\lim_{n\to\infty}\phi_{n}(x)$ exists and we have $\phi_{0}=\xi_{0}\cdot x$, where the vector $\xi_{0}\in\mathbb{R}^{d}$ is precisely the gradient $\nabla\phi(0)$. Without loss of generality we may assume that $\xi_{0}=Ae_{d}$ for some non-negative constant $A\geqslant0$. Thus,
    \begin{align}\label{Formula: R(0')}
    	|\nabla\phi(0)|=|\nabla\phi_{0}(0)|=A\quad\text{ and }\quad\phi_{0}(x)=Ax_{d}.
    \end{align}
    Moreover, we can assume that $A>0$ since otherwise the inequality $|\nabla\phi|\leqslant Q(0)=1$ holds trivially. It follows from Proposition \ref{Proposition: Properties of blow-up limits} (6), Proposition \ref{Proposition: Homogeneity of blow-up limits} (2) and Remark \ref{Remark: u_{0}notequiv0} that $u_{0}\in H_{\mathrm{loc}}^{1}(\mathbb{R}^{d})$ is a non-trivial, continuous and one-homogeneous function, harmonic in $B_{1}^{+}(u_{0})$. By the uniform convergence we have  $u_{0}\geqslant\phi_{0}$ and this implies $\{u_{0}>0\}\supset\{\phi_{0}>0\}=\{x_{d}>0\}$. Thus it follows from Lemma 5.31 \cite{RTV2019} that $\{u_{0}>0\}=\{x_{d}>0\}$ because the case $\{u_{0}>0\}=\{x_{d}\neq0\}$ is ruled out (by \eqref{Formula: W(12'')} or Proposition \ref{Proposition: Density estimates}). Moreover, $u_{0}$ is a global minimizer of $J_{0}(v)$ where $J_{0}(v)$ is defined in \eqref{Formula: W(12')} and hence satisfies the optimality condition $|\nabla u_{0}|=Q(0)=1$. In conclusion, $u_{0}(x)=x_{d}^{+}$ for $x\in\mathbb{R}^{d}$. Finally, it follows from $u_{0}\geqslant\phi_{0}$ and \eqref{Formula: R(0')} that $A\leqslant1$, and we have from \eqref{Formula: R(0')} that $|\nabla\phi(0)|\leqslant1$, as desired. The case when $\phi$ touches $u$ from above is similar and we omit it.
\end{proof}
\begin{remark}\label{Remark: Minimizers are viscisity solutions}
	Since minimizers are continuous in $\Omega$, we define 
	\begin{align}\label{Formula: R(0)}
		\tilde{f}(x):=-(f\circ u)(x),\quad\text{ for every }\quad x\in\Omega.
	\end{align}
	It follows from a direct calculation that
	\begin{align*}
		|\tilde{f}(x)|=|f(u(x))|\leqslant|f(0)|+\max_{s\geqslant0}|f'(s)|\|u\|_{L^{\infty}(\Omega)}\leqslant C(F_{0},M_{0}),
	\end{align*}
	where $M_{0}:=\sup_{\Omega}\Psi$ and $\Psi$ is given in \eqref{Formula: Psi}. Thus, $\tilde{f}\in C^{0}(\Omega)\cap L^{\infty}(\Omega)$. Moreover, we infer from Proposition \ref{Proposition: Minimizers are viscosity solutions} that $u$ is a viscosity solution to
	\begin{align}\label{Formula: R(1)}
		\Delta u=\tilde{f}(x)\quad\text{ in }\quad\varOmega^{+}(u),\qquad|\nabla u|=Q(x)\quad\text{ on }\quad\partial\varOmega^{+}(u).
	\end{align}
\end{remark}
In what follows, we will apply De Silva's viscosity approach to prove the regularity of the regular set $\partial_{\mathrm{reg}}\varOmega^{+}(u)$ of the free boundary. The first observation is that the free boundaries are $\varepsilon$-flat near points in $\partial_{\mathrm{reg}}\varOmega^{+}(u)$.
\begin{lemma}[$\varepsilon$-flat for regular free boundary points]\label{Lemma: epsilon-flat for regular free boundary points}
	Let $u\in\mathcal{K}$ be a minimizer of $J$ in $\Omega$. Assume that $0\in\partial_{\mathrm{reg}}\varOmega^{+}(u)$, then there exists $\varepsilon_{1}>0$ such that the rescaling $u_{r}$ is $\varepsilon_{1}$-flat in $B_{1}$ in the direction $\nu\in\partial B_{1}$ for every $r\in(0,1)$. That is,
	\begin{align*}
		(x\cdot\nu-\varepsilon_{1})^{+}\leqslant u_{r}(x)\leqslant(x\cdot\nu+\varepsilon_{1})^{+}\quad\text{ for every}\quad x\in B_{1},\quad\text{ and some }\quad\nu\in\partial B_{1}.
	\end{align*}
\end{lemma}
\begin{proof}
	Recall the definition $\partial_{\mathrm{reg}}\varOmega^{+}(u)$, there is a vanishing sequence $r_{n}\to0^{+}$ such that the blow-up sequence $u_{n}(x)=\frac{u(r_{n}x)}{r_{n}}$ converges uniformly in $B_{1}$ to a function $u_{0}\colon\mathbb{R}^{d}\to\mathbb{R}$ of the form $u_{0}(x)=(x\cdot\nu)^{+}$ for some $\nu\in\mathbb{R}^{d}$. Then there exists $\varepsilon>0$ such that for $n$ large enough,
	\begin{align*}
		\|u_{n}-u_{0}\|_{L^{\infty}(B_{1})}<\varepsilon,
	\end{align*}
	which implies that
	\begin{align*}
		u_{n}>0\quad\text{ in }\quad\{x\cdot\nu>\varepsilon\}\qquad\text{ and }\qquad u_{n}=0\quad\text{ in }\quad\{x\cdot\nu<-\varepsilon\}.
	\end{align*}
	This gives that $u_{n}$ is $2\varepsilon$-flat in $B_{1}$. Namely, 
	\begin{align}\label{Formula: R(2)}
		(x\cdot\nu-2\varepsilon)^{+}\leqslant u_{n}(x)\leqslant(x\cdot\nu+2\varepsilon)^{+}\quad\text{ in }\quad B_{1},
	\end{align}
	for some $\nu\in\partial B_{1}$. Let now $r\in(0,1)$ be arbitrary, then there must exist $n\in\mathbb{N}$ such that $r_{n+1}\leqslant r\leqslant r_{n}$. Thus, there is $\varrho\in(0,1]$ such that $r=\varrho r_{n}$. Now since $u_{n}$ satisfies \eqref{Formula: R(2)}, we get that $u_{r}=u_{\varrho r_{n}}$ satisfies
	\begin{align*}
		\left(x\cdot\nu-\frac{2\varepsilon}{\varrho}\right)^{+}\leqslant u_{r}(x)\leqslant\left(x\cdot\nu+\frac{2\varepsilon}{\varrho}\right)^{+}\quad\text{ in }\quad B_{1}.
	\end{align*}
	Set $\varepsilon_{1}=\tfrac{2\varepsilon}{\varrho}>0$ we obtain for every $r\in(0,1)$, $u_{r}$ satisfies
	\begin{align*}
		(x\cdot\nu-\varepsilon_{1})^{+}\leqslant u_{r}\leqslant(x\cdot\nu+\varepsilon_{1})^{+}\quad\text{ in }\quad B_{1},
	\end{align*}
	for some $\nu\in\partial B_{1}$, as desired.
\end{proof}
With Lemma \ref{Lemma: epsilon-flat for regular free boundary points} at hand we obtain the following \emph{improvement of flatness} [Lemma 4.1 in \cite{S2011}] for viscosity solutions.
\begin{lemma}[Improvement of flatness]\label{Lemma: Improvement of flatness}
	Let $u\in\mathcal{K}$ be a minimizer of $J$ in $\Omega$. Assume that $0\in\partial_{\mathrm{reg}}\varOmega^{+}(u)$. Then there are universal constants $\kappa\in(0,1)$, $\sigma\in(0,1)$, $C_{0}>0$ and $\varepsilon_{0}>0$ such that if $u_{r}$ is $\varepsilon$-flat in $B_{1}$ in the direction $\nu\in\partial B_{1}$ for every  $r\in(0,1)$ and $\varepsilon\in(0,\varepsilon_{0}]$, then $u_{\kappa r}$ is $\sigma\varepsilon$-flat in $B_{1}$ in the direction $\tilde{\nu}\in\partial B_{1}$ with $|\tilde{\nu}-\nu|\leqslant C_{0}\varepsilon$.
\end{lemma}
With Lemma \ref{Lemma: Improvement of flatness} at hand, we are able to prove the core Lemma in this section:
\begin{lemma}[Uniqueness of the blow-up]\label{Lemma: Uniqueness of the blow-ups}
	Let $u\in\mathcal{K}$ be a minimizer of $J$ in $\Omega$. Assume that $0\in\partial_{\mathrm{reg}}\varOmega^{+}(u)$. Then there is a unique vector $\nu_{0}\in\partial B_{1}\subset\mathbb{R}^{d}$ such that 
	\begin{align}\label{Formula: R(3)}
		\|u_{r}-u_{0}\|_{L^{\infty}(B_{1})}\leqslant C_{1}r^{\gamma}\quad\text{ for every }\quad r\leqslant\frac{1}{2},
	\end{align}
	where
	\begin{align}\label{Formula: R(4)}
		u_{0}(x):=(\nu_{0}\cdot x)^{+}\quad\text{ for every }\quad x\in\mathbb{R}^{d}.
	\end{align}
	Here $C_{1}$, $\gamma$ are positive constants depending on $\kappa$, $\sigma$, $C_{0}$ and $\varepsilon_{0}$ (Recall Lemma \ref{Lemma: Improvement of flatness}).
\end{lemma}
\begin{proof}
	Repeated use of Lemma \ref{Lemma: Improvement of flatness} via $n$ times gives $u_{\frac{\kappa^{n}}{2}}$ is $\varepsilon_{0}\sigma^{n}$-flat in the direction $\nu_{n}\in\partial B_{1}$ such that
	\begin{align*}
		|\nu_{n}-\nu_{n+1}|\leqslant C_{0}\varepsilon_{0}\sigma^{n}\quad\text{ for every }n\in\mathbb{N}.
	\end{align*}
	This implies that $\{\nu_{n}\}$ is a Cauchy sequence so that there is a vector $\nu_{0}$ such that
	\begin{align*}
		\nu_{0}:=\lim_{n\to\infty}\nu_{n}\quad\text{ and }\quad|\nu_{0}-\nu_{n}|\leqslant\frac{C_{0}\varepsilon_{0}}{1-\sigma}\sigma^{n}.
	\end{align*}
	This implies
	\begin{align}\label{Formula: R(5)}
		|(x\cdot\nu_{0})_{+}-u_{\kappa^{n}/2}|\leqslant|(x\cdot\nu_{0})_{+}-(x\cdot\nu_{n}+\varepsilon_{0}\sigma^{n})|\leqslant\left(1+\frac{C_{0}}{1-\sigma}\right)\sigma^{n}\quad\text{ in }\quad B_{1}.
	\end{align}
	Thus we set
	\begin{align*}
		u_{0}(x)=(x\cdot\nu_{0})^{+}.
	\end{align*}
	Now we verify \eqref{Formula: R(3)}. Let $r\in(0,\tfrac{1}{2}]$ and let $n\in\mathbb{N}$ be such that $\tfrac{\kappa^{n+1}}{2}<r\leqslant\frac{\kappa^{n}}{2}$. It follows that there must exist $\varrho\in(\kappa,1]$ such that $r=\varrho r_{n}$. Since $u_{\kappa^{n}/2}$ satisfies
	\begin{align*}
		(x\cdot\nu_{n}-\varepsilon_{0}\sigma^{n})^{+}\leqslant u_{\kappa^{n}/2}\leqslant(x\cdot\nu_{n}+\varepsilon_{0}\sigma^{n})^{+}\quad\text{ in }\quad B_{1},
	\end{align*}
	it is handy that $u_{r}$ satisfies
	\begin{align*}
		\left(x\cdot\nu_{n}-\frac{\varepsilon_{0}}{\varrho}\sigma^{n}\right)^{+}\leqslant u_{r}\leqslant\left(x\cdot\nu_{n}+\frac{\varepsilon_{0}}{\varrho}\sigma^{n}\right)^{+}\quad\text{ in }\quad B_{1}.
	\end{align*}
	Thus, $\|u_{n}-u_{r}\|_{L^{\infty}(B_{1})}\leqslant\frac{\varepsilon_{0}}{\varrho}\sigma^{n}\leqslant\frac{\varepsilon_{0}}{\kappa}\sigma^{n}$. It follows from  \eqref{Formula: R(5)} that $\|u_{r}-u_{0}\|_{L^{\infty}(B_{1})}\leqslant(1+\frac{C_{0}}{1-\sigma}+\frac{1}{\kappa})\varepsilon_{0}\sigma^{n}$. Choose $\gamma$ so that $\kappa^{\gamma}=\sigma$, and we obtain
	\begin{align*}
		\|u_{r}-u_{0}\|_{L^{\infty}(B_{1})}\leqslant(2/\kappa)^{\gamma}\left(1+\frac{C_{0}}{1-\sigma}+\frac{1}{\kappa}\right)\varepsilon_{0}r^{\gamma}.
	\end{align*}
	This concludes the proof by letting $C_{1}=\left(1+\frac{C_{0}}{1-\sigma}+\frac{1}{\kappa}\right)$.
\end{proof}
We now introduce the following notion: For any $\varepsilon>0$, $x^{0}\in\Omega$ and $\nu\in\mathbb{R}^{d}$, we define $\mathcal{C}_{\varepsilon}^{\pm}(x^{0},\nu)$ to be the cones
\begin{align*}
	\mathcal{C}_{\varepsilon}^{\pm}(x^{0},\nu):=\{x\in\mathbb{R}^{d}\colon\pm\nu\cdot(x-x^{0})>\varepsilon|x-x^{0}|\}.
\end{align*}
In fact, as an application of the uniqueness of blow-ups, we have
\begin{corollary}[$\partial_{\mathrm{reg}}\varOmega^{+}(u)$ is Lipschitz continuity]\label{Corollary: Regular free boundary is Lipschitz continuity}
	Let $u\in\mathcal{K}$ be a minimizer of $J$ in $\Omega$. Assume that $0\in\partial_{\mathrm{reg}}\varOmega^{+}(u)$. Then $\partial\varOmega^{+}(u)$ is a Lipschitz graph in a small neighborhood of $0$.
\end{corollary}
\begin{proof}
	The idea of this proof is originally from Proposition 8.6 in \cite{V2019}, we adapt it to our settings.
	
	\textbf{Step I.} In this first step, we show that if $u$ minimizes $J$ in $B_{1}$ with $u(0)=0$ and $0\in\partial_{\mathrm{reg}}\varOmega^{+}(u)$. Then for every $\varepsilon>0$ there exists $R>0$ such that
	\begin{align}\label{Formula: R(5')}
		\begin{cases}
			\begin{alignedat}{2}
				u&>0\quad&&\text{ in }\mathcal{C}_{\varepsilon}^{+}(0,\nu_{0})\cap B_{R},\\
				u&=0\quad&&\text{ in }\mathcal{C}_{\varepsilon}^{-}(0,\nu_{0})\cap B_{R},
			\end{alignedat}
		\end{cases}
	\end{align}
	where $\nu_{0}\in\partial B_{1}$. To this end, let $\gamma$ and $C_{1}$ be defined as in the Lemma \ref{Lemma: Uniqueness of the blow-ups}. Then it follows from \eqref{Formula: R(3)} that $u_{r}\geqslant(x\cdot\nu_{0}-C_{1}r^{\gamma})^{+}$ for every $x\in B_{1}$. Choose $C$ be any positive constant such that $C\geqslant C_{1}$, we obtain
	\begin{align}\label{Formula: R(6)}
		\{x\in B_{1}\colon x\cdot\nu_{0}>Cr^{\gamma}\}\subset\varOmega^{+}(u_{r}).
	\end{align}
	On the other hand, observe that  $u_{r}$ is also a minimizer of $J$ in $B_{1}$, therefore $u_{r}$ is non-degenerate in $B_{1}$ so that 
	\begin{align}\label{Formula: R(7)}
		\varOmega^{+}(u_{r})\cap\{x\in B_{1}\colon x\cdot\nu_{0}<-Cr^{\gamma}\}=\varnothing
	\end{align}
	Estimates \eqref{Formula: R(6)} and \eqref{Formula: R(7)} imply \eqref{Formula: R(5')} by taking $R>0$ such that $CR^{\gamma}\leqslant\varepsilon$.
	
	\textbf{Step II.} In this step, we prove that the free boundary near the origin can be expressed by the graph of a function $g$, precisely, there exists a $\delta>0$ such that
	\begin{align*}
		(B_{\delta}'\times(-\delta,\delta))\cap\partial\varOmega^{+}(u)=\{(x',t)\in B_{\delta}'\times(-\delta,\delta)\colon t=g(x')\}.
	\end{align*}
	Without loss of generality, we assume that $\nu_{0}=e_{d}$ and therefore $u_{0}(x)=x_{d}^{+}$ is the blow-up limit in $0$. Thus $\varOmega^{+}(u_{0})=\{x_{d}>0\}$. Let now $\varepsilon\in(0,1)$ and $R$ be as in \eqref{Formula: R(5')}, it follows that
	\begin{align*}
		\begin{cases}
			\begin{alignedat}{2}
				u&>0\quad&&\text{ in }\mathcal{C}_{\varepsilon}^{+}(0,e_{d})\cap B_{R},\\
				u&=0\quad&&\text{ in }\mathcal{C}_{\varepsilon}^{-}(0,e_{d})\cap B_{R},
			\end{alignedat}
		\end{cases}
	\end{align*}
	where $\mathcal{C}_{\varepsilon}^{\pm}$ is the cone $\{x\in\mathbb{R}^{d}\colon\pm x_{d}>\varepsilon|x|\}$. Let $x_{0}'\in B_{\delta}'$ where $\delta\leqslant R\sqrt{1-\varepsilon^{2}}$, it follows that the segment $\{(x_{0}',t)\colon t>\varepsilon R\}$ is contained in $\{u>0\}$ and the segment $\{(x_{0}',t)\colon t<-\varepsilon R\}$ is contained in $\{u=0\}$. Therefore, $g(x_{0}'):=\inf\{t\colon u(x_{0}',t)>0\}$ is well defined and $x^{0}:=(x_{0}',g(x_{0}'))\in\partial\varOmega^{+}(u)\cap B_{R}$. Moreover, one has
	\begin{align*}
		-\varepsilon|x_{0}'|\leqslant g(x_{0}')\leqslant\varepsilon|x_{0}'|,
	\end{align*}
	which implies that $|x^{0}|\leqslant|x_{0}'|\sqrt{1+\varepsilon^{2}}\leqslant\sqrt{2}\delta$. In order to finish the claim, it suffices to show that if $\delta>0$ is small enough, then
	\begin{align}\label{Formula: R(8)}
		\begin{cases}
			\begin{alignedat}{2}
				u&>0\quad&&\text{ in }\mathcal{C}_{2\varepsilon}^{+}(x^{0},e_{d})\cap B_{R}(x^{0}),\\
				u&=0\quad&&\text{ in }\mathcal{C}_{2\varepsilon}^{-}(x^{0},e_{d})\cap B_{R}(x^{0}).
			\end{alignedat}
		\end{cases}
	\end{align}
	Indeed, as a direct consequence of \eqref{Formula: R(8)}, one can easily deduce that
	\begin{align*}
		&(B_{\delta}'\times(-\delta,\delta))\cap\{u>0\}=\{(x',t)\in B_{\delta}'\times(-\delta,\delta)\colon t>g(x')\},\\
		&(B_{\delta}'\times(-\delta,\delta))\cap\{u=0\}=\{(x',t)\in B_{\delta}'\times(-\delta,\delta)\colon t\leqslant g(x')\}.
	\end{align*}
	We now prove \eqref{Formula: R(8)} to conclude this step. Applying \eqref{Formula: R(8)} for the free boundary point $x^{0}$, we have
	\begin{align*}
		\begin{cases}
			\begin{alignedat}{2}
				u&>0\quad&&\text{ in }\mathcal{C}_{\varepsilon}^{+}(x^{0},\nu_{x^{0}})\cap B_{R}(x^{0}),\\
				u&=0\quad&&\text{ in }\mathcal{C}_{\varepsilon}^{-}(x^{0},\nu_{x^{0}})\cap B_{R}(x^{0}).
			\end{alignedat}
		\end{cases}
	\end{align*}
	In view of \eqref{Formula: R(8)}, we only need to show $\mathcal{C}_{2\varepsilon}^{\pm}(x^{0},e_{d})\subset\mathcal{C}_{\varepsilon}^{\pm}(x^{0},\nu_{x^{0}})$. Since $u$ is a Lipschitz function with Lipschitz coefficient $L$ in $B_{1}$ and $0$, $x^{0}\in\partial\varOmega^{+}(u)\cap B_{R}$, let $\alpha:=\frac{\gamma}{1+\gamma}$ and $r:=|x^{0}|^{1-\alpha}$,
	\begin{align*}
		|u_{r}-u_{x^{0},r}|=\frac{1}{r}|u(rx)-u(x^{0}+rx)|\leqslant\frac{L}{r}|x^{0}|=L|x^{0}|^{\alpha},
	\end{align*}
	for every $x\in B_{1}$. This implies $\|u_{r}-u_{x^{0},r}\|_{L^{\infty}(B_{1})}\leqslant L|x^{0}|^{\alpha}$. On the other hand, if we denote $u_{0}$ and $u_{x^{0}}$ the unique blow-up limit of $u$ at $0$ and $x^{0}$ respectively. It follows from \eqref{Formula: R(8)} that
	\begin{align*}
		\|u_{x^{0}}-u_{x^{0},r}\|_{L^{\infty}(B_{1})}\leqslant C_{1}r^{\gamma}\quad\text{ and }\quad\|u_{r}-u_{0}\|_{L^{\infty}(B_{1})}\leqslant C_{1}r^{\gamma}.
	\end{align*}
	Choose $R$ such that $(2R)^{1-\alpha}\leqslant\tfrac{1}{2}$, we have by $r^{\gamma}=|x^{0}|^{\alpha}$ that
	\begin{align*}
		\|u_{x^{0}}-u_{0}\|_{L^{\infty}(B_{1})}\leqslant(L+2C_{1})|x^{0}|^{\alpha}.
	\end{align*}
	This implies $|\nu_{0}-e_{d}|\leqslant C|x^{0}|^{\alpha}$ by a direct calculation. Let $x\in\mathcal{C}_{2\varepsilon}^{\pm}(x^{0},e_{d})$. Then,
	\begin{align*}
		\nu_{x^{0}}\cdot(x-x^{0})&=e_{d}\cdot(x-x^{0})+(\nu_{x^{0}}-e_{d})\cdot(x-x^{0})\\
		&>2\varepsilon|x-x^{0}|-C(\sqrt{2}\delta)^{\alpha}|x-x^{0}|>\varepsilon|x-x^{0}|,
	\end{align*}
	where we choose $\delta$ such that $C(\sqrt{2}\delta)^{\alpha}\leqslant\varepsilon$. This finishes $\mathcal{C}_{2\varepsilon}^{\pm}(x^{0},e_{d})\subset\mathcal{C}_{\varepsilon}^{\pm}(x^{0},\nu_{x^{0}})$ and thus \eqref{Formula: R(8)} is concluded.
	
	\textbf{Step III.} We now prove that $g\colon B_{\delta}'\to\mathbb{R}$ is Lipschitz continuous on $B_{\delta}'$. Let $x_{1}'$, $x_{2}'\in B_{\delta}'$ where $x^{1}=(x_{1}',g(x_{1}'))$ and $x^{2}=(x_{2}',g(x_{2}'))$. Since $x^{1}\notin\mathcal{C}_{2\varepsilon}^{+}(x^{2},e_{d})$, we have that
	\begin{align*}
		g(x_{1}')-g(x_{2}')=(x^{1}-x^{2})\cdot e_{d}\leqslant2\varepsilon|x_{1}'-x_{2}'|+2\varepsilon|g(x_{1}')-g(x_{2}')|.
	\end{align*}
	A same inequality can be derived since $x^{2}\notin\mathcal{C}_{2\varepsilon}^{+}(x^{1},e_{d})$. Two estimates give $(1-2\varepsilon)|g(x_{1}')-g(x_{2}')|\leqslant2\varepsilon|x_{1}'-x_{2}'|$. Choosing $\varepsilon\leqslant\frac{1}{4}$, we obtain
	\begin{align*}
		|g(x_{1}')-g(x_{2}')|\leqslant4\varepsilon|x_{1}'-x_{2}'|,
	\end{align*}
	which concludes the proof of the Lipschitz continuity of $g$.
\end{proof}
We are now in a position to prove the regularity of the free boundary near regular points.
\begin{proposition}[Regularity of $\partial_{\mathrm{reg}}\varOmega^{+}(u)$]\label{Proposition: Regularity of regular points}
	Let $u\in\mathcal{K}$ be a minimizer of $J$ in $\Omega\subset\mathbb{R}^{d}$. Assume that $0\in\partial_{\mathrm{reg}}\varOmega^{+}(u)$, then there exists a small neighborhood, the size of which depends on $d$, $[Q]_{C^{\beta};\Omega}$, $F_{0}$ and $M_{0}$ so that $\partial_{\mathrm{reg}}\varOmega^{+}(u)$ is $C^{1,\gamma}$ regularity in that neighborhood for $\gamma\in(0,\beta)$.
\end{proposition}
\begin{proof}
	Consider the blow-up sequence $u_{k}(x):=u_{\delta_{k}}(x)=\frac{u(\delta_{k}x)}{\delta_{k}}$ with $\delta_{k}\to0$ as $k\to\infty$. It is easy to see that for every $k$, $u_{k}$ is a solution to
	\begin{align}\label{Formula: R(9)}
		\Delta u_{k}=\tilde{f}_{k}(x)\quad\text{ in }\quad B_{1}^{+}(u_{k}),\qquad|\nabla u_{k}|=Q_{k}(x)\quad\text{ on }\quad\partial B_{1}^{+}(u_{k}),
	\end{align}
	where $\tilde{f}_{k}(x)=\delta_{k}\tilde{f}(\delta_{k}x)$ and $Q_{k}(x)=Q(\delta_{k}x)$. It follows that there is a positive constant $\bar{\varepsilon}$ such that the following does hold:
	\begin{align*}
		|\tilde{f}_{k}(x)|=|\delta_{k}\tilde{f}(\delta_{k}x)|=\delta_{k}\|\tilde{f}\|_{L^{\infty}}\leqslant\bar{\varepsilon},
	\end{align*}
	and
	\begin{align*}
		&|Q_{k}(x)-1|=|Q(\delta_{k}x)-Q(0)|\leqslant\delta_{k}^{\beta}[Q]_{\beta}\leqslant\bar{\varepsilon}.
	\end{align*}
	Thus, using Proposition \ref{Proposition: Properties of blow-up limits}, up to extracting a subsequence: $u_{k}\to u_{0}$ locally and uniformly in $\mathbb{R}^{d}$ and $\partial\varOmega^{+}(u_{k})\to\partial\varOmega^{+}(u_{0})$ locally in the Hausdorff distance in $\mathbb{R}^{d}$ for a globally defined function $u_{0}\colon\mathbb{R}^{d}\to\mathbb{R}$. Moreover, every blow-up limit $u_{0}$ is a solution to $\Delta u_{0}=0$ in $B_{1}^{+}(u_{0})$ and $|\nabla u_{0}|=1$ on $\partial B_{1}^{+}(u_{0})$. It follows from Corollary \ref{Corollary: Regular free boundary is Lipschitz continuity} that $\partial\varOmega^{+}(u)$ is a Lipschitz graph in a small neighborhood of $0$. We also see from (1)-(2) above that $\partial\varOmega^{+}(u_{0})$ is Lipschitz continuous. Thus, $u_{0}$ is a so-called one-plane solution, namely (up to rotations), $u_{0}(x)=x_{d}^{+}$. Consequently, we conclude that for all $k$ large enough, $u_{k}$ is $\bar{\varepsilon}$-flat in $B_{1}$, i.e.
	\begin{align*}
		(x_{d}-\bar{\varepsilon})^{+}\leqslant u_{k}(x)\leqslant(x_{d}+\bar{\varepsilon})^{+},\quad\text{ in }\quad B_{1}.
	\end{align*}
	Thus, $u_{k}$ satisfies the flatness condition required by Theorem 1.1 in \cite{S2011}. Moreover, $\tilde{f}_{k}$ and $Q_{k}$ also satisfy the requirements of Theorem 1.1 in  \cite{S2011}. Combining these facts, we conclude from Theorem 1.1 in \cite{S2011} that $\partial\varOmega^{+}(u)$ is $C^{1,\gamma}$ in $B_{1/2}$. This finishes the proof.
\end{proof}
\section{Dimension of the singular set}
This concluding section dedicates to investigate the structure of the singular set. We will prove Theorem \ref{Theorem: Main theorem2} for our semilinear case. Our results discover that the non-linear right-hand side in our settings does not affect the structure of singularities and our argument follows by "dimension reduction" process. We begin by giving the definition of $d^{*}$.
\begin{definition}[Definition of $d^{*}$, compare to \cite{W1999}]\label{Definition: d*}
	We define $d^{*}$ to be the smallest dimension $d$ so that there exists a function $z\colon\mathbb{R}^{d}\to\mathbb{R}$ such that
	\begin{enumerate}
		\item $z$ is a non-negative one-homogeneous absolute minimum of $J_{0}$ given in \eqref{Formula: W(12')};
		\item The free boundary $\partial\{z>0\}$ is not a $(d-1)$-dimensional $C^{1}$-regular surface in $\mathbb{R}^{d}$.
	\end{enumerate}
\end{definition} 
In \cite{CD2020}, it was proved that $\partial_{\mathrm{sing}}\varOmega^{+}(u)=\varnothing$ when $d=2$. We present here a different proof, which depends mainly on the homogeneity of minimizers and this is completely different from \cite{CD2020}. 
\begin{proposition}\label{Proposition: d*geq3}
	Let $u\in\mathcal{K}$ be a minimizer of $J$ in $\Omega\subset\mathbb{R}^{2}$, then $\partial\varOmega^{+}(u)=\partial_{\mathrm{reg}}\varOmega^{+}(u)$.
\end{proposition}
\begin{proof}
	Let $u\in\mathcal{K}$ be a minimizer of $J$ in $\Omega\subset\mathbb{R}^{2}$. Then it follows from Proposition \ref{Proposition: Properties of blow-up limits} (4) that $u_{0}\colon\mathbb{R}^{2}\to\mathbb{R}$ is a global minimizer of $J_{0}(v)$ (recall \eqref{Formula: W(12')} in $\mathbb{R}^{2}$. We write $u_{0}$ in polar coordinates $u_{0}(\rho,\theta)=\rho c(\theta)$. Since $c$ is continuous, we have that $\{c>0\}$ is open and so it is a countable union of disjoint arcs. Based on the fact that local minimizers of $J_{0}(v)$ cannot have isolated zeros, we have that $\{c>0\}\neq\mathbb{S}^{1}$. It follows from Proposition \ref{Proposition: Properties of blow-up limits} and Proposition \ref{Proposition: Homogeneity of blow-up limits} (2) that up to a rotation, the trace $c$ is a solution to
	\begin{align*}
		c''(\theta)+c(\theta)=0\quad\text{ in }\quad(0,\pi),\qquad c(0)=c(\pi)=0.
	\end{align*}
	It follows that $c(\theta)$ is a multiple of $\sin\theta$ on $(0,\pi)$. Thus, $\{c>0\}$ is a union of disjoint arcs, each one of length $\pi$. Since $0\in\partial\varOmega^{+}(u_{0})$, thanks \eqref{Formula: W(12'')}, one has $|B_{1}\cap\{u_{0}>0\}|<|B_{1}|=\pi$. The homogeneity of $u_{0}$ then implies that  $\mathcal{H}^{1}(\{c>0\})<2\pi$. Therefore,  $\{c>0\}$ is an arc of length $\pi$ and $u_{0}$ is of the form $u_{0}(x)=a(x\cdot\nu)^{+}$ for some $a>0$. Thanks to Proposition \ref{Proposition: Properties of blow-up limits} (5), we obtain $a=Q(0)$. Consequently, $u_{0}(x)=Q(0)(x\cdot\nu)^{+}$ and it follows from the definition of $\partial_{\mathrm{reg}}\varOmega^{+}(u)$ that $\partial\varOmega^{+}(u)=\partial_{\mathrm{reg}}\varOmega^{+}(u)$. Thus by Definition \ref{Definition: d*},  $d^{*}\geqslant3$.
\end{proof}
When $d<d^{*}$, we have the following fact
\begin{proposition}
	Let $u\in\mathcal{K}$ be a minimizer of $J$ in $\Omega\subset\mathbb{R}^{d}$ with $2\leqslant d<d^{*}$, then $\partial_{\mathrm{sing}}\varOmega^{+}(u)=\varnothing$ and $\partial\varOmega^{+}(u)=\partial_{\mathrm{reg}}\varOmega^{+}(u)$.
\end{proposition}
\begin{proof}
	Let $x^{0}\in\partial\varOmega^{+}(u)$ be an arbitrary free boundary point and let $r_{n}\to0$ be a infinitesimal sequence. Then it follows that the blow-up sequence $u_{n}$ converges uniformly to a blow-up limit $u_{0}$. Thanks to Proposition \ref{Proposition: Properties of blow-up limits} (4), $u_{0}$ is a global minimizer of $J_{0}$ in $\mathbb{R}^{d}$ where $J_{0}$ is defined in \eqref{Formula: W(12')}. It then follows from the definition of $d^{*}$ that $\partial_{\mathrm{sing}}\varOmega^{+}(u_{0})=\varnothing$. By the definition of $\partial_{\mathrm{reg}}\varOmega^{+}(u)$, we obtain that every blow-up $u_{00}$ of $u_{0}$ is of the form $Q(0)(x\cdot\nu)^{+}$ for some $\nu\in\partial B_{1}$. In particular, it holds for every blow-up limit in $0$. Since (by Proposition \ref{Proposition: Homogeneity of blow-up limits} (2)) $u_{0}$ is one-homogeneous function, we have
	\begin{align*}
		u_{00}(x)=\lim_{n\to\infty}\frac{u_{0}(r_{n}x)}{r_{n}}=\lim_{n\to\infty}u_{0}(x)=u_{0}(x),
	\end{align*}
	which implies that the blow-up limit $u_{00}$ of $u_{0}$ is indeed $u_{0}$ itself, and so
	\begin{align*}
		u_{0}(x)=Q(0)(x\cdot\nu)^{+},\quad\text{ for some }\nu\in B_{1}.
	\end{align*}
	Thus, for $n$ large enough, we have
	\begin{align*}
		\|u_{n}(x)-Q(0)(x\cdot\nu)^{+}\|_{L^{\infty}(B_{1})}\leqslant\varepsilon,\quad\text{ for some }\nu\in B_{1}.
	\end{align*}
	Set $y=x^{0}+r_{n}x$ and we have from the definition of $u_{n}$ that
	\begin{align*}
		\left\Vert\frac{u(y)}{r_{n}}-\frac{Q(0)}{r_{n}}((y-x^{0})\cdot\nu)^{+}\right\Vert_{L^{\infty}(B_{r_{n}}(x^{0}))}\leqslant\varepsilon,\quad\text{ for some }\nu\in B_{1}.
	\end{align*}
	This implies that there exists $r\in(0,R)$ such that
	\begin{align}\label{Formula: Sing (1(1))}
		\|u(y)-Q(0)((y-x^{0})\cdot\nu)^{+}\|_{L^{\infty}(B_{r}(x^{0}))}\leqslant\varepsilon r,\quad\text{ for some }\nu\in B_{1}.
	\end{align}
	In via of \eqref{Formula: Sing (1(1))}, there must exist $\bar{\varepsilon}>0$ such that $u$ is $\bar{\varepsilon}$-flat in $B_{1}$. Moreover, since $u$ minimizes $J$ in $B_{1}$, it follows from Proposition \ref{Proposition: Minimizers are viscosity solutions} that $u$ is a viscosity solution in $B_{1}\subset\mathbb{R}^{d}$. Thanks to Proposition \ref{Proposition: Regularity of regular points}, we know that $\partial\varOmega^{+}(u)$ is $C^{1,\gamma}$ in a small neighborhood of $x^{0}\in\partial\varOmega^{+}(u)$. Thus, $x^{0}\in\partial_{\mathrm{reg}}\varOmega^{+}(u)$. Since $x^{0}$ is arbitrary, we conclude that $\partial_{\mathrm{sing}}\varOmega^{+}(u)=\varnothing$ and this concludes the proof.
\end{proof}
We now investigate the case $d=d^{*}$. Let $u\in\mathcal{K}$ be a minimizer of $J$ in $\Omega\subset\mathbb{R}^{d^{*}}$ (Proposition \ref{Proposition: d=d^{*}} below). To this end, we begin by studying the singular set of a one-homogeneous function in $\mathbb{R}^{d}$. In \cite{W1999}, Weiss showed that the singular set of one-homogeneous function is a cone in $\mathbb{R}^{d}$ in homogeneous case. We follow his approach and extend the same result to our case.
\begin{lemma}
	Let $u\in\mathcal{K}$ be a minimizer of $J$ in $\Omega\subset\mathbb{R}^{d}$. Then the graph of any blow-up limit $u_{0}$ is a minimal cone with vertex at $0\in\mathbb{R}^{d}$.
\end{lemma}
\begin{proof}
	It suffices to prove the following two facts:
	\begin{enumerate}
		\item Every blow-up limit $u_{0}$ of $u$ is a non-negative one-homogeneous global minimizer of $J_{0}$ in $\mathbb{R}^{d}$.
		\item For any open and bounded set $D\subset\mathbb{R}^{d}$ satisfying $\bar{D}\cap(\partial\varOmega^{+}(u_{0})\setminus\partial_{\mathrm{red}}\varOmega^{+}(u_{0}))=\varnothing$,
		\begin{align*}
			\operatorname{Per}(\varOmega^{+}(u_{0}),D)\leqslant\operatorname{Per}(F,D),
		\end{align*}
		for any $F\subset\varOmega^{+}(u_{0})$ so that $\partial F\cap D$ is a $C^{2}$ surface and $\operatorname{supp}(\chi_{F}-\chi_{\{u_{0}>0\}})\subset D$. Here $\operatorname{Per}(E)$ denotes the perimeter of a given set $E\subset\mathbb{R}^{d}$.
	\end{enumerate}
	Notice that (1) is a direct consequence of Proposition \ref{Proposition: Properties of blow-up limits} (5) and Proposition \ref{Proposition: Homogeneity of blow-up limits} (2). To prove the above mentioned second property, since $Q(x)\in C^{0,\beta}(\Omega)$, it follows from Theorem 5.5 in \cite{AC1981} that $u_{0}$ is a weak solution. Then a same argument as in Theorem 2.8 in \cite{W1999} gives the desired result.
\end{proof}
We now consider a blow-up limit $u_{00}$ of $u_{0}$ at $x^{0}\neq0\in\partial\varOmega^{+}(u_{0})$, the following fact was first proved by Weiss in \cite{W1999} for minimizers of $J_{0}$, we adapt it to our settings.
\begin{lemma}\label{Lemma: Graph of u_{00}}
	Let $u\in\mathcal{K}$ be a minimizer of $J$ in $\Omega\subset\mathbb{R}^{d}$ and let $u_{0}$ be any blow-up limits of $u$ at any free boundary points. Then the graph of $u_{00}$ is a minimal cone with vertex at $0$ and $u_{00}$ is constant in the direction $x^{0}$. That is, $u_{00}(x)=u_{00}(x+tx^{0})$ for $x\in\mathbb{R}^{d}$ and $t\in\mathbb{R}$. Moreover,  $\{t\in\mathbb{R}\colon tx^{0}\}\subset\partial_{\mathrm{sing}}\varOmega^{+}(u)$.
\end{lemma}
\begin{proof}
	The proof is similar to Lemma 3.1 in \cite{W1999}. Notice that $u_{0}$ is a global minimizer of $J_{0}$ in $\mathbb{R}^{d}$, therefore one has
	\begin{align*}
		u_{00}(x+tx^{0})&=\lim_{n\to\infty}\frac{u_{0}(x^{0}+r_{n}(x+tx^{0}))}{r_{n}}\\
		&=\lim_{n\to\infty}\left(\frac{1+tr_{n}}{r_{n}}\right)u_{0}\left(x^{0}+\frac{r_{n}}{1+tr_{n}}x\right),
	\end{align*}
	where we have used the homogeneity of $u_{0}$ in the second identity. We now claim that
	\begin{align}\label{Formula: Sing (2)}
		\lim_{n\to\infty}\left(\frac{1+tr_{n}}{r_{n}}\right)u_{0}\left(x^{0}+\frac{r_{n}}{1+tr_{n}}x\right)=\lim_{n\to\infty}\frac{u_{0}(x^{0}+r_{n}x)}{r_{n}}.
	\end{align}
	Since $u_{0}$ is a global minimizer of $J_{0}$, thanks to Corollary 3.3 in \cite{AC1981}, $u_{0}$ is a Lipschitz function with Lipschitz coefficient $L_{0}:=\|\nabla u_{0}\|_{L^{\infty}(B_{1})}$. Then a straightforward estimate gives
	\begin{align*}
		&\left|\frac{1+tr_{n}}{r_{n}}u_{0}\left(x^{0}+\frac{r_{n}}{1+tr_{n}}x\right)-\frac{1}{r_{n}}u_{0}(x^{0}+r_{n}x)\right|\\
		&\leqslant t\left|u_{0}\left(x^{0}+\frac{r_{n}}{1+tr_{n}}x\right)-u_{0}(x^{0})\right|+\frac{1}{r_{n}}\left|u_{0}\left(x^{0}+\frac{r_{n}}{1+tr_{n}}x\right)-u_{0}(x^{0}+r_{n}x)\right|\\
		&\leqslant\frac{2tr_{n}L_{0}|x|}{1+tr_{n}}\to0\quad\text{ as }n\to\infty.
	\end{align*}
	Then the desired result follows from \eqref{Formula: Sing (2)}.
\end{proof}
In via of a same argument as in Lemma 3.2 in \cite{W1999}, one has the following property for $u_{00}$ given in Lemma \ref{Lemma: Graph of u_{00}}.
\begin{lemma}\label{Lemma: Graph of u_{00}'}
	The graph of $u_{00}|_{\mathbb{R}^{d-1}}=\bar{u}$ is a minimal cone and $0\in\partial_{\mathrm{sing}}\varOmega^{+}(\bar{u})$.
\end{lemma}
As a direct corollary of Lemma \ref{Lemma: Graph of u_{00}} and Lemma \ref{Lemma: Graph of u_{00}'}, one has
\begin{proposition}\label{Proposition: d=d^{*}}
	Let $u\in\mathcal{K}$ be minimizer of $J$ in $\Omega\subset\mathbb{R}^{d^{*}}$. Then $\partial_{\mathrm{sing}}\varOmega^{+}(u)\setminus\{0\}=\varnothing$. In particular, this means $\dim_{\mathcal{H}}\partial_{\mathrm{sing}}\varOmega^{+}(u)=0$.
\end{proposition}
\begin{proof}
	\textbf{Step I.} Suppose there is a sequence of singular free boundary points $x_{n}\in\partial_{\mathrm{sing}}\varOmega^{+}(u)$ converging to some point $x^{0}$ as $n\to\infty$, and we claim $x^{0}\in\partial_{\mathrm{sing}}\varOmega^{+}(u)$. Let us consider the blow-up sequence $u_{n}$ at $x^{0}$ and we denote by $u_{0}$ the blow-up limit of $u_{n}$ as $n\to\infty$. On the one hand, the uniform convergence of $u_{n}$ implies that $u_{0}(x^{0})=0$. On the other hand, since for each $n$, $u_{n}$ minimizes $J$, we may infer from Corollary \ref{Corollary: Optimal linear nondegeneracy} that for every $r>0$ small
	\begin{align*}
		\|u_{0}\|_{L^{\infty}(B_{r}(x^{0}))}&\geqslant\lim_{n\to\infty}(\|u_{n}\|_{L^{\infty}(B_{r}(x^{0}))}-\|u_{n}-u_{0}\|_{L^{\infty}(B_{r}(x^{0}))})\\
		&\geqslant\liminf_{n\to\infty}\|u_{n}\|_{L^{\infty}(B_{r/2}(x^{0}))}\geqslant\eta r,
	\end{align*}
	which gives $x^{0}\in\partial\varOmega^{+}(u_{0})$. If $x^{0}\in\partial_{\mathrm{reg}}\varOmega^{+}(u_{0})$, so that every blow-up limit $u_{00}$ of $u_{0}$ at $x^{0}$ is of the form $Q(0)((x-x^{0})\cdot\nu)^{+}$ for some $\nu\in\partial B_{1}$. Moreover, there is a sequence $r_{n}\to0$ so that
	\begin{align*}
		\lim_{n\to\infty}\left\Vert\frac{u_{0}(x^{0}+r_{n}(x-x^{0}))}{r_{n}}-Q(0)((x-x^{0})\cdot\nu)^{+}\right\Vert_{L^{\infty}(B_{1})}=0.
	\end{align*}
	Let $y:=x^{0}+r_{n}(x-x^{0})$, one has
	\begin{align*}
		0=\lim_{n\to\infty}\frac{1}{r_{n}}\left\Vert u_{0}(y)-Q_{0}((y-x^{0})\cdot\nu)^{+}\right\Vert_{L^{\infty}(B_{r_{n}}(x^{0}))}.
	\end{align*}
	In particular, there exists $r\in(0,R)$ so that 
	\begin{align*}
		\|u_{0}(\cdot)-Q_{0}((\cdot-x^{0})\cdot\nu)^{+}\|_{L^{\infty}(B_{r}(x^{0}))}\leqslant\frac{\varepsilon r}{3}.
	\end{align*}
	Since $u_{0}$ is a continuous function, we have for sufficiently large $n$,
	\begin{align*}
		\|u_{0}(\cdot)-Q(0)((\cdot-x_{n})\cdot\nu)^{+}\|_{L^{\infty}(B_{r}(x_{n}))}\leqslant\frac{2\varepsilon r}{3}
	\end{align*}
	Since $u_{n}$ converges to $u_{0}$ uniformly, we get for sufficiently large $n$,
	\begin{align*}
		\|u_{n}(\cdot)-Q(0)((\cdot-x_{n})\cdot\nu)^{+}\|_{L^{\infty}(B_{r}(x_{n}))}\leqslant\varepsilon r.
	\end{align*}
	This implies that there exists $\bar{\varepsilon}>0$ so that $u_{n}$ is $\bar{\varepsilon}$-flat in $B_{r}(x_{n})$. Moreover, $u_{n}$ is a viscosity solution. Thanks to the regularity results we have $\partial\varOmega^{+}(u_{n})$ is $C^{1,\gamma}$ in a small neighborhood of $x_{n}$ and this implies that $x_{n}\in\partial_{\mathrm{reg}}\varOmega^{+}(u_{n})$, which yields a contradiction with the initial assumption.
	
	\textbf{Step II.} The minimality of $u$ and Proposition \ref{Proposition: Homogeneity of blow-up limits} (2) imply that $u_{0}\colon\mathbb{R}^{d^{*}}\to\mathbb{R}$ is a non-negative global minimizer of $J_{0}$ in $\mathbb{R}^{d^{*}}$. Then it follows from Lemma \ref{Lemma: Graph of u_{00}}, Lemma \ref{Lemma: Graph of u_{00}'} and similar argument in \cite{W1999} that
	\begin{align}\label{Formula: Sing (3)}
		\partial_{\mathrm{sing}}\varOmega^{+}(u_{0})\setminus\{0\}=\varnothing.
	\end{align}
	Let $x_{n}\in\partial_{\mathrm{sing}}\varOmega^{+}(u)$ and let  $r_{n}:=|x_{n}-x^{0}|$. Notice now that for every $n>0$ the point $\xi_{n}:=\frac{x_{n}-x^{0}}{r_{n}}\in\partial B_{1}$ is a singular point for $u_{n}$. Up to extracting a subsequence, we may assume that $\xi_{n}$ converges to a point $\xi_{0}\in\partial B_{1}$. It follows from Step I that $\xi_{0}\in\partial_{\mathrm{sing}}\varOmega^{+}(u_{0})$. This yields a contradiction to \eqref{Formula: Sing (3)}.
\end{proof}
\begin{proposition}\label{Proposition: d>d*}
	Let $u\in\mathcal{K}$ be a minimizer of $J$ in $\Omega\subset\mathbb{R}^{d}$ for $d>d^{*}$, then 
	\begin{align}\label{Formula: Sing (4)}
		\dim_{\mathcal{H}}(\partial_{\mathrm{sing}}\varOmega^{+}(u))\leqslant d-d^{*}.
	\end{align} 
\end{proposition}
\begin{proof}
	Let $d>d^{*}$ and let $s>0$ be fixed. We will prove
	\begin{align}\label{Formula: Sing (5)}
		\mathcal{H}^{d-d^{*}+s}(\partial_{\mathrm{sing}}\varOmega^{+}(u))=0.
	\end{align}
	Suppose that this is not the case and $\mathcal{H}^{d-d^{*}+s}(\partial_{\mathrm{sing}}\varOmega^{+}(u))>0$. Then it follows that there is a point $x^{0}\in\partial_{\mathrm{sing}}\varOmega^{+}(u)$ and a sequence $r_{n}\to0$ such that
	\begin{align*}
		\mathcal{H}^{d-d^{*}+s}(\partial_{\mathrm{sing}}\varOmega^{+}(u)\cap B_{r_{n}}(x^{0}))\geqslant\varepsilon r_{n}^{d-d^{*}+s}.
	\end{align*} 
	Choosing $u_{n}$ a blow-up sequence, we get
	\begin{align*}
		\mathcal{H}^{d-d^{*}+s}(\partial_{\mathrm{sing}}\varOmega^{+}(u_{n})\cap B_{1})\geqslant\varepsilon.
	\end{align*}
	Up to extracting a subsequence, $u_{n}$ converges to a blow-up limit $u_{0}$ and $u_{0}$ is a one-homogeneous global minimizer of $J_{0}$ in $\mathbb{R}^{d}$. Consequently,
	\begin{align*}
		\mathcal{H}^{d-d^{*}+s}(\partial_{\mathrm{sing}}\varOmega^{+}(u_{0})\cap B_{1})\geqslant\varepsilon,
	\end{align*}
	which yields a contradiction to Federer dimension reduction Theorem \cite{F1996} for one-homogeneous global minimizer of $J_{0}$ in $\mathbb{R}^{d}$.
\end{proof}

\end{document}

%% file: mystyle.tex
\usepackage{amsthm,amssymb,amsmath,soul,bropd,float,setspace,enumerate,ragged2e}
\usepackage[mathcal]{eucal}
\usepackage{color,verbatim}
\usepackage{hyperref}
\usepackage{tikz}
\usepackage[top=1in, left=1in, right=1in, bottom=1in, marginpar=2cm]{geometry}

\def\Xint#1{\mathchoice
	{\XXint\displaystyle\textstyle{#1}}%
	{\XXint\textstyle\scriptstyle{#1}}%
	{\XXint\scriptstyle\scriptscriptstyle{#1}}%
	{\XXint\scriptscriptstyle\scriptscriptstyle{#1}}%
	\!\int}
\def\XXint#1#2#3{{\setbox0=\hbox{$#1{#2#3}{\int}$ }
		\vcenter{\hbox{$#2#3$ }}\kern-.6\wd0}}

\def\dashint{\Xint-}
\newtheorem{theorem}{Theorem}[section]
\newtheorem{lemma}[theorem]{Lemma}
\newtheorem{proposition}[theorem]{Proposition}
\newtheorem{corollary}[theorem]{Corollary}

\newtheorem{remark}[theorem]{Remark}

\newtheorem{definition}{Definition}[section]

\def\R{{\mathbb R}}

\def\cQ{{\mathcal Q}}

\def\Lace{\Delta}

\def\0{\varnothing}
\def\1{\left(}
\def\2{\right)}
\def\3{\left[}
\def\4{\right]}
\def\5{\left\{}
\def\6{\right\}}
\def\8{\infty}

\def\cc{\subset\subset}

\newcommand{\mres}{\mathbin{\vrule height 1.6ex depth 0pt width
		0.13ex\vrule height 0.13ex depth 0pt width 1.3ex}}

%% file: MarkIV.bbl
\begin{thebibliography}{60}
	\bibitem{AC1981} H. W. Alt, L. A. Caffarelli, \newblock Existence and regularity for a minimum problem with free boundary, {\it J. Reine Angew. Math.}, {\bf 325}, 105-144, (1981).
	%\bibitem{ACF1984IUMJ1} H. W. Alt, L. A. Caffarelli, A. Friedman, \newblock Jets with two fluids. I. One free boundary, {\it Indiana Univ. Math. J.}, {\bf 33}, 213-247, (1984).
	%\bibitem{ACF1984IUMJ2} H. W. Alt, L. A. Caffarelli, A. Friedman, \newblock Jets with two fluids. II. Two free boundaries, {\it Indiana Univ. Math. J.}, {\bf 33}, 367-391, (1984).
	\bibitem{ACF1984}H. W. Alt, L. A. Caffarelli, A. Friedman, \newblock A free boundary problem for quasilinear elliptic equations, {\it Ann. Scuola Norm. Sup. Pisa Cl. Sci.}, {\bf 11}, 1-44, (1984).
	\bibitem{C1987} L. A. Caffarelli, \newblock A Harnack inequality approach to the regularity of free boundaries. I. Lipschitz free boundaries are $C^{1,\alpha}$, {\it Rev. Mat. Iberoamericana.}, {\bf 3}, 139-162, (1987).
	\bibitem{C1989} L. A. Caffarelli, \newblock A Harnack inequality approach to the regularity of free boundaries. II. Flat free boundaries are Lipschitz, {\it  Comm. Pure Appl. Math.}, {\bf 42}, 55-78, (1989).
	\bibitem{C1988} L. A. Caffarelli,  \newblock A Harnack inequality to the regularity of free boundaries. Part III: existence theory, compactness, and dependence on X, {\it Ann. Scuola Norm. Sup. Pisa}, {\bf 15}(4), 583-602, (1988).
	\bibitem{ACF1984TMS} H. W. Alt, L. A. Caffarelli, A. Friedman, \newblock Variational problems with two phases and their free boundaries, {\it Trans. Amer. Math. Soc.}, {\bf 282}(2), 431-461, (1984).
	\bibitem{CSY2018} L. A. Caffarelli, H. Shahgholian, K. Yeressian, \newblock A minimization problem with free boundary related to a cooperative system, {\it Duke. Math. J.}, {\bf 167}(10), 1825-1882, (2018).
	\bibitem{CIL1992} M. G. Crandall, H. Ishii, P. L. Lions, \newblock User's guide to viscosity solutions of second order partial differential equations, {\it Bull. Amer. Math. Soc.}, {\bf 27}, 1-67, (1992).
	\bibitem{CJK2004} L. A. Caffarelli, D. S. Jerison, C. E. Kenig,   \newblock Global energy minimizers for free boundary problems and full regularity in three diemsnions, {\it Contemp. Math. Amer. Math. Soc. Providence.}, {\bf 350}, 83-97, (2004).
	\bibitem{CS2004} A. Constantin, W. Strauss, \newblock Exact steady periodic water waves with vorticity, {\it Comm. Pure Appl. Math.}, {\bf 57}(4), 481-527, (2004).
	\bibitem{CDW2017} J. F. Cheng, L. L. Du, Y. F. Wang, \newblock Two-dimensional impinging jets in hydrodynamic rotational flows, {\em Ann. Inst. H. Poincar\'{e} Anal. Non Lin\'eaire.}, {\bf 34}, 1355-1386, (2017).
	\bibitem{CDZ2017} J. F. Cheng, L. L. Du, Q. Zhang, \newblock Two-dimensional cavity flow in an infinitely long channel with non-zero vorticity, {\it J. Differential Equations.}, {\bf 263}, 4126-4155, (2017).
	\bibitem{CD2020} J. F. Cheng, L. L. Du, \newblock A free boundary problem for semilinear elliptic equation and its applications, \emph{arXiv:2006.02269vl}.
	\bibitem{DP2005} D. Danielli, A Petrosyan, \newblock A minimum problem with free boundary for a degenerate quasilinear operator, {\it Calc. Var. PDE.}, {\bf 23}(1), 97-124, (2005).
	\bibitem{SJ2009} D. De Silva, D. Jerison, \newblock A singular energy minimizing free boundary, {\it J. Reine Angew. Math.}, {\bf 635}, 1-21, (2009).
	\bibitem{S2011} D. De Silva,  \newblock Free boundary regularity for a problem with right hand side, {\it Interfaces Free Bound.}, {\bf 13}, 223-238, (2011).
	\bibitem{DY2023} L. L. Du, C. L. Yang, \newblock The free boundary of steady axisymmetric inviscid flow with vorticity II: near the non-degenerate points, {\it 	arXiv:2310.09477}.
	\bibitem{E2015} L. C. Evans, \newblock Measure Theory and Fine Properties of Functions, {\it Boca Raton}, CRC Press, 2nd revised,  edition (2015).
	\bibitem{F1996} H. Federer, \newblock Geometric Measure Theory, {\it Springer Berlin, Heidelberg}, (1996).
	\bibitem{GS1996} B. Gustaffsson, H. Shahgholian, Existence and geometric properties of solutions of a free boundary problem in potential theory, {\it J. Reine Angew. Math.}, {\bf 473}, 137-179, (1996).
	\bibitem{GT1998} D. Gilbarg, N. S. Trudinger, \newblock Elliptic Partial Differential Equations of Second Order, {\it Springer-Verlag}, Berlin, (1998).
	\bibitem{JS2015} D. S. Jerison, O. Savin, \newblock Some remarks on stability of cones for the one-phase free boundary problem, {\it Geom. Funct. Anal.}, {\bf 25}, 1240-1257, (2015).
	\bibitem{LT2015} R. Leit\~{a}o, E V. Teixeira, Regularity and geometric estimates for minima of discontinuous functionals, {\it Rev. Mat. Iberoam.}, {\bf 31}, 69-108, (2015).
	\bibitem{M2008} C. B. Morrey, \newblock Multiple Integrals in The Calculus of Variations, {\it  Classics in Mathematics.}, Springer, Berlin, (2008).
	\bibitem{MTV2017} D. Mazzoleni, S, Terracini, B. Velochkov, \newblock Regularity of the optimal sets for some spectral functionals, {\it Geom. Funct. Anal.}, {\bf 27}, 373–426 (2017). 
	\bibitem{MTV2022} D. Mazzoleni, S, Terracini, B. Velochkov, \newblock  Regularity of the free boundary for the vectorial Bernoulli problem, {\it Anal. PDE.}, {\bf 13}, 741-764, (2020).
	\bibitem{R1964} E. R. Reifenberg, \newblock An epiperimetric inequality related to the analytic of minimal surfaces, {\it Ann. of Math.}, {\bf 80}, 1-14, (1964).
	\bibitem{RTV2019} E. Russ, B. Tery, B.Velichkov, \newblock Existence and regularity of optimal shapes for elliptic operators with drift, {\it Calc. Var. PDE.}, {\bf 58}, 199, (2019).
	\bibitem{SV2019} L. Spolaor, B. Velichkov, \newblock An epiperimetric inequality for the regularity of some free boundary problems: the $2$-dimensional case, {\it Comm. Pure. Appl. Math.}, {\bf 72}(2), 375-421, (2019).
	\bibitem{T1967} N. S. Trudinger,  \newblock On Harnack type inequalities and their application to quasilinear elliptic equations, {\it Comm. Pure Appl. Math.}, {\bf 20}, 721-747, (1967).
	\bibitem{V2019} B. Velichkov, \newblock Regularity of the One-phase Free Boundaries, Lect. Notes(2023). https://doi.org/10.1007/ 978- 3- 031- 13238- 4.
	\bibitem{V2008} E.  V\u{a}rv\u{a}ruc\u{a}, \newblock On some properties of the traveling water waves with vorticity, {\it Siam. J. Math. Anal.}, {\bf 39}(5), 1686-1692, (2008). 
	\bibitem{VW2012} E. V\u{a}rv\u{a}ruc\u{a}, G. Weiss,   \newblock The stokes conjecture for waves with vorticity, {\it Ann. Inst. H. Poincar\'{e}.}, {\bf 29}, 861-885, (2012).
	\bibitem{WZ2012} G. Weiss, G. H. Zhang, \newblock A free boundary approach to two-dimensional steady capillary gravity water waves, {\it Arch. Ration. Mech. Anal.}, {\bf 203}, 747-768, (2012).
	\bibitem{W1999} G. Weiss, \newblock Partial regularity for a minimum problem with free boundary, {\it J. Geom. Anal.}, {\bf 9}(2), 317-326, (1999).
\end{thebibliography}
